\documentclass[a4paper,11pt,reqno]{amsart}
          \usepackage{amssymb}
	  \usepackage{amsmath}
          \usepackage{amsfonts}
          \usepackage[english]{babel}
          \usepackage[utf8]{inputenc}

\usepackage[margin=2.5cm]{geometry}

 \usepackage[unicode,colorlinks,plainpages=false,hyperindex=true,bookmarksnumbered=true,bookmarksopen=false,pdfpagelabels]{hyperref}
 \hypersetup{urlcolor=cyan,linkcolor=blue,citecolor=red,colorlinks=true}
\usepackage{color}

\usepackage{enumerate}

\usepackage{tikz} 
\usetikzlibrary{arrows}

\theoremstyle{plain}
\newtheorem{theorem}{Theorem}[section]
\newtheorem{lemma}[theorem]{Lemma}
\newtheorem{proposition}[theorem]{Proposition}
\newtheorem{corollary}[theorem]{Corollary}

\theoremstyle{definition}
\newtheorem{definition}[theorem]{Definition}
\newtheorem{example}[theorem]{Example}

\theoremstyle{remark}
\newtheorem{remark}[theorem]{Remark}

\numberwithin{equation}{section}

\newcommand{\N}{\ensuremath{\mathbb{N}}}

\newcommand{\R}{\ensuremath{\mathbb{R}}}
\newcommand{\Z}{\ensuremath{\mathbb{Z}}}
\newcommand{\kk}{\ensuremath{\Bbbk}}

\newcommand{\mc}[1]{\mathcal{#1}}

\newcommand{\mb}[1]{\mathbb{#1}}

\newcommand{\tint}{{\textstyle\int}}

\newcommand{\tr}{\operatorname{tr}}

\newcommand{\Hom}{\operatorname{Hom}}
\newcommand{\Aut}{\operatorname{Aut}}
\newcommand{\Der}{\operatorname{Der}}
\newcommand{\Mat}{\operatorname{Mat}}
\newcommand{\Id}{\operatorname{Id}}

\newcommand{\Rep}{\operatorname{Rep}}

\newcommand{\PP}{\operatorname{P}}
\DeclareMathOperator{\mres}{mRes}
\DeclareMathOperator{\mult}{m}
\DeclareMathOperator{\Vect}{Vect}
\DeclareMathOperator{\nonloc}{NonLoc}
\DeclareMathOperator{\loc}{Loc}
\DeclareMathOperator{\rat}{Rat}

\newcommand{\VV}{\ensuremath{\mathcal{V}}}
\newcommand{\del}{\ensuremath{\partial}}
\newcommand\br[1]{\{ #1 \}}
\newcommand\dgal[1]{  \left\{\!\!\left\{#1\right\}\!\!\right\} }
\newcommand{\ldb}{\{\!\!\{}
\newcommand{\rdb}{\}\!\!\}}

\def\restriction#1#2{\mathchoice
              {\setbox1\hbox{${\displaystyle #1}_{\scriptstyle #2}$}
              \restrictionaux{#1}{#2}}
              {\setbox1\hbox{${\textstyle #1}_{\scriptstyle #2}$}
              \restrictionaux{#1}{#2}}
              {\setbox1\hbox{${\scriptstyle #1}_{\scriptscriptstyle #2}$}
              \restrictionaux{#1}{#2}}
              {\setbox1\hbox{${\scriptscriptstyle #1}_{\scriptscriptstyle #2}$}
              \restrictionaux{#1}{#2}}}
\def\restrictionaux#1#2{{#1\,\smash{\vrule height .8\ht1 depth .85\dp1}}_{\,#2}} 

\definecolor{light}{gray}{.9}

\begin{document}

\title{Double Multiplicative Poisson Vertex Algebras}

\author{Maxime Fairon}
 \address[Maxime Fairon]{School of Mathematics and Statistics\\ University of Glasgow, University Place\\ Glasgow G12 8QQ, UK}
 \email{Maxime.Fairon@glasgow.ac.uk}
 \address[current address]{Department of Mathematical Sciences, Schofield Building, Loughborough University, Epinal Way, Loughborough LE11 3TU, UK}
 \email{M.Fairon@lboro.ac.uk} 

\author{Daniele Valeri}
 \address[Daniele Valeri]{Dipartimento di Matematica \& INFN, Sapienza Universit\`a di Roma,
 P.le Aldo Moro 5, 00185 Roma, Italy}
 \email{daniele.valeri@uniroma1.it}


\begin{abstract}
We develop the theory of  double multiplicative Poisson vertex algebras. 
These structures, defined at the level of associative algebras, are shown to be such that they induce a classical structure of multiplicative Poisson vertex algebra on the corresponding representation spaces.  
Moreover, we prove that they are in one-to-one correspondence with local lattice double Poisson algebras, a new important class among Van den Bergh's double Poisson algebras. 
We derive several classification results, and we exhibit their relation to non-abelian integrable differential-difference equations. 
A rigorous definition of  double multiplicative Poisson vertex algebras in the non-local and rational cases is also provided.
\end{abstract}

\maketitle

 \setcounter{tocdepth}{1}

\tableofcontents

\section{Introduction}

Given a unital associative algebra $\VV$, Van den Bergh \cite{VdB1} introduced the structure of a double bracket on $\VV$ as a map 
$$
    \dgal{-,-}:\VV\times \VV \to \VV\otimes \VV\,, \quad (a,b) \mapsto \dgal{a,b}\,,
$$
which is linear in both arguments and which enjoys properties of derivation and skewsymmetry, see Subsection \ref{ss:Remind} for the definition. The importance of double brackets can then be realised through representation theory as follows. Denoting by $\kk$ the base field of $\VV$, we can form the representation space $\Rep(\VV,N)$ parametrised by representations of $\VV$ over $\kk^N$, $N\geq1$. The coordinate ring of this affine scheme is generated by the functions $a_{ij}$, where $(a_{ij})_{1\leq i,j\leq N}$ is the matrix-valued function on $\Rep(\VV,N)$ corresponding to $a\in \VV$.  Then, it was observed by Van den Bergh that the operation $\br{-,-}$ on $\Rep(\VV,N)$ satisfying 
\begin{equation} \label{Eq:dbrRep}
    \br{a_{ij},b_{kl}}=\dgal{a,b}'_{kj} \dgal{a,b}''_{il}\,, \quad 
    a,b\in \VV,\,\, 1\leq i,j,k,l \leq N\,,
\end{equation}
(here, we use a strong version of Sweedler's notation : $d'\otimes d'':=d\in \VV\otimes \VV$) defines a unique skewsymmetric biderivation over $\Rep(\VV,N)$. This gives an example of application of the Kontsevich-Rosenberg principle \cite{Ko,KR}, which states that the non-commutative version $\PP_{\mathrm{nc}}$ of a property $\PP$ defined over commutative algebras should give back $\PP$ when we go from an associative algebra $\VV$ to (the coordinate ring of) its representation spaces $\Rep(\VV,N)$. Furthermore, it was possible to generalise this construction using \eqref{Eq:dbrRep} to introduce non-commutative versions of Lie algebras \cite{Sch,ORS,DSKV}, (quasi-)Poisson algebras \cite{VdB1}, Lie-Rinehart algebras \cite{VdB2}, or Poisson vertex algebras \cite{DSKV} and Courant-Dorfman algebras \cite{FH}. This paper is devoted to pursue the Kontsevich-Rosenberg principle even further using \eqref{Eq:dbrRep} by bringing to light the non-commutative version of multiplicative Poisson vertex algebras, which have recently been introduced by De Sole, Kac, Valeri and Wakimoto \cite{DSKVW1,DSKVWclass}. 

\medskip 

To  understand the objects at stake, recall that a Poisson algebra is a commutative algebra $V$ endowed with a Poisson bracket $\br{-,-}$. In other words, $V$ is equipped with a Lie bracket that is compatible with the commutative product on $V$, see Definition \ref{Def:PA}. 
Let us assume that $V$ admits an infinite order automorphism $S\in \Aut(V)$ such that it commutes with the Poisson bracket, i.e. 
\begin{equation} \label{Eq:I-Scom}
    S\circ \br{-,-} = \br{-,-}\circ (S\times S)\,,
\end{equation}
and such that, for $a,b\in V$, we have $\{S^n(a),b\}=0$, for all but finitely many $n\in\mb Z$. In this case, we call $V$
a local lattice Poisson algebra.
Then, it was observed in \cite{DSKVW1} that the Poisson bracket on $V$ can be equivalently understood in terms of a bilinear operation 
\begin{equation} \label{Eq:I-brlambda}
    \br{-_\lambda -} : V\times V \to V[\lambda^{\pm 1}]\,,
\end{equation}
(here $\lambda$ should be seen as a formal parameter) defined for any $a,b\in V$ by 
\begin{equation} \label{Eq:I-brFormula}
    \br{a_\lambda b}=\sum_{n\in \Z} \lambda^n \br{S^n(a),b}\,.
\end{equation}
Note that only finitely many terms are non-zero in the right-hand side of \eqref{Eq:I-brFormula}. We refer to Proposition \ref{Prop:Corresp} for a precise statement. Due to the defining properties of the Poisson bracket and the compatibility with $S$ given by \eqref{Eq:I-Scom}, the map \eqref{Eq:I-brlambda} inherits several useful properties: it is skewsymmetric \eqref{Eq:MA2}, sesquilinear \eqref{Eq:MA1}, and it satisfies Leibniz rules \eqref{Eq:ML}--\eqref{Eq:MR} as well as an analogue of Jacobi identity \eqref{Eq:MA3}. In full generalities, an operation on $V$ of the form \eqref{Eq:I-brlambda} satisfying these assumptions is called a multiplicative Poisson vertex algebra, see Definition \ref{Def:MPVA}. 
Hence, this proves a correspondence, that can be depicted as follows:
\begin{equation}\label{intro1}
\{\text{local lattice Poisson algebras}\}
\stackrel{1-1}{\longleftrightarrow}
\{\text{multiplicative Poisson vertex algebras}\}
\end{equation}
As mentioned earlier, there is a non-commutative version of Poisson algebras due to Van den Bergh \cite{VdB1}. Our definition of a non-commutative version of multiplicative Poisson vertex algebras is motivated by the correspondence \eqref{intro1}. Namely, if we replace Poisson algebras by Van den Bergh's double Poisson algebras in \eqref{intro1}, we end up with the notion of a double multiplicative Poisson vertex algebra. The main ingredient used in the definition
is then a double multiplicative $\lambda$-bracket, which is a suitable ``double" generalization of 
\eqref{Eq:I-brFormula}. This is a bilinear map 
\begin{equation*} 
    \dgal{-_\lambda -} : \VV\times \VV \to (\VV\otimes\VV)[\lambda^{\pm 1}]\,,
\end{equation*}
defined over an associative algebra $\VV$ endowed with an automorphism $S$, 
see Definition \ref{def:DMPVA}.
Hence, in Proposition \ref{prop:20210506}, we get the correspondence
$$
\{\text{local lattice double Poisson algebras}\}
\stackrel{1-1}{\longleftrightarrow}
\{\text{double multiplicative Poisson vertex algebras}\}
$$
which is the ``double" analogue of \eqref{intro1}.
This analogy is precisely stated in Corollary \ref{Cor:CommuteRep}, where we show that the non-commutative correspondence (Proposition \ref{prop:20210506}) implies the commutative correspondence (Proposition \ref{Prop:Corresp}) when going to representation spaces, i.e. when going from $\VV$ to $V=\kk[\Rep(\VV,N)]$. That result crucially depends on Van den Bergh's work \cite{VdB1}, and the fact that a double multiplicative Poisson vertex algebra induces on representation spaces a structure of multiplicative Poisson vertex algebra through a mapping of the form \eqref{Eq:dbrRep}. 

The relation that we have just outlined to the commutative theory developed in \cite{DSKVW1,DSKVWclass}  is crucial for applications. Indeed, it is shown in \cite{DSKVW1,DSKVWclass} how multiplicative Poisson vertex algebras are useful to understand the structure of  differential-difference equations. In our case, the non-commutative version of this theory allows us to understand the structure of non-abelian (i.e. matrix-valued) differential-difference equations. We are able to construct several families of integrable hierarchies of differential-difference equations in that way, see e.g. \S\ref{ss:IntR2}.  
This important application to integrable systems motivates us in the same way to investigate the double multiplicative Poisson vertex algebra structures on algebras of non-commutative difference functions extending the algebra of non-commutative difference polynomials in $\ell\geq 1$ variables $u_i:=u_{i,0}$ 
$$
\mc R_\ell=\kk\langle u_{i,n}\mid 1\leq i\leq \ell,n\in\mb Z\rangle\,,\quad 
S(u_{i,n})=u_{i,n+1}\,.
$$
We perform a classification of double multiplicative Poisson vertex algebra structures on $\mc R_1$ and $\mc R_2$, see Proposition \ref{Pr:class1} and Theorem \ref{Thm:2-explicit} respectively.

\subsection*{Layout of the paper} In Section \ref{S:Prelim}, we review the correspondence between Poisson algebras and multiplicative Poisson vertex algebras. We also introduce some operations induced on tensor products of an algebra. In Section \ref{S:dmPVA}, we state the key definition of a double multiplicative Poisson vertex algebra, before deriving several properties and examples. Next, we provide  classification results for  double multiplicative Poisson vertex algebras as part of Section \ref{Sec:Classif}. 
Then, we explain in Section \ref{S:Represent} how  a double multiplicative Poisson vertex algebra structure on an algebra $\VV$ induces a multiplicative Poisson vertex algebra structure on the associated (commutative) algebra $\VV_N=\kk[\Rep(\VV,N)]$. 
In Section \ref{S:Integr}, we apply our theory to the study of differential-difference equations. 
Finally, in Section \ref{S:nonloc} we outline how to modify double multiplicative Poisson vertex algebras in the non-local or rational cases, and we provide several examples. 

\subsection*{Relation to the work of Casati-Wang} While we were working on this project, we became aware that a parallel investigation on double multiplicative Poisson vertex algebras was carried out independently by Casati and Wang \cite{CW2}. For the reader's convenience, let us outline the main differences between these two works. 
Firstly, Casati and Wang \cite{CW2} introduced double multiplicative Poisson vertex algebras on
the space of non-commutative Laurent polynomials with an infinite order automorphism, while we work in the more general
setup of algebras of non-commutative difference functions and we provide classification results in
Section \ref{Sec:Classif}. Their motivation stems from the integrability of non-abelian difference equations, which we also consider in Section \ref{S:Integr}, though we do not study the several non-local examples gathered in \cite[Sect. 6]{CW2}. 
Secondly, they compare the formalism of double multiplicative $\lambda$-brackets to the (difference version of the)  $\theta$-formalism of Olver and Sokolov \cite{OS}. In particular, their comparison uses a graded version of double multiplicative $\lambda$-brackets, which we do not consider. 
In the present work, we exclusively use the formalism of double multiplicative Poisson vertex algebras for computations, and we present a deeper study of their algebraic structure. For example, we give a correspondence with lattice double Poisson algebras in \S\ref{ss:ReldmPVA-dPA}, and we explain in Section \ref{S:Represent} that double multiplicative Poisson vertex algebras induce \emph{usual} multiplicative Poisson vertex algebra structures (cf. \cite{DSKVW1,DSKVWclass}) on their representation spaces, in agreement with the Kontsevich-Rosenberg principle. 

\medskip 
 
\subsection*{Acknowledgments.}
We wish to thank Sylvain Carpentier for useful and stimulating discussions.
We are also grateful to Matteo Casati for sharing with us a preliminary version of the work \cite{CW2} and for enlightening discussions on this topic. 
M.F. is supported by a Rankin-Sneddon Research Fellowship of the University of Glasgow.
D.V. acknowledges the financial support of the
project MMNLP (Mathematical Methods in Non Linear Physics) of the
INFN.


\section{Preliminaries}
\label{S:Prelim}

Throughout the paper $\kk$ denotes a field of characteristic zero (assumed to be algebraically closed for computations), and unadorned tensor products are taken over $\kk$.

\subsection{Commutative Poisson structures}\label{sec:2.1}

In this subsection, we follow \cite{DSKVW1}. All algebras are unital commutative algebras over $\kk$.

\begin{definition}  \label{Def:PA}
 A \emph{Poisson algebra} is an algebra $V$ endowed with a Poisson bracket, i.e. a linear map 
 $$\br{-,-}:V\otimes V\to V\,, \quad a\otimes b \mapsto \br{a,b}\,,$$
which is skewsymmetric, i.e. $\{a,b\}=-\{b,a\}$, satisfies the Left and right Leibniz rules
$$
\{a,bc\}=\{a,b\}c+b\{a,c\}
\,,
\quad
\{ab,c\}=\{a,c\}b+a\{b,c\}\,,
$$
and the Jacobi identity
$$\br{a,\br{b,c}}-\br{b,\br{a,c}}-\br{\br{a,b},c}=0\,,$$
for all $a,b,c\in V$.

 A \emph{lattice Poisson algebra} is a Poisson algebra $V$ with an infinite order automorphism $S\in \Aut(V)$, namely for all $a,b\in V$, 
$$S(ab)=S(a)S(b)\quad \text{ and } S(\br{a,b})=\br{S(a),S(b)}\,.$$ 
It is called local if, for every $a,b\in V$, $\br{S^n(a),b}=0$ for all but finitely many $n\in \Z$. 
\end{definition}

\begin{definition}  \label{Def:MPVA} 
Let $V$ be an algebra endowed with an automorphism $S\in \Aut(V)$. 

 A \emph{multiplicative $\lambda$-bracket} on $V$ is a linear map 
$$\br{-_\lambda-}:V\otimes V\to V[\lambda^{\pm1}],\quad a\otimes b\mapsto \br{a_\lambda b}\,,$$
such that for any $a,b,c\in V$, 
\begin{subequations}
  \begin{align}
    & \quad \br{S(a)_\lambda b}=\lambda^{-1} \br{a_\lambda b}\,, \quad 
\br{a_\lambda S(b)}=\lambda S(\br{a_\lambda b})\,, \quad & \text{(sesquilinearity)} \label{Eq:MA1} \\
& \quad \br{a_\lambda bc}=\br{a_\lambda b}c+b \br{a_\lambda c}\,, 
\quad & \text{(left Leibniz rule)} \label{Eq:ML} \\
& \quad \br{ab_\lambda c}=\br{a_{\lambda x}c} \restriction{\Big(}{x=S} b\Big) 
+\restriction{\Big(}{x=S} a\Big) \br{b_{\lambda x}c}\,.  
\quad & \text{(right Leibniz rule)} \label{Eq:MR}
  \end{align}
\end{subequations}

We say that $V$ is a \emph{multiplicative Poisson vertex algebra} if it admits a multiplicative $\lambda$-bracket $\br{-_\lambda -}$ satisfying 
\begin{subequations}
 \begin{align}
& \quad \br{a_\lambda b}=-\big|_{x=S} \br{b_{\lambda^{-1} x^{-1}} a}\,,  \quad & \text{(skewsymmetry)} \label{Eq:MA2} \\
& \quad \br{a_\lambda\br{b_\mu c}} 
-\br{b_\mu\br{a_\lambda c}}
-\br{\br{a_\lambda b}_{\lambda\mu}c}=0\,.  \quad & \text{(Jacobi identity)} \label{Eq:MA3} 
 \end{align}
\end{subequations}
\end{definition}
In the above formulas, given an element $a(\lambda)=\sum_{k\in \Z} a_k \lambda^k \in V[\lambda^{\pm1}]$,
we use the notation
$$
a(\lambda x)\restriction{\Big(}{x=S}b\Big)=
\sum_{k\in \Z} a_k S^k(b)\lambda^k
\,.
$$
Furthermore, let us set $\mres_\lambda a(\lambda)=a_0$. 
The next result can be found in \cite[\S3.1]{DSKVW1}. 
\begin{proposition} \label{Prop:Corresp}
If $V$ is a multiplicative Poisson vertex algebra with multiplicative $\lambda$-bracket
$\br{-_\lambda-}$ and automorphism $S\in\Aut(V)$, then $V$ is a local
lattice Poisson algebra
with the Poisson bracket 
\begin{equation}\label{eq:Corr-1}
\br{a,b}=\mres_\lambda\br{a_\lambda b}
\,,\quad a,b\in\ V \,.
\end{equation}
Conversely,
if $V$ is a local lattice Poisson algebra 
with Poisson bracket $\br{-,-}$ and 
automorphism $S\in\Aut(V)$,
then we can endow it with a structure of multiplicative Poisson vertex algebra
with the multiplicative $\lambda$-bracket
\begin{equation} \label{eq:Corr-2}
\br{a_\lambda b}:=\sum_{n \in \Z}\lambda^n \br{S^n(a),b}\,,\quad a,b\in V\,.
\end{equation}
\end{proposition}

\begin{remark} \label{Rem:mLCA}
If $V$ is simply a vector space with an invertible endomorphism $S$, 
we can define the notion of a \emph{local lattice Lie algebra} from Definition \ref{Def:PA} by forgetting the derivation rules. 
Similarly, a \emph{multiplicative Lie conformal algebra} is obtained from Definition \ref{Def:MPVA} by omitting the Leibniz rules \eqref{Eq:ML}--\eqref{Eq:MR}. 
Then, Proposition \ref{Prop:Corresp} can be weakened to an equivalence between these two structures, and this result originally appeared in \cite{GKK98}. 
\end{remark}

\subsection{Operations on an algebra}

Let $\VV$ be a unital associative algebra over a field $\kk$ of characteristic $0$. 
\subsubsection{Basic operations}
We introduce several notations following \cite{VdB1,DSKV}. 
Given $n \geq 2$, we can form the tensor product $\VV^{\otimes n}$, which we see as an associative algebra for 
$$
    (a_1\otimes \ldots\otimes a_n)(b_1\otimes \ldots \otimes b_n)
    =a_1b_1\otimes \ldots \otimes a_nb_n\,.
$$

When $n=2$, we use a strong version of Sweedler's notation 
$$
    A= \sum_l A_l'\otimes A_l'' =:A'\otimes A'' \,\in \VV\otimes \VV
$$
to denote elements. 
The permutation endomorphism $(-)^\sigma$ on $\VV\otimes \VV$ is given by 
\begin{equation}\label{eq:sigma2}
    (a\otimes b)^\sigma=b\otimes a\,.
\end{equation}
In full generalities for $n\geq 2$, 
we introduce the permutation 
\begin{equation}\label{eq:sigma}
  \VV^{\otimes n}\to \VV^{\otimes n}\,: \, 
a_1 \otimes \ldots \otimes a_n \mapsto (a_1 \otimes \ldots \otimes a_n)^\sigma:=a_n \otimes a_1 \otimes \ldots \otimes a_{n-1}\,. 
\end{equation}

We introduce the outer and inner bimodule structures on $\VV\otimes \VV$ by 
\begin{equation}\label{eq:bimodule}
    a\,A \,b=aA'\otimes A''b\,, \quad 
    a\ast A\ast b=A'b\otimes aA''\,, \quad 
    a,b\in \VV,\, A\in \VV\otimes\VV\,.
\end{equation}
We define left and right $\VV$-module structures on $\VV^{\otimes n}$ as follows. For $0\leq i \leq n-1$, 
\begin{align*}
b \ast_i (a_1 \otimes \ldots \otimes a_n)=& a_1 \otimes \ldots \otimes a_i \otimes b a_{i+1} \otimes a_{i+2} \otimes \ldots \otimes a_n \,, \\
(a_1 \otimes \ldots \otimes a_n) \ast_i b=& a_1 \otimes \ldots \otimes a_{n-i-1} \otimes  a_{n-i}b \otimes a_{n-i+1} \otimes \ldots \otimes a_n\,.
\end{align*}
For $i=0$, these are just the multiplication on the left of the left-most component, and on the right of the right-most component. In that case, we omit to write $\ast_0$, so that $b A=b \ast_0 A$ and $A b = A \ast_0 b$ for any $A\in \VV^{\otimes n}$.  
We set $\ast_{i+n}=\ast_i$ to define the operation for any $i \in \Z$.

Next, we introduce tensor product rules as maps $\VV\otimes \VV^{\otimes n}\to \VV^{\otimes (n+1)}$. For $0\leq i \leq n-1$, 
\begin{align*}
    b \otimes_i (a_1 \otimes \ldots \otimes a_n)=& a_1 \otimes \ldots a_i \otimes b  \otimes a_{i+1} \otimes \ldots \otimes a_n \,, \\
(a_1 \otimes \ldots \otimes a_n) \otimes_i b=& a_1 \otimes \ldots\otimes  a_{n-i} \otimes b \otimes a_{n-i+1} \otimes \ldots \otimes a_n\,,
\end{align*}
We omit the subscript for $i=0$ from now on.

\medskip 

Finally, there is an associative product $\bullet$ on $\VV\otimes \VV$ defined by 
\begin{equation}\label{eq:bullet}
    A\bullet B = A'B' \otimes B'' A''\,.
\end{equation}
The permutation $\sigma$ defined in \eqref{eq:sigma2} is an antihomomorphism for this product:
\begin{equation}\label{eq:bullet-anti}
(A\bullet B)^{\sigma}=B^{\sigma}\bullet A^{\sigma}
\,.
\end{equation}
From \eqref{eq:bimodule} we get that the inner and outer bimodule structures of $\mc V\otimes\mc V$
are related to the associative product in \eqref{eq:bullet} as follows
$$
A' B A''=A\bullet B=B''\ast A\ast B'\,.
$$

We can  define three possible left and right module structures for $(\VV^{\otimes 2},\bullet)$ on $\VV^{\otimes 3}$,
denoted by $\bullet_i$, $i=1,2,3$, as follows 
\begin{equation}\label{20210624:eq1}
    \begin{aligned}
&A \bullet_1 (x\otimes y \otimes z)=x\otimes A' y \otimes z A'',\quad
 (x\otimes y \otimes z)\bullet_1 A= x\otimes y A' \otimes A'' z, \\
&A \bullet_2 (x\otimes y \otimes z)=A' x\otimes y \otimes z A'',\quad
 (x\otimes y \otimes z)\bullet_2 A= xA'\otimes y \otimes A'' z, \\
&A \bullet_3 (x\otimes y \otimes z)=A' x\otimes yA'' \otimes z,\quad
 (x\otimes y \otimes z)\bullet_3 A= xA'\otimes A''y \otimes z,
    \end{aligned}
\end{equation}
The following result appeared in \cite{DSKV}.
\begin{lemma}\phantomsection\label{lemma:bullet-i}
\begin{enumerate}[(a)]
\item
The $\bullet_i$ left (and right) actions of $V^{\otimes2}$ on $V^{\otimes3}$ are 
indeed actions, i.e. they are associative
with respect to the $\bullet$-product of $V^{\otimes2}$:
$$
A\bullet_i(B\bullet_i X)=(A\bullet B)\bullet_i X
\,\text{ and }\,
(X\bullet_i A)\bullet_i B=X\bullet_i(A\bullet B)
\,,
$$
for every $A,B\in V^{\otimes2}$ and $X\in V^{\otimes3}$.
\item
The left $\bullet_i$ and the right $\bullet_j$ actions commute
for every $i,j=1,2,3$ such that $|i-j|\neq2$\footnote{There is a misprint in Lemma \ref{lemma:bullet-i}(b) in \cite{DSKV} and the further assumption $|i-j|\neq2$ is omitted. 
}:
$$
A\bullet_i(X\bullet_j B)=(A\bullet_i X)\bullet_j B
$$
for every $A,B\in V^{\otimes2}$ and $X\in V^{\otimes3}$.
\item
The $\bullet_1$ and $\bullet_3$ left (resp. right) actions of $V^{\otimes2}$ on $V^{\otimes3}$ 
commute:
$$
A\bullet_1(B\bullet_3 X)=
B\bullet_3(A\bullet_1 X)
\,\text{ and }\,
(X\bullet_1 A)\bullet_3 B=(X\bullet_3 B)\bullet_1 A
\,,
$$
for every $A,B\in V^{\otimes2}$ and $X\in V^{\otimes3}$.
(In general, the $\bullet_i$ and $\bullet_j$ left (resp. right) actions do NOT commute
if $|i-j|=1$.)
\end{enumerate}
\end{lemma}

\subsubsection{Extending derivations and homomorphisms}

Consider a derivation $\del\in \Der(\VV,M)$ where $M$ is a $\VV$-bimodule. 
We can extend $\del$  to $\VV^{\otimes m}$ by acting only on one copy of $\VV$ as 
\begin{equation} \label{Eq:DerExt}
\del_{(i)}:\VV^{\otimes m}\to  \VV^{\otimes (i-1)}\otimes M \otimes \VV^{\otimes (m-i)}\,, \quad 
    \del_{(i)}(a_1\otimes\dots\otimes a_m)
=
a_1\otimes\dots\otimes \del(a_i)\otimes\dots\otimes a_m
\,,
\end{equation}
for any $1\leq i \leq m$. 
In particular, we denote the induced derivations $\del_{(1)},\del_{(m)}$ on the leftmost and rightmost factors by $\del_L,\del_R$ respectively, i.e. 
\begin{equation} \label{Eq:DelLRExtend}
\begin{aligned}
\del_L :\VV^{\otimes m}\to M\otimes \VV^{\otimes (m-1)}\,, &\quad 
\del_R :\VV^{\otimes n}\to \VV^{\otimes (m-1)}\otimes M\,, \\
  \del_L(a_1 \otimes \ldots \otimes a_m)=\del(a_1) \otimes \ldots \otimes a_m\,, 
&\quad 
\del_R(a_1 \otimes \ldots \otimes a_m)=a_1 \otimes \ldots \otimes \del(a_m)\,.
\end{aligned}
\end{equation}
We also extend $\del$ to $\VV^{\otimes m}$ by 
\begin{equation} \label{Eq:DelExtend}
\begin{aligned}
&\del:=\sum_{i=1}^m \del_{(i)} :\VV^{\otimes m} \longrightarrow \oplus_{i=1}^m
\Big(\VV^{\otimes (i-1)}\otimes M \otimes \VV^{\otimes (m-i)}\Big)\,, \\
 &\del(a_1 \otimes \ldots \otimes a_m)=\sum_{i=1}^n a_1 \otimes \ldots \otimes a_{i-1} \otimes \del(a_i) \otimes a_{i+1} \otimes \ldots \otimes a_m\,.
  \end{aligned}
\end{equation}
When $M=\mc V^{\otimes n}$, we call $\partial$ and $n$-fold derivation. For a $2$-fold derivation $\del \in \Der(\VV,\VV^{\otimes 2})$ which satisfies by definition 
$$
    \del(ab)=a\del(b)'\otimes \del(b)'' + \del(a)'\otimes \del(a)''b\,,
$$
we get for example that 
\begin{equation}\label{eq:DLR}
\begin{aligned}
&\del(a\otimes b)=\del_L(a\otimes b)+\del_R(a\otimes b)\,, \quad \text{ where }\\
&\del_L(a\otimes b)=\del(a)'\otimes \del(a)''\otimes b \,, \quad 
\del_R(a\otimes b)=a \otimes \del(b)'\otimes \del(b)''\,.
\end{aligned}
\end{equation}
We will use the natural bimodule structure on $\Der(\mc V,\mc V^{\otimes2})$ induced by the inner bimodule structure of $\mc V\otimes\mc V$:
$$
(a\partial b)(f)=a\ast \partial(f)\ast b
=\partial(f)'b\otimes a\partial(f)''\,,
\quad a,b,f\in\mc V\,,\partial\in\Der(\mc V,\mc V^{\otimes2})\,.
$$
Given a collection of $2$-fold derivations $\partial_1,\dots,\partial_\ell$ in $\Der(\mc V,\mc V^{\otimes2})$,
we say that they are \emph{linearly independent} if the identity ($a_i,b_i\in\mc V$)
\begin{equation}\label{linear-ind}
\sum_{i=1}^\ell a_i\partial_ib_i=0\,,    
\end{equation}
implies $a_i=0$ or $b_i=0$, for every $i=1,\dots,\ell$. If $\ell$ is infinite, we require that
this condition is satisfied for any finite subset of $\{\partial_i\}_{i=1}^\ell$\,.

Assume that $S \in \Hom(\VV)$ is an algebra homomorphism. Then we can extend $S$ to $\VV^{\otimes m}$ by 
\begin{equation}\label{20210616:eq1}
  S(a_1 \otimes \ldots \otimes a_m)=S(a_1) \otimes \ldots \otimes S( a_m)\,.
\end{equation}
In particular, if $S\in \Aut(\VV)$, this defines an automorphism of $\Aut(\VV^{\otimes m})$. As in the case of derivations, we can adapt the construction to act on one copy of $\VV$ only, or to consider an arbitrary algebra homomorphism $\VV_1\to \VV_2$.

\subsubsection{Representation algebra} \label{sss:RepAlg}

Let $N \in \N^\times$. We define $\VV_N$ as the commutative algebra generated by symbols  $a_{ij}$ for $a\in \VV$ and $1\leq i,j \leq N$, which are subject to the relations 
\begin{equation*}
  1_{ij}=\delta_{ij}\,, \quad (ab)_{ij}=\sum_{1\leq k \leq N}a_{ik}b_{kj}\,, \quad 
(\alpha a+\beta b)_{ij}=\alpha a_{ij}+\beta b_{ij}\,,
\end{equation*}
for any $a,b \in \VV$, $\alpha,\beta \in \kk$, $1\leq i,j \leq N$. We call $\VV_N$ the \emph{$N$-th representation algebra} of $\VV$. Clearly, $\VV_N$ is finitely generated if $\VV$ has this property. Recall that $\VV_N$ is the coordinate ring of the \emph{representation scheme} $\Rep(\VV,N)$ parametrised by representations of $\VV$ on $\kk^N$.  
If $\del \in \Der(\VV)$, it induces a derivation of $\VV_N$ from its definition on generators as $\del(a_{ij})=(\del(a))_{ij}$. 
If $S\in \Aut(\VV)$, it induces an automorphism of $\VV_N$ as $S(a_{ij})=(S(a))_{ij}$.


\section{Double multiplicative Poisson vertex algebras}
\label{S:dmPVA}

\subsection{Review on double Poisson algebras} 
\label{ss:Remind}
In this subsection we let $\mc V$ be a unital associative algebra over $\kk$. We review the notion of double bracket and double Poisson algebra introduced in \cite{VdB1}. Example
\ref{Ex:Powell} is taken from \cite{P16}, while Example \ref{Exmp:Symp} is a special case of \cite[\S6.3]{VdB1} for a one-loop quiver.

\begin{definition} \label{def:DBr} 
  A \emph{double bracket}  (or \emph{2-fold bracket}) on $\mc V$ is a linear map 
\begin{equation*}
\dgal{-,-} :  \mc V \otimes \mc V \to \mc V\otimes \mc V\,, \quad a \otimes b \mapsto \dgal{a, b}\,,
\end{equation*}
such that for all $a,b,c\in \mc V$
\begin{subequations}
  \begin{align}
& \quad \dgal{a,b}=- \dgal{b,a}^\sigma\,, \quad & \text{(cyclic skewsymmetry)} \label{Eq:DA} \\
& \quad \dgal{a, bc}=\dgal{a, b}c+b \dgal{a, c}\,, 
\quad & \text{(left Leibniz rule)} \label{Eq:Dl} \\
& \quad \dgal{ab, c}=\dgal{a,c}\ast_1 b 
+a \ast_1 \dgal{b,c}\,.  
\quad & \text{(right Leibniz rule)} \label{Eq:Dr}
  \end{align}
\end{subequations}
\end{definition}
Given a double bracket, we introduce the maps 
\begin{align*}
&    \dgal{a,b' \otimes b''}_L=\dgal{a,b'}\otimes b''\,, \quad 
&&\dgal{a,b' \otimes b''}_R=b' \otimes \dgal{a,b''}\,, \\
&\dgal{a' \otimes a'',b}_L=\dgal{a',b}\otimes_1 a''\,, \quad 
&&\dgal{a' \otimes a'',b}_R=a' \otimes_1\dgal{a'',b}\,.
\end{align*}

\begin{definition} \label{def:DPA}  
  A \emph{double Poisson algebra} is an algebra $\mc V$ endowed with a double bracket such that for all $a,b,c\in \mc V$
\begin{equation}
    \quad \dgal{a, \dgal{b,c}}_L-\dgal{b,\dgal{a,c}}_R-\dgal{\dgal{a,b},c}_L=0
\,. \quad \text{(Jacobi identity)} \label{Eq:DJ} 
\end{equation}
In that case, we say that $\dgal{-,-}$ is a \emph{double Poisson bracket}.
\end{definition}
\begin{remark}
In Definition \ref{def:DPA}, we chose the condition \eqref{Eq:DJ} which is given in \cite{DSKV} and is equivalent to the original condition of Van den Bergh \cite{VdB1} : 
\begin{equation*}
    \quad \dgal{a, \dgal{b,c}}_L+(\dgal{b,\dgal{c,a}}_L)^\sigma
+(\dgal{c,\dgal{a,b}}_L)^{\sigma^2}=0
\,.
\end{equation*}
\end{remark}

\begin{example}\label{Ex:Powell} 
Let $\mc V=\kk[u]$. It is shown in \cite{P16,VdB1} that a double bracket $\dgal{-,-}$ on $\mc V$ is a double Poisson bracket if and only if it satisfies  
\begin{equation*}
  \dgal{u,u}=\alpha(u \otimes 1 - 1 \otimes u) + \beta (u^2 \otimes 1 - 1 \otimes u^2) 
+ \gamma (u^2 \otimes u - u \otimes u^2)\,,
\end{equation*}
where $\alpha,\beta,\gamma\in \kk$ satisfy $\beta^2=\alpha \gamma$.   
\end{example}
\begin{example}\label{Exmp:Symp} 
  Let $\mc V=\kk\langle u,v \rangle$. Then $\dgal{u,u}=0=\dgal{v,v}$, $\dgal{v,u}=1 \otimes 1$  defines a double Poisson bracket.
\end{example}

Double Poisson brackets can be seen as a non-commutative version of Poisson brackets due to the next result, where we use the notations from \S\ref{sss:RepAlg}. 

\begin{theorem} \label{Thm:RepdP}
Assume that $\dgal{-,-}$ is a double bracket on $\mc V$. Then there is a unique skewsymmetric biderivation on $\mc V_N$ which satisfies for any $a,b \in \mc V$, $1 \leq i,j \leq N$, 
\begin{equation} \label{Eq:relPA}
  \br{a_{ij},b_{kl}}= \dgal{a, b}'_{kj} \dgal{a, b}''_{il}\,.
\end{equation}
Furthermore, if $\dgal{-,-}$ is a double Poisson bracket, then $(\mc V_N,\br{-,-})$ is a  Poisson algebra. 
\end{theorem}


\subsection{Definition and first properties}

In the sequel we assume that $\VV$ is a unital associative algebra endowed with an infinite order automorphism $S\in \Aut(\VV)$.

\begin{definition} \label{def:DMlamBr}
  A \emph{double multiplicative $\lambda$-bracket} on $\VV$ is a linear map 
\begin{equation*}
\dgal{-_\lambda-} : \VV \otimes \VV \to (\VV \otimes\VV) [\lambda^{\pm1}]\,, a \otimes b \mapsto \dgal{a_\lambda b}
\end{equation*}
such that 
\begin{subequations}
  \begin{align}
    & \quad \dgal{S(a)_\lambda b}=\lambda^{-1} \dgal{a_\lambda b}\,, \quad 
\dgal{a_\lambda S(b)}=\lambda S(\dgal{a_\lambda b})\,, \quad & \text{(sesquilinearity)} \label{Eq:DMA1} \\
& \quad \dgal{a_\lambda bc}=\dgal{a_\lambda b}c+b \dgal{a_\lambda c}\,, 
\quad & \text{(left Leibniz rule)} \label{Eq:DML} \\
& \quad \dgal{ab_\lambda c}=\dgal{a_{\lambda x}c} \ast_1 \restriction{\Big(}{x=S} b\Big) 
+\restriction{\Big(}{x=S} a\Big) \ast_1 \dgal{b_{\lambda x}c}\,.  
\quad & \text{(right Leibniz rule)} \label{Eq:DMR}
  \end{align}
\end{subequations}
\end{definition}
In \eqref{Eq:DMR} and further we use the following notation (cf. Subsection \ref{sec:2.1}): for
$P(\lambda)=\sum_np_n\lambda^n\in(\mc V\otimes\mc V)[\lambda,\lambda^{-1}]$ and $a,b\in\mc V$, we let
\begin{equation}\label{notation-bar}
a(|_{x=S} P(\lambda x)b)=\sum_n a(\lambda S)^n(p_nb)\,,
\end{equation}
namely, we substitute the variable $x$ by the automorphism $S$ acting on the terms enclosed in parenthesis.
Note that the two equations in \eqref{Eq:DMA1} imply
\begin{equation}\label{Eq:DMA1b}
S\dgal{a_\lambda b}=\dgal{S(a)_\lambda S(b)}\,,
\quad a,b\in\mc V\,.
\end{equation}

Given a double multiplicative $\lambda$-bracket, we introduce the maps 
\begin{subequations}
\begin{align}
   & \dgal{a_\lambda b' \otimes b''}_L=\dgal{a_\lambda b'}\otimes b''\,, \quad 
\dgal{a_\lambda b' \otimes b''}_R=b' \otimes \dgal{a_\lambda b''}\,, \label{20210616:eq2a}\\
&\dgal{a' \otimes a''{}_\lambda b}_L=\dgal{a'{}_{\lambda x} b}\otimes_1 \restriction{\Big(}{x=S} a''\Big) \,, \label{20210616:eq2}\\ 
&\dgal{a' \otimes a''{}_\lambda b}_R=\restriction{\Big(}{x=S} a'\Big)  \otimes_1\dgal{a''{}_{\lambda x}b}\,.
\end{align}
\end{subequations}

\begin{definition} \label{def:DMPVA}
  A \emph{double multiplicative Poisson vertex algebra} is an algebra $\VV$ endowed with a multiplicative $\lambda$-bracket such that 
\begin{subequations}
  \begin{align}
    &\quad \dgal{a_\lambda b}=-\big|_{x=S} \dgal{b_{\lambda^{-1}x^{-1}}a}^\sigma
\,, \quad & \text{(skewsymmetry)} \label{Eq:DMA2} \\
   &\quad \dgal{a_\lambda \dgal{b_\mu c}}_L
-\dgal{b_\mu \dgal{a_\lambda c}}_R-\dgal{\dgal{a_\lambda b}_{\lambda\mu}c}_L=0
\,. \quad & \text{(Jacobi identity)} \label{Eq:DMA3} 
  \end{align}
\end{subequations}
\end{definition}

\begin{remark}
  Under \eqref{Eq:DMA2} the rules \eqref{Eq:DML} and \eqref{Eq:DMR} are equivalent. 
\end{remark}

The following result will be useful for computations.
\begin{lemma}\phantomsection\label{20210917:lem1}
\begin{enumerate}[(a)]
\item
The sesquilinearity relations
$$
\begin{array}{l}
\displaystyle{
\vphantom{\Big(}
\dgal{S (A)_{\lambda} B}_{L(\text{resp.}R)}
=\lambda^{-1} \dgal{ A_\lambda B}_{L(\text{resp.}R)}
\,,} \\
\displaystyle{
\vphantom{\Big(}
\dgal{A_{\lambda}S(B)}_{L(\text{resp.}R)}
=\lambda S\dgal{A_\lambda B}_{L(\text{resp.}R)}
\,,}
\end{array}
$$
hold if either $A$ or $B$ lies in $\mc V\otimes\mc V$, and the other one lies in $\mc V$.
\item
For every $a\in \mc V$ and $B,C\in\mc V^{\otimes2}$, we have
\begin{equation}\label{20140305:eq1}
\begin{array}{l}
\displaystyle{
\dgal{ a_\lambda B\bullet C}_L
=B\bullet_2\dgal{ a_\lambda C}_L
+\dgal{a_\lambda B}_L\bullet_1 C\,,
}\\
\displaystyle{
\dgal{a_\lambda B\bullet C}_R
=B\bullet_2\dgal{ a_\lambda C}_R
+\dgal{a_\lambda B}_R\bullet_3C\,.
}\\
\displaystyle{
\dgal{B\bullet C_\lambda a}_L
=\dgal{ B_{\lambda x} a}_L\bullet_3(|_{x=S}C)
+\dgal{ C_{\lambda x} a}_L\bullet_1(|_{x=S}B^\sigma)
\,.
}
\end{array}
\end{equation}
In the last equation 
we are using the notation \eqref{notation-bar}.
\end{enumerate}
\end{lemma}
\begin{proof}
Straightforward.
\end{proof}

\subsubsection{Property of Jacobi identity}

Given a double multiplicative $\lambda$-bracket on $\VV$, introduce the map 
\begin{equation} \label{Eq:Triple}
\begin{aligned}
    &\dgal{-_\lambda -_\mu-}:\VV^{\otimes 3}\to \VV^{\otimes 3}[\lambda^{\pm1},\mu^{\pm 1}]\,, \\
\dgal{a_\lambda b_\mu c}:=&\dgal{a_\lambda \dgal{b_\mu c}}_L
-\dgal{b_\mu \dgal{a_\lambda c}}_R-\dgal{\dgal{a_\lambda b}_{\lambda\mu}c}_L\,. 
\end{aligned}
\end{equation}
A direct comparison with \eqref{Eq:DMA3} yields that a double multiplicative $\lambda$-bracket which is skewsymmetric and such that the map \eqref{Eq:Triple} vanishes yields, by definition, a double multiplicative Poisson vertex algebra structure on $\VV$. 

\begin{lemma} \label{Lem:Trick}
 Given a skewsymmetric double multiplicative $\lambda$-bracket on $\VV$, we have 
\begin{equation*}
\dgal{\dgal{b_\mu a}_{\lambda \mu}c}_L=-\dgal{\dgal{a_\lambda b}^\sigma_{\lambda \mu} c}_L\,.
\end{equation*}
\end{lemma}
\begin{proof}
Using skewsymmetry \eqref{Eq:DMA2} and the first identity in Lemma \ref{20210917:lem1}(a), we have 
$$ \dgal{\dgal{b_\mu a}_{\lambda \mu}c}_L
=-\dgal{\restriction{\big(}{x=S}\dgal{a_{\mu^{-1}x^{-1}} b}^\sigma\big)_{\lambda\mu} c}_L
=-\dgal{\dgal{a_{\lambda} b}^\sigma_{\lambda\mu} c}_L\,,
$$
as desired. 
\end{proof}

As an application of this lemma, remark that we can equivalently define \eqref{Eq:Triple} as 
\begin{equation} \label{Eq:Triple2}
  \dgal{a_\lambda b_\mu c}:=\dgal{a_\lambda \dgal{b_\mu c}}_L
-\dgal{b_\mu \dgal{a_\lambda c}}_R+\dgal{\dgal{b_\mu a}^\sigma_{\lambda\mu}c}_L\,.
\end{equation}

The following properties of the operation  \eqref{Eq:Triple} can also be proven. 
\begin{lemma}
Given a double multiplicative $\lambda$-bracket $\dgal{-_\lambda-}$ on $\VV$,  
we have 
\begin{align}
  \dgal{a_\lambda b_\mu S(c)}&=\lambda\mu S(\dgal{a_\lambda b_\mu c})\,, \label{Eq:TriSes}\\
  \dgal{a_\lambda b_\mu cd}&=c \dgal{a_\lambda b_\mu d}+\dgal{a_\lambda b_\mu c}d\,. \label{Eq:TriDer}
\end{align}
Furthermore, if $\dgal{-_\lambda-}$ is skewsymmetric, we have 
\begin{equation}
\dgal{a_\lambda b_\mu c}=\Big{|}_{x=S} 
  \dgal{ b_\mu c_{\lambda^{-1}\mu^{-1}x^{-1}}a}^\sigma\,, \label{Eq:TriSkew}    
\end{equation}
In particular, given a subset 
$\mc K_0\subset \VV$ such that 
the elements of $\mc{K}=\{S^i(u)\mid u\in \mc K_0,\, i\in \Z\}$ generates $\VV$
as an associative algebra,
then the map \eqref{Eq:Triple} vanishes identically on $\mc V$ if and only if 
we have $\dgal{a_\lambda b_\mu c}=0$ for any $a,b,c\in \mc K_0$.
\end{lemma}
\begin{proof}
The proof goes along the lines of Lemma 3.4 in \cite{DSKV}. 
It is easy to get \eqref{Eq:TriSes} by combining sesquilinearity \eqref{Eq:DMA1} (only for the second argument) with the definition of the map \eqref{Eq:Triple}. In the same way, we can obtain \eqref{Eq:TriDer} from the left Leibniz rule \eqref{Eq:DML}.  

To check \eqref{Eq:TriSkew}, we note the following identities which require skewsymmetry \eqref{Eq:DMA2}: 
\begin{align*}
\dgal{a_\lambda \dgal{b_\mu c}}_L=& 
-\left(\Big|_{x=S} \dgal{\dgal{b_\mu c}'{}_{\lambda^{-1}x^{-1}} a}^\sigma \right)
\otimes \dgal{b_\mu c}''\\
=& -\Big|_{x=S} \left( \dgal{\dgal{b_\mu c}'{}_{\lambda^{-1}x^{-1}y} a}^\sigma 
\otimes \big|_{y=S}\dgal{b_\mu c}'' \right)\\
=& -\Big|_{x=S} \left( \dgal{\dgal{b_\mu c}'{}_{\lambda^{-1}x^{-1}y} a}   \otimes_1 (\big|_{y=S}\dgal{b_\mu c}'') \right)^\sigma \\
=& -\Big|_{x=S} \dgal{\dgal{b_\mu c}_{\lambda^{-1}x^{-1}}a}_L^\sigma\,,\\
\dgal{b_\mu \dgal{a_\lambda c}}_R=&
-(\big|_{x=S}\dgal{c_{\lambda^{-1}x^{-1}} a}'') \otimes
\dgal{b_\mu (\big|_{x=S}\dgal{c_{\lambda^{-1}x^{-1}} a}')}\\
=&-\Big|_{x=S}\left( \dgal{b_\mu \dgal{c_{\lambda^{-1}(x\mu)^{-1}} a}'} 
\otimes \dgal{c_{\lambda^{-1}(x\mu)^{-1}} a}''\right)^\sigma\\
=&-\Big|_{x=S}\dgal{b_\mu \dgal{c_{\lambda^{-1}\mu^{-1}x^{-1}} a}}_L^\sigma\,,\\
\dgal{\dgal{b_\mu a}^\sigma_{\lambda\mu}c}_L=&  
\left( (\big|_{x=S}\dgal{b_\mu a}') \otimes 
\dgal{\dgal{b_\mu a}''{}_{\lambda\mu x}c}^\sigma \right)^\sigma \\
=& - \big|_{x=S} \left( \dgal{b_\mu a}' \otimes 
\dgal{c{}_{\lambda^{-1}\mu^{-1} x^{-1}} \dgal{b_\mu a}'' } \right)^\sigma\\
=& - \big|_{x=S} \dgal{c{}_{\lambda^{-1}\mu^{-1} x^{-1}} \dgal{b_\mu a} }_R^{\sigma}\,.
\end{align*}
Thus, writing $\dgal{a_\lambda b_\mu c}$ through \eqref{Eq:Triple2}, we get that  \eqref{Eq:TriSkew} holds after writing the right-hand side with \eqref{Eq:Triple}.

For the second part of the lemma, note that as a consequence of \eqref{Eq:TriSkew} we can write 
\begin{align}
\dgal{a_\lambda S(b)_\mu c}=&\mu^{-1} \,\dgal{a_\lambda b_\mu c}\,, \quad 
\dgal{S(a)_\lambda b_\mu c}=\lambda^{-1}\, \dgal{a_\lambda b_\mu c}\,, \notag\\
\dgal{a_\lambda bd_\mu c}=& (\big|_{x=S} b) \ast_2 \dgal{a_\lambda d_{\mu x} c} 
+ \dgal{a_\lambda b_{\mu x} c} \ast_1 (\big|_{x=S} d)\,, \label{Eq:TriDer2}\\
\dgal{ad_\lambda b_\mu c}=& (\big|_{x=S} a) \ast_1 \dgal{d_{\lambda x} b_{\mu} c} 
+ \dgal{a_{\lambda x} b_{\mu} c} \ast_2 (\big|_{x=S} d)\,. \label{Eq:TriDer3}
\end{align}
Since we have derivation and sesquilinearity rules in the three arguments of the map \eqref{Eq:Triple}, this operation is completely determined by its value on the elements of $\mc K_0$. In particular, the map \eqref{Eq:Triple} vanishes if and only if it does when evaluated on the elements of $\mc K_0$. 
\end{proof}


\subsection{Relation to local lattice double Poisson algebras}
\label{ss:ReldmPVA-dPA}
We introduce here the notion of a local lattice double Poisson algebra,
which is equivalent to that of a double multiplicative Poisson vertex algebra.

\begin{definition}
A \emph{lattice double Poisson algebra} is a double Poisson algebra $\mc V$ with an infinite order automorphism
$S\in\Aut(\mc V)$, namely ($a,b\in\mc V$)
$$
S(a b)=S(a)S(b)\qquad\text{and} \qquad S(\dgal{a,b})=\dgal{S(a),S(b)}\,.
$$
It is called \emph{local} if, for every $a,b\in\mc V$, we have
\begin{equation}\label{20210520:eq1}
\dgal{S^n(a), b}= 0\,,\quad \text{for all but finitely many values of }n\in\mb Z\,.
\end{equation}
\end{definition}

For an element $a(\lambda)=\sum a_k\lambda^k\in\mc V^{\otimes n}[\lambda^{\pm1}]$, $n\geq1$, we define its
\emph{multiplicative residue} by
$$
\mres_\lambda a(\lambda)=a_0\,.
$$

\begin{proposition}\label{prop:20210506}
If $\mc V$ is a double multiplicative Poisson vertex algebra with double multiplicative $\lambda$-bracket
$\dgal{-_\lambda-}$ and automorphism $S\in\Aut(\mc V)$, then $\mc V$ is a local
lattice double Poisson algebra
with the double Poisson bracket 
\begin{equation}\label{eq:2.15}
\dgal{a,b}=\mres_\lambda\dgal{a_\lambda b}
\,,\quad
a,b\in\ \mc V
\,.
\end{equation}
Conversely,
if $\mc V$ is a local lattice double Poisson algebra with double Poisson bracket $\dgal{-,-}$ and 
automorphism $S\in\Aut(\mc V)$,
then we can endow it with a structure of a double multiplicative Poisson vertex algebra
with the double multiplicative $\lambda$-bracket
\begin{equation} \label{Eq:DPtoDPV}
\dgal{a_\lambda b}:=\sum_{n \in \Z}\lambda^n \dgal{S^n(a),b}\,,\quad a,b\in\mc V\,.
\end{equation}
\end{proposition}
\begin{proof}
Applying $\mres_\lambda$ to both sides of equation \eqref{Eq:DMA1b} it follows that
$S\dgal{a,b}=\dgal{S(a),S(b)}$. Cyclic skewsymmetry of the double bracket \eqref{eq:2.15} follows by
applying $\mres_\lambda$ to both sides of \eqref{Eq:DMA2} and using the fact that
$\mres_\lambda a(\lambda)=\mres_\lambda a(\lambda^{-1})$. Left and right Leibniz rules
\eqref{Eq:Dl}-\eqref{Eq:Dr}, respectively Jacobi identity, for the double bracket
\eqref{eq:2.15} follow by applying $\mres_\lambda\mres_\mu$ to \eqref{Eq:DML}-\eqref{Eq:DMR},
respectively \eqref{Eq:DMA3}.

Conversely, let $\dgal{a_\lambda b}$, $a,b\in\mc V$ be defined by \eqref{Eq:DPtoDPV}. Note that
$\dgal{a_\lambda b}\in(\mc V\otimes\mc V)[\lambda,\lambda^{-1}]$ by \eqref{20210520:eq1}.
Moreover, we have
$$
\dgal{S(a)_\lambda b}=\sum_{n\in\mb Z}\lambda^n\dgal{S^{n+1}(a),b}
=\lambda^{-1}\sum_{n\in\mb Z}\lambda^{n+1}\dgal{S^{n+1}(a),b}
=\lambda^{-1}\dgal{a_\lambda b}\,,
$$
and, using the fact that $S$ is an automorphism,
$$
\dgal{a_\lambda S(b)}=\sum_{n\in\mb Z}\lambda^n\dgal{S^{n}(a),S(b)}
=\lambda S\sum_{n\in\mb Z}\lambda^{n-1}\dgal{S^{n-1}(a),b}
=\lambda S\dgal{a_\lambda b}\,,
$$
proving sesquilinearity \eqref{Eq:DMA1}.
Next, we have
\begin{equation*}
  \begin{aligned}
\dgal{ab_\lambda c}=&\sum_n \lambda^n \left( 
S^n(a) \ast_1 \dgal{S^n(b),c} + \dgal{S^n(a),c} \ast_1 S^n(b) \right)  \\
=&\sum_n (\lambda y)^n \left( 
\restriction{\Big(}{y=S} a\Big)  \ast_1 \dgal{S^n(b),c} + \dgal{S^n(a),c} \ast_1 \restriction{\Big(}{y=S} b\Big) \right)   \\
=& \restriction{\Big(}{y=S} a\Big) \ast_1 \dgal{b_{\lambda y}c} + \dgal{a_{\lambda y}c} \ast_1 \restriction{\Big(}{y=S} b\Big) \,,
  \end{aligned}
\end{equation*}
which proves the right Leibniz rule \eqref{Eq:DMR} for \eqref{Eq:DPtoDPV}. The left Leibniz rule \eqref{Eq:DML} can be proven
similarly.
Finally, we prove the Jacobi identity for \eqref{Eq:DPtoDPV}. Note that
\begin{equation*}
  \begin{aligned}
&\dgal{a_\lambda \dgal{b_\mu c}}_L -\dgal{b_\mu \dgal{a_\lambda c}}_R \\
=&\sum_{n,m}\lambda^m \mu^n \left( 
\dgal{S^m(a), \dgal{S^n(b), c}}_L -\dgal{S^n(b), \dgal{S^m(a), c}}_R \right)\\
=& \sum_{n,m}\lambda^m \mu^n  
\dgal{\dgal{S^m(a), S^n(b)}, c}_L  \quad \text{ by \eqref{Eq:DJ}} \\
=& \sum_{n} \mu^n  
\dgal{\dgal{a_\lambda  S^n(b)}, c}_L  \,.
  \end{aligned}
\end{equation*}
Hence, by \eqref{Eq:DMA1}, we have 
\begin{equation*}
  \begin{aligned}
&\dgal{a_\lambda \dgal{b_\mu c}}_L -\dgal{b_\mu \dgal{a_\lambda c}}_R \\
=&\sum_{n} \mu^n \lambda^n
\dgal{S^n(\dgal{a_\lambda  b}') \otimes S^n(\dgal{a_\lambda  b}''), c}_L  \\
=&\sum_{n} \mu^n \lambda^n \dgal{S^n(\dgal{a_\lambda  b}') , c} \otimes_1 S^n(\dgal{a_\lambda  b}'') \\
=&\sum_{n} (\lambda \mu y)^n  \dgal{S^n(\dgal{a_\lambda  b}') , c} \otimes_1 \restriction{\Big(}{y=S}  \dgal{a_\lambda  b}''\Big) 
=\dgal{\dgal{a_\lambda b}_{\lambda \mu}c}_L\,.
  \end{aligned}
\end{equation*}
This concludes the proof.
\end{proof}

We can get several examples of double multiplicative Poisson vertex algebras using Proposition \ref{prop:20210506}. 
\begin{example}\label{Ex:Free0}
Let $V$ be a double Poisson algebra with double Poisson bracket $\dgal{-,-}$. Consider the
unital associative algebra $\mc V=\kk[S,S^{-1}]\otimes V$.
In other words, $\mc V$ is isomorphic to the direct sum of infinitely many copies of $V$ and the automorphism $S$ is the “shift” operator.
We can endow $\mc V$ with a lattice
double Poisson algebra structure with
\begin{equation}\label{20210506:eq1}
\dgal{S^m\otimes a,S^n\otimes b}=\delta_{m,n}(S^n\otimes S^n)\dgal{a,b}\,,\quad a,b\in V\,.
\end{equation}
\end{example}

\begin{example} \label{Ex:FreeA}
As a special case of Example \ref{Ex:Free0} consider the double Poisson algebra
$V$ from Example \ref{Ex:Powell} with $\alpha=1$ and $\beta=\gamma=0$.
Then $\mc V=\kk[S,S^{-1}]\otimes V=\kk\langle u_{i} \mid  i \in \Z \rangle$,
where $S(u_i)=u_{i+1}$, $i\in\mb Z$. The double Poisson bracket \eqref{20210506:eq1}
on generators then reads
$$
  \dgal{u_{i},u_{j}}=\delta_{ij} (u_j \otimes 1 - 1 \otimes u_j)\,.
$$
Using Proposition \ref{prop:20210506}, we get a double multiplicative Poisson vertex algebra structure on $\VV$ defined on generators by 
$$
 \dgal{u_{i\,\lambda}u_{j}}= \lambda^{j-i} (u_j \otimes 1 - 1 \otimes u_j)
 = (\lambda S)^{j-i} (u_i \otimes 1 - 1 \otimes u_i)\,,
$$
and extended using sequilinearity and the Leibniz rules.
\end{example}

\begin{example} \label{Ex:FreeC}
Fix $\ell\geq 1$, and consider the algebra $\VV=\kk\langle u_{r,i},v_{r,i} \mid 1\leq r \leq \ell, i \in \Z \rangle$ with the automorphism $S$ defined by $S(u_{r,i})= u_{r,i+1}$
and $S(v_{r,i})=v_{r,i+1}$. We can endow $\mc V$ with a lattice double Poisson algebra structure (cf. Example \ref{Exmp:Symp}) with
$$
  \dgal{u_{r,i},u_{s,j}}=0=\dgal{v_{r,i},v_{s,j}}\,, \quad \dgal{v_{r,i},u_{s,j}}=\delta_{rs}\delta_{ij}\, 1\otimes 1\,.
$$
By Proposition \ref{prop:20210506}, we get a double multiplicative Poisson vertex algebra structure on $\VV$ defined on generators by 
$$\dgal{u_{r,i\,\lambda}u_{s,j}}=0\,, \,\, \dgal{v_{r,i\,\lambda}v_{s,j}}=0 \,, \quad 
\dgal{v_{r,i\,\lambda}u_{s,j}}=\delta_{rs}\lambda^{j-i} \, 1\otimes 1\,,
$$
and extended using sesquilinearity and Leibniz rules.
\end{example}

\subsubsection{Finite order automorphism}

The construction of this section still holds in the case when $S$ is an automorphism of finite order $e\geq1$. In that case, from the sesquilinearity axiom \eqref{Eq:MA1}, we get that the following relation should be satisfied
for every $a,b\in\mc V$:
$$
\{a_\lambda b\}=\{S^{-e}(a)_\lambda b\}=\lambda^e\{a_\lambda b\}\,.
$$
Hence, if $S$ is an automorphism of finite order $e\geq1$,
a double multiplicative $\lambda$-bracket is a map $\dgal{-_{\lambda}-} : \mc V\otimes\mc V\to(\mc V\otimes\mc V)[\lambda]/\langle\lambda^e-1\rangle$, satisfying
\eqref{Eq:DMA1}, \eqref{Eq:DML} and \eqref{Eq:DMR}.
Then we still have Example \ref{Ex:Free0} where $\kk [S,S^{-1}]$ should be replaced by
$\kk[S]/\langle S^e-1\rangle$. Furthermore, all results of this and the next sections extend to this framework with little changes.

\begin{example} \label{Ex:NonLocA}
Fix $e\geq 1$ and consider the algebra $\VV=\kk \langle u_j \mid j \in \Z/e\Z \rangle$ with the
automorphism $S$ of $\mc V$ given by $u_i\to u_{i+1}$. On $\mc V$ we can define the double Poisson bracket
$$
  \dgal{u_{i},u_{j}}=\delta_{ij} (u_j \otimes 1 - 1 \otimes u_j)\,,
$$
which can be obtained from $e$ copies of Example \ref{Ex:Powell} with $\alpha=1,\beta=\gamma=0$.  
The automorphism $S$ of $\VV$ has order $e$ and commutes with $\dgal{-,-}$. Using Proposition \ref{prop:20210506}, we get a double multiplicative Poisson vertex algebra structure on $\VV$ completely determined by 
$$
 \dgal{u_{i\,\lambda}u_{j}}=\lambda^{[j-i]} (u_j \otimes 1 - 1 \otimes u_j)\,,
$$
where $0\leq [j-i]<e$ is the remainder of $j-i$ modulo $e$.
This example can be seen as a closed chain version of Example \ref{Ex:FreeA}. 
\end{example}

\begin{example} \label{Ex:NonLocB}
For $\ell \geq 1$, we form the algebra $\VV=\kk\langle u_1,\ldots,u_\ell,v_1,\ldots,v_\ell \rangle$, which admits a double Poisson bracket by taking $\ell$  copies of Example \ref{Exmp:Symp} :
$$
  \dgal{u_i,u_j}=0=\dgal{v_i,v_j}\,, \quad 
  \dgal{v_i,u_j}=\delta_{ij}\, 1\otimes 1\,.
$$
If we consider the automorphism $S$ of $\VV$ given by $u_i \mapsto v_i$, $v_i \mapsto - u_i$, we can show that ($k\geq0$)
\begin{equation*}
  S^{2k+1}(u_i)=(-1)^kv_i,\, S^{2k}(u_i)=(-1)^k u_i\,, \quad 
S^{2k+1}(v_i)=(-1)^{k+1}u_i,\, S^{2k}(v_i)=(-1)^k v_i\,.
\end{equation*}
Hence, $S$ has order $4$. Moreover, we have 
\begin{equation*}
S(\dgal{v_i,u_j})=\delta_{ij} 1\otimes 1=-\dgal{u_i,v_j}=\dgal{S(v_i),S(u_j)} \,,
\end{equation*}
from which it follows that $S$ commutes with $\dgal{-,-}$. Using Proposition \ref{prop:20210506}, we get a double multiplicative Poisson vertex algebra structure on $\VV$ completely determined by 
$$
\dgal{u_{i\,\lambda}u_j}=\delta_{ij}(\lambda-\lambda^3)(1\otimes 1)\,, \quad 
\dgal{v_{i\,\lambda}v_j}=\delta_{ij}(\lambda-\lambda^3)(1\otimes 1)\,, \quad
\dgal{v_{i\,\lambda}u_j}=\delta_{ij}(1-\lambda^2)(1\otimes 1)\,.
$$
\end{example}


\section{Double multiplicative Poisson vertex algebra structure on an algebra of (non-commutative) difference functions}

\label{Sec:Classif}

\subsection{The algebra of non-commutative difference operators}\label{sec:Classif1}
Let $\mc V$ be a unital associative algebra with an automorphism $S$
and consider the space $(\mc V\otimes\mc V)[S,S^{-1}]$. We extend the associative product $\bullet$ on $\mc V\otimes\mc V$ defined by \eqref{eq:bullet}
to an associative product on $(\mc V\otimes\mc V)[S,S^{-1}]$ by letting, for $a,b\in\mc V\otimes\mc V$ and $m,n\in\mb Z$:
\begin{equation}\label{eq:bullet-op}
aS^m\bullet b S^n=\left(a\bullet S^{m}(b)\right)S^{m+n}=\left(a'S^m(b')\otimes S^{m}(b'')a''\right)S^{n+m}\,,
\end{equation}
and extending it by linearity to $(\mc V\otimes\mc V)[S,S^{-1}]$. We then call $(\mc V\otimes\mc V)[S,S^{-1}]$ the algebra of scalar
(non-commutative) difference operators. The action of a scalar difference operator
$A(S)=\sum_{n\in\mb Z}a_n S^n\in(\mc V\otimes\mc V)[S,S^{-1}]$ on $f\in\mc V$ is given by
\begin{equation}\label{eq:action-diff}
A(S)f=\sum_{n\in\mb Z}a_n' S^{n}(f) a_n''
\,.
\end{equation}
The adjoint of $A(S)$ is the difference operator
\begin{equation}\label{eq:adjoint}
A^*(S)=\sum_{n\in\mb Z}S^{-n}\bullet a_n^{\sigma}
=\sum_{n\in\mb Z}\left(S^{-n}(a_n'')\otimes S^{-n}(a_n')\right)S^{-n}\,,
\end{equation}
where in the second identity we used \eqref{eq:bullet-op} (the element $(1\otimes 1)S^k\in(\mc V\otimes\mc V)[S,S^{-1}]$, $k\in\mb Z$, will be
usually simply denoted by $S^k$).
Using \eqref{eq:bullet-anti}, it is immediate to check that
\begin{equation}\label{eq:adjoint-prod}
\left(A(S)\bullet B(S)\right)^*=B^*(S)\bullet A^*(S)\,,\quad
A(S),B(S)\in(\mc V\otimes\mc V)[S,S^{-1}]
\,.
\end{equation}
The symbol of a scalar difference operator
$A(S)=\sum_{n\in\mb Z}a_n S^n\in(\mc V\otimes\mc V)[S,S^{-1}]$ is the Laurent polynomial
\begin{equation}\label{eq:symb}
A(z)=\sum_{n\in\mb Z}a_n z^n\in(\mc V\otimes\mc V)[z,z^{-1}]
\,.
\end{equation}
The formula for the symbol of products of scalar difference operators $A(S)\bullet B(S)$,
and its adjoint $(A(S)\bullet B(S))^*$ is
\begin{equation}\label{20210914:eq1}
(A\bullet B)(z)=A(zS)\bullet B(z)\,,\qquad
(A\bullet B)^*(z)=B^*(zS)\bullet A^*(z) \,.
\end{equation}
%
More generally, for every $\ell\geq1$, the space $\Mat_{\ell\times\ell}\big((\mc V\otimes\mc V)[S,S^{-1}]\big)$ is an algebra with the product
\eqref{eq:bullet-op} extended componentwise using matrix multiplication: if $H(S)=\left(H_{ij}(S)\right)_{i,j=1}^\ell$,
$K(S)=\left(K_{ij}(S)\right)_{i,j=1}^\ell\in\Mat_{\ell\times\ell}\left((\mc V\otimes\mc V)[S,S^{-1}]\right)$,
then $(H\bullet K)(S)=\left((H\bullet K)_{ij}(S)\right)_{i,j=1}^\ell\in\Mat_{\ell\times\ell}\left((\mc V\otimes\mc V)[S,S^{-1}]\right)$, where
$$
(H\bullet K)_{ij}(S)=\sum_{k=1}^{\ell} H_{ik}(S)\bullet K_{kj}(S)\in(\mc V\otimes\mc V)[S,S^{-1}]\,.$$
We call it the algebra of (non-commutative) matrix difference operators over $\mc V$.
The action of $H(S)=\left(H_{ij}(S)\right)_{i,j=1}^\ell\in\Mat_{\ell\times\ell}\left((\mc V\otimes\mc V)[S,S^{-1}]\right)$,
where
$$
H_{ij}(S)=\sum_{n\in\mb Z}H_{i,j;n}S^n\in(\mc V\otimes\mc V)[S,S^{-1}]\,,
$$
(note that the sum is finite) on a vector
$F\in\mc V^\ell$ is given by extending \eqref{eq:action-diff} componentwise
\begin{equation}\label{20210803:eq1}
(H(S)F)_i=\sum_{j=1}^\ell H_{ij}(S)F_j=\sum_{j=1}^\ell H_{i,j;n}'S^n( F_j)H_{i,j;n}''
\,,\quad i=1,\dots\ell\,.
\end{equation}
The adjoint of $H(S)$ is the matrix difference operator $H^*(S)=\left(H^*_{j,i}(S)\right)_{i,j=1}^\ell$, where
the scalar difference operator $H^*_{ji}(S)$ is obtained using \eqref{eq:adjoint}. Equation \eqref{eq:adjoint-prod} holds for
$A(S),B(S)\in\Mat_{\ell\times\ell}\left((\mc V\otimes\mc V)[S,S^{-1}]\right)$ as well.
The symbol of the matrix difference operator $H(S)$ is obtained using equation
\eqref{eq:symb} componentwise.

\subsection{Algebras of non-commutative difference functions and double multiplicative Poisson vertex algebras} \label{ss:AlgRl}
Consider the algebra of non-commutative difference polynomials 
$\mc R_\ell$ in $\ell$ variables $u_i$, $i\in I=\{1,\dots,\ell\}$.
It is the algebra of non-commutative polynomials in the indeterminates $u_{i,n}$,
$$
\mc R_\ell=\kk\langle u_{i,n}\mid i\in I,n\in\mb Z\rangle\,,
$$
endowed with an automorphism $S$, defined on generators by 
$S (u_{i,n})=u_{i,n+1}$,
and partial derivatives $\frac{\partial}{\partial u_{i,n}}:\mc R\to\mc R\otimes\mc R$, for every $i\in I$ and
$n\in\mb Z$, defined on monomials by
\begin{equation}\label{20140627:eq1}
\frac{\partial}{\partial u_{i,n}} (u_{i_1,n_1}\dots u_{i_s,n_s})
=
\sum_{k=1}^s
\delta_{i_k,i}\delta_{n_k,n}
\,
u_{i_1,n_1}\dots u_{i_{k-1},n_{k-1}}
\otimes
u_{i_{k+1},n_{k+1}}\dots u_{i_s,n_s}
\,,
\end{equation}
which are commuting $2$-fold derivations of $\mc R_\ell$ (using the terminology of \cite{DSKV}) such that
\begin{equation}\label{eq:comm}
S\circ\frac{\partial}{\partial u_{i,n}}=\frac{\partial}{\partial u_{i,n+1}}\circ S\,.
\end{equation}
In \eqref{eq:comm}, $S$ is extended to $\mc V^{\otimes 2}$ using \eqref{20210616:eq1}. 
Given the partial derivative $\frac{\partial}{\partial u_{i,n}}$, $i\in I$, $n\in\mb Z$, recall the derivations
$\left(\frac{\partial}{\partial u_{i,n}}\right)_L$,$\left(\frac{\partial}{\partial u_{i,m}}\right)_R:\mc V^{\otimes 2}\to\mc V^{\otimes3}$
defined by \eqref{Eq:DelLRExtend}.
\begin{lemma}\label{20140626:lem2}
For any non-commutative difference polynomial $f\in \mc R_\ell$, and $i,j\in I$ and
$n,m\in\mb Z$, 
the partial derivatives strongly commute, i.e. we have
$$
\left(\frac{\partial}{\partial u_{i,m}}\right)_L\frac{\partial f}{\partial u_{j,n}}
=\left(\frac{\partial}{\partial u_{j,n}}\right)_R\frac{\partial f}{\partial u_{i,m}}
\,.
$$
\end{lemma}
\begin{proof}
Same as the proof of Lemma 2.6 in \cite{DSKV}.
\end{proof}
\begin{definition}\label{def:diff-func}
An \emph{algebra of difference functions} in $\ell$ variables
is a unital associative algebra $\mc V$, with an automorphism $S$,
endowed with strongly commuting linearly independent (cf. \eqref{linear-ind}) $2$-fold derivations
$\frac{\partial}{\partial u_{i,n}}:\,\mc V\to\mc V\otimes\mc V$, $i\in I=\{1,\dots,\ell\},\,n\in\mb Z$,
such that \eqref{eq:comm} holds
and, for every $f\in\mc V$, we have $\frac{\partial f}{\partial u_{i,n}}=0$
for all but finitely many choices of indices $(i,n)\in I\times\mb Z$.
\end{definition}
An example of such an algebra is the algebra $\mc R_\ell$,
endowed with the $2$-fold derivations defined in \eqref{20140627:eq1},
or its localization by non-zero elements.
Note that $\mc R_\ell$ is in fact an algebra of difference functions in $m$ variables,
where $1\leq m\leq\ell$. In this case we should think of the variables $u_i$, $i>m$ as quasiconstants,
i.e. they lie in the kernel of the 2-fold derivations defining the structure of algebra of difference functions of $\mc V$
(see \cite{DSK-nonloc}).

\begin{theorem}\phantomsection\label{20130921:prop1}
\begin{enumerate}[(a)]
\item
Any double multiplicative $\lambda$-bracket on $\mc R_\ell$ has the form ($f,g\in \mc R_\ell$):
\begin{equation}\label{master-infinite}
\ldb f_{\lambda}g\rdb
=\sum_{\substack{i,j\in I\\m,n\in\mb Z}}
\frac{\partial g}{\partial u_{j,n}}
\bullet
\lambda ^nS^n
\ldb u_{i}{}_{\lambda x}u_{j}\rdb
\lambda^{-m}S^{-m}
\bullet
\left(\Big|_{x=S}\frac{\partial f}{\partial u_{i,m}}\right)^\sigma\,.
\end{equation}
where $\bullet$ denotes the associative product on $\mc R_\ell\otimes\mc R_\ell$ defined in \eqref{eq:bullet}, and we are using the notation \eqref{notation-bar}.
\item
Let $\mc V$ be an algebra of difference functions in $\ell$ variables.
Let $H(\lambda)$ be an $\ell\times\ell$ matrix with entries in $(\mc V\otimes\mc V)[\lambda,\lambda^{-1}]$.
We denote its entries by $H_{ij}(\lambda)=\ldb{u_j}_\lambda{u_i}\rdb$,
$i,j\in I$.
Then formula \eqref{master-infinite} defines a double multiplicative $\lambda$-bracket on $\mc V$.
\item
Equation \eqref{master-infinite} defines a structure of a double multiplicative Poisson vertex algebra on $\mc V$
if and only if the skewsymmetry axiom \eqref{Eq:DMA2} and the Jacobi identity \eqref{Eq:DMA3} hold on the $u_i$'s.
\end{enumerate}
\end{theorem}
\begin{proof}
Similar to the proof for the analogue result in \cite[Th.3.10]{DSKV}.
\end{proof}
\begin{remark}\label{rem:skew}
Let $H(S)=\left(H_{ij}(S)\right)\in\Mat_{\ell\times\ell}\left((\mc V\otimes\mc V)[S,S^{-1}]\right)$, where $H_{ij}(S)$ is the scalar difference operator with symbol $H_{ij}(\lambda)=\dgal{{u_j}_\lambda u_i}$
(cf. \eqref{eq:symb}). Recalling the definition of the adjoint matrix difference operator $H^*(S)$
(cf. equation \eqref{eq:adjoint}), then skewsymmetry axiom \eqref{Eq:DMA2}
on generators means that the matrix difference operator $H(S)$ is skew-adjoint, that is $H^*(S)=-H(S)$.
\end{remark}

\begin{definition}\label{poisson-structure}
A matrix difference operator 
$H(S)=\big( H_{ij}(S)\big)_{i,j=1}^\ell$,
with entries $H_{ij}(S)\in(\mc V\otimes\mc V)[S,S^{-1}]$,
such that the double multiplicative $\lambda$-bracket \eqref{master-infinite}
with $\dgal{{u_j}_\lambda u_i}=H_{ij}(\lambda)$
satisfies skewsymmetry \eqref{Eq:DMA2} and Jacobi identity \eqref{Eq:DMA3} on the generators $u_i$'s
of the algebra of difference functions $\mc V$
is called a (local) \emph{Poisson structure} on $\mc V$.
\end{definition}

\subsection{Double multiplicative Poisson vertex algebra structures on \texorpdfstring{$\mc R_1$}{R1}} 

In this section we provide
some classification results
of double multiplicative Poisson vertex algebra structures
on $\mc R:=\mc R_1=\kk\langle u_i|i\in\mb Z\rangle$.
Let
$$
H(S)=\sum_{k\in\mb Z}f_kS^k\in(\mc R\otimes\mc R)[S,S^{-1}]
$$
be a difference operator with coefficients in $\mc R\otimes\mc R$ and define 
a double multiplicative $\lambda$-bracket on $\mc R$ using the Master Formula
\eqref{master-infinite} where
\begin{equation}\label{20200713:eq1}
\dgal{u_\lambda u}=H(\lambda)=\sum_{k\in\mb Z}f_k\lambda^k
\,.
\end{equation}
We have that
$$
-\ldb u_{\lambda^{-1}S^{-1}} u\rdb^\sigma
=-\sum_{k\in\mb Z}S^kf_{-k}^\sigma\lambda^k
\,.
$$
Hence, skewsymmetry holds on generators if and only if
\begin{equation} \label{Eq:R1-skew}
f_k=-S^{k}f_{-k}^\sigma
\,,
\qquad
k\in\mb Z
\,.
\end{equation}
Next, using the Master Formula \eqref{master-infinite} and equation \eqref{20200713:eq1} we have that Jacobi identity on generators
becomes the following equation
\begin{equation}\label{20200713:eq2}
\begin{split}
\sum_{i,j,k\in\mb Z}&\left(
\lambda^{i+k}\mu^j\left(\frac{\partial f_j'}{\partial u_i}\bullet S^if_k\right)\otimes f_j''
-\lambda^{j}\mu^{i+k}f_j'\otimes\left(\frac{\partial f_j''}{\partial u_i}\bullet S^if_k\right)
\right.
\\
&\left.
-\lambda^{i+j+k}\mu^{i+k}\left(f_k\bullet S^{i+k}\left(\frac{\partial f_j'}{\partial u_{-i}}\right)^{\sigma}\right)\otimes_1 S^{i+k}f_j''
\right)=0
\,.
\end{split}
\end{equation}
Equation \eqref{20200713:eq2} can also be rewritten as
\begin{equation}\label{20200713:eq2b}
\begin{split}
\sum_{i,j,k\in\mb Z}&\left(
\lambda^{i+k}\mu^j\left(\frac{\partial }{\partial u_i}\right)_L(f_j)\bullet_3 S^if_k
-\lambda^{j}\mu^{i+k}\left(\frac{\partial }{\partial u_i}\right)_R(f_j)\bullet_1S^i f_k
\right.
\\
&\left.
-\lambda^{i+j+k}\mu^{i+k}\left( S^{i+k}\left(\frac{\partial }{\partial u_{-i}}\right)_L(f_j)\bullet_3 f_k^\sigma\right)^{\sigma^2} \right)=0
\,.
\end{split}
\end{equation}
Arguing similarly to Lemma 2.6 in \cite{DSKVWclass}, it is possible to show that
\begin{equation} \label{Eq:R1-trunc}
f_k=f_k(u,u_1,\dots u_k)
\,,
\quad k\geq 0
\,.    
\end{equation}
Indeed, for any $i>k\geq 0$ we obtain that $f_k$ is independent of $u_i$ by looking at the terms in $\lambda^{i+N}\mu^k$ and $\lambda^k \mu^{N+i}$ with $N=\max\{i|f_i\neq 0\}$. We can 
obtain in the same way that for $i<-k\leq 0$, $f_{-k}$ is independent of $u_i$; this is equivalent to having $f_k$ independent of $u_i$ for $i<0$.

Let us assume that $f_k=0$, for $k\neq0$, in \eqref{20200713:eq1}. By \eqref{Eq:R1-trunc}
and \eqref{Eq:R1-skew}
we have that
$$
\dgal{u_\lambda u}=f\,,\quad\text{where }f=f(u)=-f^{\sigma}
\,.
$$
In this case, the Jacobi identity \eqref{20200713:eq2b} reads
$$
\left(\frac{\partial }{\partial u}\right)_L(f)\bullet_3 f
-\left(\frac{\partial }{\partial u}\right)_R(f)\bullet_1 f
-\left( \left(\frac{\partial }{\partial u}\right)_L(f)\bullet_3 f^\sigma\right)^{\sigma^2}=0
\,.
$$
This is the same condition defining double Poisson structures on $\kk[u]$. By the results
in \cite{P16,VdB1} (see Example \ref{Ex:Powell}) we have that
$$
f=\alpha(u\otimes1-1\otimes u)+\beta(u^2\otimes1-1\otimes u^2)
+\gamma(u^2\otimes u-u\otimes u^2)\,,
$$
where $\alpha,\beta,\gamma\in\kk$ are such that $\beta^2=\alpha\gamma$.

Next, we study the case when $f_k\neq0$ in \eqref{20200713:eq1}, for some $k\in\mb Z$.
In the sequel, we will use the following result.
\begin{lemma}\label{20210812:lem1}
Let $f,g\in\mc R\otimes\mc R$ and $k\in\mb Z$.
\begin{enumerate}[(a)]
\item If
$$
\left(\frac{\partial }{\partial u_k}\right)_L(f)\bullet_3 g
=
\left( \left(\frac{\partial }{\partial u_k}\right)_L(g)\bullet_3 f^\sigma\right)^{\sigma^2}
\,,
$$
then
$$
f'=s u_k+q\,,\quad g'=u_k r+p\,,
\quad \text{sp=qr}
\,,
$$
where $s,q,r,p$ lie in the kernel of $\frac{\partial}{\partial u_k}$.
\item If
$$
\left(\frac{\partial}{\partial u_k}\right)_R(f)\bullet_1g
=\left(\left(\frac{\partial}{\partial u_{k}}\right)_L(g^\sigma)\bullet_3 f^\sigma\right)^{\sigma^2}
\,,
$$
then
$$
f''=u_k s+q\,,\quad g''=ru_k+p\,,
\quad \text{ps=rq}
\,,
$$
where $s,q,r,p$ lie in the kernel of $\frac{\partial}{\partial u_k}$.
\end{enumerate}
\end{lemma}
\begin{proof}
Let us prove part (a). Using \eqref{eq:DLR}, \eqref{eq:sigma} and \eqref{20210624:eq1}
we have
$$
\left(\frac{\partial }{\partial u_k}\right)_L(f)\bullet_3 g
=\left(\frac{\partial f'}{\partial u_k}\right)'g'
\otimes g''\left(\frac{\partial f'}{\partial u_k}\right)''
\otimes f''
$$
and
$$
\left( \left(\frac{\partial }{\partial u_k}\right)_L(g)\bullet_3 f^\sigma\right)^{\sigma^2}
=f'\left(\frac{\partial g'}{\partial u_k}\right)''
\otimes g''
\otimes \left(\frac{\partial g'}{\partial u_k}\right)'f''
\,.
$$
Hence we need to solve the equation
\begin{equation}\label{20210813:eq3}
\left(\frac{\partial f'}{\partial u_k}\right)'g'
\otimes g''\left(\frac{\partial f'}{\partial u_k}\right)''
\otimes f''
=
f'\left(\frac{\partial g'}{\partial u_k}\right)''
\otimes g''
\otimes \left(\frac{\partial g'}{\partial u_k}\right)'f''
\,.
\end{equation}
This gives the conditions
$$\left(\frac{\partial f'}{\partial u_k}\right)''\in\kk\,,
\quad
\left(\frac{\partial g'}{\partial u_k}\right)'\in\kk\,,
$$
from which we get $f'=s u_k+q$ and $g'=u_k r+p$, with $s,q,r,p$ in the kernel of $\frac{\partial}{\partial u_k}$.
Hence, $\frac{\partial f'}{\partial u_k}=s\otimes1$ and $\frac{\partial g'}{\partial u_k}=1\otimes r$. Substituting these expressions in \eqref{20210813:eq3} we get
that $s,q,r,p$ should satisfy
$$
s(u_k r+p)=(su_k+q)r
\,,
$$
which implies $sp=qr$ and concludes the proof of the claim. Part (b) is proven similarly.
\end{proof}
Let $g=g(u)\in\kk[u]\otimes\kk[u]\subset\mc R\otimes\mc R$ (that is $\frac{\partial g}{\partial u_n}=0$, for every $n\neq0$) and let $r(\lambda)\in\kk[\lambda,\lambda^{-1}]$ be such that
$r(\lambda^{-1})=-r(\lambda)$. We consider the double multiplicative $\lambda$-bracket on $\mc R$
defined by 
\begin{equation}\label{20210812:eq1}
\dgal{u_\lambda u}=g\bullet r(\lambda S)g^{\sigma}
\,.
\end{equation}
\begin{proposition}\label{20210812:prop1}
The $\lambda$-bracket \eqref{20210812:eq1} defines a double multiplicative Poisson vertex algebra
structure on $\mc R$ if and only if $g$ is of the form (up to a constant multiple that can be absorbed in $r(\lambda)$)
\begin{equation}\label{eq:g}
g=(\alpha u+\beta)\otimes(\alpha u+\beta)\,,\quad
\alpha,\beta\in\kk
\,.
\end{equation}
\end{proposition}
\begin{proof}
The $\lambda$-bracket \eqref{20210812:eq1} is clearly skewsymmetric 
in view of the assumption on $r(\lambda)$ and \eqref{eq:bullet-anti}.
By a direct computation, using Lemma \ref{20210917:lem1}, equations \eqref{20210616:eq2a}, \eqref{20210616:eq2}, the Master Formula \eqref{master-infinite},
the assumption on $r(\lambda)$ and Lemma \ref{lemma:bullet-i}(b)-(c), the Jacobi identity
on generators becomes
\begin{equation}\label{20210812:eq2}
    \begin{split}
&\left(\left(
\left(\frac{\partial}{\partial u}\right)_L(g)\bullet_3g-\left(\frac{\partial}{\partial u}\right)_R(g)\bullet_1g\right)
\bullet_1r(\lambda S)g^\sigma\right)\bullet_3r(\mu S)g^\sigma
\\
&+g\bullet_2 r(\lambda\mu S)\left(
\left(
\left(\frac{\partial}{\partial u}\right)_L(g^\sigma)\bullet_3g-g^\sigma\bullet_2\left(\left(\frac{\partial}{\partial u}\right)_L(g)\right)^{\sigma^2}\right)
\bullet_3r(\lambda S)g^\sigma
\right)
\\
&-g\bullet_2 r(\lambda\mu S)\left(
\left(
\left(\frac{\partial}{\partial u}\right)_R(g^\sigma)\bullet_1g-g^\sigma\bullet_2\left(\left(\frac{\partial}{\partial u}\right)_R(g)\right)^{\sigma}\right)
\bullet_1r(\mu S)g^\sigma
\right)=0\,.
\end{split}
\end{equation}
It is straightforward to check that if $g$ is as in \eqref{eq:g} then the LHS of
\eqref{20210812:eq2} vanishes.
On the other hand, let us assume that $r(\lambda)$ has order $N>1$. Then, vanishing of
the
coefficient of $\lambda^{2N}\mu^N$ in the LHS of \eqref{20210812:eq2} gives
\begin{equation}\label{20210813:eq1}
\left(\frac{\partial}{\partial u}\right)_L(g^\sigma)\bullet_3g-g^\sigma\bullet_2\left(\left(\frac{\partial}{\partial u}\right)_L(g)\right)^{\sigma^2}=0\,.
\end{equation}
Using the identity ($A\in\mc R^{\otimes2},B\in\mc R^{\otimes3}$)
$$
A^\sigma\bullet_{2}B^{\sigma^2}=\left(B\bullet_3A\right)^{\sigma^2}
$$
we have that
\begin{equation}\label{20210813:eq2}
g^\sigma\bullet_2\left(\left(\frac{\partial}{\partial u}\right)_L(g)\right)^{\sigma^2}
=\left(\left(\frac{\partial }{\partial u}\right)_L(g)\bullet_3 g\right)^{\sigma^2}\,.
\end{equation}
From \eqref{20210813:eq1} and \eqref{20210813:eq2} it follows
that $g\in\kk[u]\otimes\kk[u]$ must satisfy the equation
$$
\left(\frac{\partial}{\partial u}\right)_L(g^\sigma)\bullet_3g
=\left(\left(\frac{\partial }{\partial u}\right)_L(g)\bullet_3 g\right)^{\sigma^2}
\,.
$$
By Lemma \ref{20210812:lem1}(a) with $k=0$ and $g^\sigma$ in place of $f$,
and using the fact that $\ker\frac{\partial}{\partial u}\cap\kk[u]=\kk$, we have
$g'=a u+b$, $g''=c u+d$, where $a,b,c,d\in\kk$ satisfy $ad=bc$. Hence,
up to a constant factor, it is necessary for $g$ to be as in \eqref{eq:g}.
This concludes the proof.
\end{proof}
For $N\geq1$, more generally, let us define a
skewsymmetric double multiplicative $\lambda$-bracket on $\mc R$ by 
\begin{equation} \label{Eq:class1-main}
    \ldb u_\lambda u\rdb=f\lambda^N-(\lambda S)^{-N}f^{\sigma}
\,,
\end{equation}
where $f\in\mc R\otimes\mc R$. The next result provides a classification of all the simplest non-trivial examples of  double multiplicative Poisson vertex algebra structures on $\mc R$.
\begin{proposition} \label{Pr:class1}
Fix $N\geq 1$. Then, \eqref{Eq:class1-main} defines a 
double multiplicative Poisson vertex algebra structure on $\mc R$ if and only if 
$f=g\bullet S^Ng$, where $g$ is as in \eqref{eq:g}.
\end{proposition}
We will prove the classification result by checking the conditions
given by the Jacobi identity \eqref{20200713:eq2b} for $f_N:=f$, $f_{-N}:=-S^{-N}(f^\sigma)$ and $f_k=0$ if $k\neq N,-N$. Recall that by \eqref{Eq:R1-trunc}, $f=f(u,u_1,\ldots,u_N)$.
First we need the following result.
\begin{lemma}
If the double $\lambda$-bracket \eqref{Eq:class1-main} satisfies Jacobi identity, then $f=f(u,u_N)$. \end{lemma}
\begin{proof}
For $1\leq \alpha \leq N-1$, we get by looking at the terms in $\lambda^{N+\alpha}\mu^N$ and $\lambda^{N}\mu^{N+\alpha}$ in \eqref{20200713:eq2b}  that 
 $$\left(\frac{\partial }{\partial u_\alpha}\right)_L(f)\bullet_3 S^\alpha (f)=0\,, \quad 
 \left(\frac{\partial }{\partial u_\alpha}\right)_R(f)\bullet_1 S^\alpha (f)=0\,.$$
More explicitly, we expand the above identities using \eqref{20210624:eq1} and get
 $$
 \left(\frac{\partial f'}{\partial u_\alpha}\right)'S^\alpha(f')\otimes S^\alpha(f'')\left(\frac{\partial f'}{\partial u_\alpha}\right)''
 \otimes f''=0\,,
 \quad
f'\otimes \left(\frac{\partial f''}{\partial u_\alpha}\right)'S^\alpha(f')\otimes S^\alpha(f'')\left(\frac{\partial f''}{\partial u_\alpha}\right)''
=0\,.
 $$
 Since $f\neq 0$, $\frac{\partial f'}{\partial u_\alpha}=0$ and  $\frac{\partial f''}{\partial u_\alpha}=0$ for $1\leq \alpha \leq N-1$, so that 
 $f=f'(u,u_N)\otimes f''(u,u_N)$. 
\end{proof}
\begin{proof}[Proof of Proposition \ref{Pr:class1}]
Vanishing of the coefficient of $\lambda^{2N}\mu^N$ in the LHS of \eqref{20200713:eq2b}
gives the equation
$$
\left(\frac{\partial}{\partial u_N}\right)_L(f)\bullet_3 S^N f
=
\left(\left(\frac{\partial}{\partial u_N}\right)_L(S^Nf)\bullet_3 f^\sigma\right)^{\sigma^2}
\,.
$$
By Lemma \ref{20210812:lem1}(a) with $k=N$ and $S^Nf$ in place of $g$ we get
$$
f'=s(u) u_N+q(u)\,,\quad S^Nf'=u_Nr(u_{2N})+p(u_{2N})\,,\quad sp=qr
\,,
$$
where
$s,q\in\kk[u]=\ker\frac{\partial}{\partial u_N}\cap\kk\langle u,u_N\rangle$ and
$r,p\in\kk[u_{2N}]=\ker\frac{\partial}{\partial u_N}\cap\kk\langle u_N,u_{2N}\rangle$. The condition
$sp=qr$ then implies that 
$f'=(\alpha u+\beta)(\alpha u_N+\beta)$ for some $\alpha,\beta\in\kk$. 
Similarly,
vanishing of the coefficient of $\lambda^{N}\mu^{2N}$ in the LHS of \eqref{20200713:eq2b}
and using \eqref{eq:comm} gives the equation
$$
\left(\frac{\partial}{\partial u_N}\right)_R(f)\bullet_1S^N f
=\left(\left(\frac{\partial}{\partial u_{N}}\right)_L(S^{N}f^\sigma)\bullet_3 f^\sigma\right)^{\sigma^2}
\,,
$$
which, again by Lemma \ref{20210812:lem1}(b) gives $f''=(\gamma u_N+\delta)(\gamma u+\delta)$.
Hence,
\begin{equation}\label{20210812:f}
\begin{split}
f&=f'\otimes f''=(\alpha u+\beta)(\alpha u_N+\beta)\otimes
(\gamma u_N+\delta)(\gamma u+\delta)
\\
&=\left((\alpha u+\beta)\otimes(\gamma u+\delta)\right)\bullet
S^N\left((\alpha u+\beta)\otimes(\gamma u+\delta)\right)
\,.
\end{split}
\end{equation}
Next, we show that $\alpha\delta=\beta\gamma$. By Proposition \ref{20210812:prop1}, this will conclude the proof of the claim.
To do so, we look at the vanishing of the coefficient of
$\lambda^N\mu^N$ in the LHS of 
\eqref{20200713:eq2b}. This gives the identity
\begin{equation}\label{20210812:toprove}
\left(\frac{\partial}{\partial u}\right)_L(f)\bullet_3 f
=
\left(\frac{\partial}{\partial u}\right)_R(f)\bullet_1 f
\,.
\end{equation}
From \eqref{20140627:eq1} and \eqref{20210812:f} we have
\begin{equation}\label{20210812:eq3}
\begin{split}
&\left(\frac{\partial}{\partial u}\right)_L(f)
=\alpha\otimes (\alpha u_N+\beta)\otimes(\gamma u_N+\delta)(\gamma u+\delta)
\,,
\\
&\left(\frac{\partial}{\partial u}\right)_R(f)
=(\alpha u+\beta)(\alpha u_N+\beta)\otimes(\gamma u_N+\delta)\otimes\gamma
\,.
\end{split}
\end{equation}
Substituting equations \eqref{20210812:eq3} in \eqref{20210812:toprove} we get that
$f$ need to satisfy the identity
\begin{align*}
&
\alpha (\alpha u+\beta)(\alpha u_N+\beta)
\otimes
(\gamma u_N+\delta)(\gamma u+\delta)(\alpha u_N+\beta)
\otimes
(\gamma u_N+\delta)(\gamma u+\delta)
\\
&=
(\alpha u+\beta)(\alpha u_N+\beta)
\otimes
(\gamma u_N+\delta)(\alpha u+\beta)(\alpha u_N+\beta)
\otimes
\gamma(\gamma u_N+\delta)(\gamma u+\delta)
\,,
\end{align*}
which is equivalent to $\alpha(\gamma u+\delta)=\gamma(\alpha u+\beta)$ and implies
$\alpha\delta=\beta\gamma$.
\end{proof}
\begin{remark}
In \cite{CW2}, it is shown that if $\mc R$ is a double multiplicative Poisson vertex algebra for 
$$
\dgal{u_\lambda u}=f\lambda+g+(\lambda S)^{-1}f^\sigma\,,
$$
then $g=0$ and $f$ is as in Proposition \ref{Pr:class1}.
\end{remark}

\subsection{Double multiplicative Poisson vertex algebra structures on \texorpdfstring{$\mc R_2$}{R2}}

Let us consider $\mc R_2=\kk\langle u_i,v_i|i\in\mb Z\rangle$ with a double multiplicative Poisson vertex algebra structure such that $\ldb u_\lambda u \rdb=0$ and $\ldb v_\lambda v\rdb=0$. 
The following result gives a criterion for such structure, and it is proven in \S\ref{sss:Pf-Prop2}. 

\begin{proposition} \label{Pr:class2}
 Assume that $\mc R_2$ is equipped with the skewsymmetric double multiplicative $\lambda$-bracket given by 
 \begin{equation} \label{Eq:cl2-dbr} 
\dgal{u_\lambda u}=0,\,\,\, \dgal{v_\lambda v}=0\,, \quad 
\dgal{u_\lambda v}= \sum_{k\in \Z} g_k \lambda^k \, \in \mc R_2 \otimes \mc R_2 [\lambda^{\pm 1}]\,.
 \end{equation}
Then $\mc R_2$ is a double multiplicative Poisson vertex algebra if and only if 
$$
 g_k=\sum_{a,b,c,d=0,1} K_{abcd}^k \, v^a u_k^b \otimes u_k^c v^d\,, \quad K_{abcd}^k\in \kk\,,
$$
where the following conditions are satisfied:
\begin{itemize}
 \item for all $k,l\in \Z$ distinct and for any $a,b,c,d,a',b',c',d'\in \{0,1\}$, 
 \begin{subequations}
 \begin{align}
  K_{1 b c d}^k K_{a' b' c' 0}^l=& K_{0 b c d}^k K_{a' b' c' 1}^l \,, \label{Eq:kl1} \\
K_{a b 1 d}^k K_{a' 0 c' d'}^l=& K_{a b 0 d}^k K_{a' 1 c' d'}^l \,; \label{Eq:kl2}
 \end{align}
\end{subequations}
\item for any $k\in\Z$ and for any $a,b,c,d\in \{0,1\}$, 
\begin{subequations}
 \begin{align}
  K_{a b 1 \epsilon}^k K_{\epsilon 0 c d}^k =& K_{a b 0 \epsilon}^k K_{\epsilon 1 c d}^k\,, \quad \forall \epsilon=0,1\,, \label{Eq:d1} \\
  K_{a b \epsilon 0}^k K_{1 \epsilon c d}^k =& K_{a b \epsilon 1}^k K_{0 \epsilon c d}^k\,, \quad \forall \epsilon=0,1\,, \label{Eq:c1} \\
    K_{a b 1 0}^k K_{1 0 c d}^k=&  K_{a b 0 1}^k K_{0 1 c d}^k \,, \label{Eq:d2} \\
 K_{a b 0 0}^k K_{1 1 c d}^k=& K_{a b 1 1}^k K_{0 0 c d}^l \,. \label{Eq:c2}   
 \end{align}
\end{subequations}
\end{itemize}
\end{proposition}

\begin{example}
 If only finitely many coefficients $\alpha_k:=K^k_{1111}\in \kk$ are non-zero, we get that 
 \begin{equation} \label{Eq:duv-4}
  \dgal{u_\lambda v}= \sum_{k\in \Z}\alpha_k\, vu_k \otimes u_k v\,\, \lambda^k\,, 
 \end{equation}
 yields a double multiplicative Poisson vertex algebra. 
 Indeed, all the conditions gathered in Proposition \ref{Pr:class2} are quadratic relations in which at least one factor on each side has an index $0$.
In the same way, for $\alpha_k:=K^k_{0000}\in \kk$ 
  \begin{equation}  \label{Eq:duv-triv}
  \dgal{u_\lambda v}= \sum_{k\in \Z} \alpha_k \,1 \otimes 1\,\, \lambda^k\,, 
 \end{equation}
 yields trivially a double multiplicative Poisson vertex algebra. 
\end{example}
The double $\lambda$-brackets \eqref{Eq:duv-4} and \eqref{Eq:duv-triv} are very similar to the two cases from the classification with one generator given in Proposition \ref{Pr:class1}. The following example is quadratic and has no analogue in the case of an algebra generated by one element. 

\begin{example} \label{Exmp:2-quad}
For any $\alpha\in \kk^\times$, the skewsymmetric double multiplicative $\lambda$-bracket given by 
$$
   \dgal{u_\lambda u}=0,\,\,\, \dgal{v_\lambda v}=0,\quad 
   \dgal{u_\lambda v}= \left(v\otimes u_k+u_k\otimes v + \alpha v\otimes v + \alpha^{-1} u_k\otimes u_k \right) \lambda^k\,,
$$
 yields a double multiplicative Poisson vertex algebra. 
This can be obtained by checking the conditions from Proposition \ref{Pr:class2} where for fixed $k\in \Z$, 
\begin{equation*}
 K^k_{1010}=1=K^k_{0101}\,, \quad K^k_{1001}=\alpha\,, \quad K^k_{0110}=\alpha^{-1}\,.
\end{equation*}
The four coefficients can not be chosen independently since, for example, \eqref{Eq:d2} yields   
$$(K^k_{1010})^2=  K_{1 0 0 1}^k K_{0 1 1 0}^k = (K_{0 10 1}^k)^2\,.$$ 
\end{example}

Building on the previous example, the following result is proven in \S\ref{sss:Pf-Thm2} and provides a classification when there is only one non-zero element $g_k$ in \eqref{Eq:cl2-dbr}. 

\begin{theorem} \label{Thm:2-explicit}
  Assume that $\mc R_2$ is equipped with the skewsymmetric double multiplicative $\lambda$-bracket given by 
 \begin{equation} \label{Eq:cl2-explicit}
   \dgal{u_\lambda u}=0,\,\,\, \dgal{v_\lambda v}=0, \quad
   \dgal{u_\lambda v}= g \lambda^k \,, \quad \text{ for } g\in \mc R_2 \otimes \mc R_2\,, \quad k\in \Z\,.
 \end{equation}
Then $\mc R_2$ is a double multiplicative Poisson vertex algebra if and only if after a translation 
\begin{equation} \label{Eq:cl2-transfo}
 (u,v)\mapsto (u+\mu,v+\nu)\,, \quad \mu,\nu \in \kk\,,
\end{equation}
the element $g$ satisfies exactly one of the following five conditions: 
\begin{enumerate}
 \item[(i)] $g=a\, 1\otimes 1$, $a\in \kk$;  
 \item[(ii)] $g=a\, v\otimes v$, $a\in \kk^\times$;
 \item[(iii)] $g=a\, u_k\otimes u_k$, $a\in \kk^\times$;
 \item[(iv)] $g=a\, v \otimes v+b \, [v\otimes u_k+ u_k\otimes v]+\frac{b^2}{a}\, u_k \otimes u_k$, $a,b\in \kk^\times$;
  \item[(v)] $g=a\, vu_k \otimes u_k v+ b \, [vu_k\otimes 1 + 1\otimes u_kv]+\frac{b^2}{a}\, 1\otimes 1$, $a\in \kk^\times$, $b\in \kk$.
\end{enumerate}
\end{theorem}

Note that the distinct cases can not be related through \eqref{Eq:cl2-transfo}, but some are equivalent if one uses linear transformations. Indeed, cases (ii) and (iii) in Theorem \ref{Thm:2-explicit} are related through 
$$
 (u,v)\mapsto (v,u)\,, \quad a\mapsto - a\,, \quad k \mapsto -k\,,
$$
while cases (ii) and (iv) are equivalent under the map  $(u,v)\mapsto \big(u,v-\frac{b}{a}u_k\big)$.

\begin{remark}
 If we have a double $\lambda$-bracket satisfying Theorem \ref{Thm:2-explicit} of the form 
$$
  \dgal{u_\lambda u}=0\,, \quad \dgal{v_\lambda v}=0\,, \quad    \dgal{u_\lambda v}= g_0 \in \mc R_2\otimes \mc R_2\,,
$$
then it defines a double Poisson bracket that is compatible with the automorphism $S$. Note that quadratic double Poisson brackets on free algebras are classified in \cite{ORS}. Modulo a linear transformation, the three quadratic cases from Theorem \ref{Thm:2-explicit} (with $k=0$) are all equivalent to 
$$
  \dgal{u,u}=0\,, \quad \dgal{v,v}=0\,, \quad    \dgal{u,v}= v\otimes v\,,
$$
which corresponds to Case 4 in \cite[Theorem 1]{ORS}. We also note that the quartic case from condition \emph{(v)} in Theorem \ref{Thm:2-explicit} with $k=0$ is a new example of double Poisson bracket, to the best of our knowledge. 
\end{remark}

\begin{remark}
In view of the theory presented in Section \ref{S:Integr}, the interest of Proposition \ref{Pr:class2} lies in the fact that these $\lambda$-brackets give rise to commuting families of  differential-difference equations as we have $\mult \dgal{u^k{}_\lambda u^l}=0$ trivially.  The same holds with $v$ replacing $u$.
\end{remark}

\subsubsection{Proof of Proposition \ref{Pr:class2}}  \label{sss:Pf-Prop2}

We assume that $\dgal{u_\lambda v}=\sum_k g_k \lambda^k\neq 0$ from now on. We first remark that the Jacobi identity \eqref{Eq:DMA3} with $a=b=u$ and $c=v$ has only its first two terms which are non-zero since $\dgal{u_\lambda u}=0$. Denoting it $\dgal{u_\lambda u_\mu v}$, we can then compute that 
\begin{equation}
 \begin{aligned} \label{Eq:uuv1}
  \dgal{u_\lambda u_\mu v}=& 
  \sum_{k,l,n \in \Z} \lambda^{n+l} \mu^k \, \left(\frac{\partial}{\partial v_n} \right)_L(g_k) \bullet_3 S^n(g_l) \\
  &- \sum_{j,m,s \in \Z} \lambda^{j} \mu^{m+s} \, \left(\frac{\partial}{\partial v_m} \right)_R(g_j) \bullet_1 S^m(g_s)\,,
 \end{aligned}
\end{equation}
which must vanish. Similarly, we compute 
\begin{equation}
 \begin{aligned} \label{Eq:vvu1}
  \dgal{v_\lambda v_\mu u}=& 
  \sum_{k,l,n \in \Z} \lambda^{n+l} \mu^k \, S^k\left(\frac{\partial}{\partial u_{n-k}} \right)_L(g_{-k}^\sigma) 
\bullet_3 S^{n+l}(g_{-l}^\sigma) \\
  &- \sum_{j,m,s \in \Z} \lambda^{j} \mu^{m+s} \, S^j \left(\frac{\partial}{\partial u_{m-j}} \right)_R(g_{-j}^\sigma) \bullet_1 S^{m+s}(g_{-s}^\sigma)\,,
 \end{aligned}
\end{equation}
which must also vanish.

\begin{lemma}
For any $k\in\Z$, $g_k$ depends only on $v$ and $u_k$.
\end{lemma}
\begin{proof}
 Since $\dgal{u_\lambda v} \in \mc R_2 \otimes \mc R_2 [\lambda^{\pm 1}]$, there exists 
$$
  N_+=\max\{k \mid g_k \neq 0\}\,, \quad N_-=\min\{k \mid g_k \neq 0\}\,.
$$
We now adapt \cite[Lemma 2.6]{DSKVWclass}. Denoting $g_k=g_k' \otimes g_k''$, we introduce for all $k$ 
$$
 i_k'=\max\{i \mid \frac{\partial g_k'}{\partial v_i}\neq0\}\,, \quad 
 j_k'=\min\{i \mid \frac{\partial g_k'}{\partial v_i}\neq0\}\,.
$$
Assuming that $i_k'>0$, we see that the term in $\lambda^{i_k'+N_+}\mu^k$ in \eqref{Eq:uuv1} is 
$\left(\frac{\partial}{\partial v_{i_k'}} \right)_L(g_k) \bullet_3 S^{i_k'}(g_{N_+})$, which is non-zero by assumption, a contradiction. Thus $i_k'\leq 0$. 
Similarly, if $j_k'<0$, we look at the term in $\lambda^{j_k'+N_-}\mu^k$ in \eqref{Eq:uuv1} and get a contradiction. Thus the dependence of $g_k'$ on the $(v_s)$ is only on $v=v_0$. 

In the exact same way, introduce 
$$
 i_k''=\max\{i \mid \frac{\partial g_k''}{\partial v_i}\neq0\}\,, \quad 
 j_k''=\min\{i \mid \frac{\partial g_k''}{\partial v_i}\neq0\}\,.
$$
If $i_k''>0$ or $j_k''<0$, we look at the terms in $\lambda^{k}\mu^{i_k''+N_+}$ or $\lambda^{k}\mu^{j_k''+N_-}$ in \eqref{Eq:uuv1} and get contradictions. Thus $g_k$ depends only on $v$ and the $(u_s)$. 

Next, we do the same with \eqref{Eq:vvu1}. For 
$$
 r_k''=\max\{i \mid \frac{\partial g_k''}{\partial u_i}\neq0\}\,, \quad 
 s_k''=\min\{i \mid \frac{\partial g_k''}{\partial u_i}\neq0\}\,,
$$
we note that if $r_k''>k$ or $s_k''<k$, then $k+r_{-k}''>0$ or $k+s_{-k}''<0$, so that by looking at the term in 
$\lambda^{k+r_{-k}''-N_-}\mu^k$ or $\lambda^{k+s_{-k}''-N_+}\mu^k$ in \eqref{Eq:vvu1} we get contradictions. 
Similarly, for 
$$
 r_k'=\max\{i \mid \frac{\partial g_k'}{\partial u_i}\neq0\}\,, \quad 
 s_k'=\min\{i \mid \frac{\partial g_k'}{\partial u_i}\neq0\}\,,
$$
we get contradictions if $r_k'>k$ or $s_k'<k$ by looking at the terms in $\lambda^k \mu^{k+r_{-k}'-N_-}$ or $\lambda^k \mu^{k+s_{-k}'-N}$ in \eqref{Eq:vvu1}. Hence $g_k$ can only depend on $u_k$. 
\end{proof}

\begin{lemma}
 We have that 
$$
 g_k= \sum_{a,b,c,d=0,1} K_{abcd}^k \, v^a u_k^b \otimes u_k^c v^d\,, \quad K_{abcd}^k\in \kk\,.
$$
\end{lemma}
\begin{proof}
 Due to the previous lemma, we can write \eqref{Eq:uuv1} as 
 \begin{equation}
 \begin{aligned} \label{Eq:uuv2}
  \dgal{u_\lambda u_\mu v}=& 
  \sum_{k,l \in \Z} \lambda^{l} \mu^k \,\left[  \left(\frac{\partial}{\partial v} \right)_L(g_k) \bullet_3 g_l 
  - \left(\frac{\partial}{\partial v} \right)_R(g_l) \bullet_1 g_k   \right]\,. 
 \end{aligned}
\end{equation}
This vanishes if the following coefficient of $\lambda^l \mu^k$ is zero : 
\begin{equation} \label{T1lk}
 T^1_{l,k}:= \left(\frac{\del g_k'}{\del v} \right)' g_l' \otimes g_l'' \left(\frac{\del g_k'}{\del v} \right)'' \otimes g_k'' 
 - g_l' \otimes \left(\frac{\del g_l''}{\del v} \right)' g_k' \otimes g_k'' \left(\frac{\del g_l''}{\del v} \right)''\,.
\end{equation}
We can look at the first and third copies in the tensor product, and we see that to have cancellations we need 
$$
 \left(\frac{\del g_k'}{\del v} \right)'\in \kk\,, \quad \left(\frac{\del g_l''}{\del v} \right)''\in \kk\,.
$$
These conditions mean that we must have for any $k$ that 
$$
 g_k=v p_{k,1}(u_k) v + v p_{k,2}(u_k) + p_{k,3}(u_k) v + p_{k,4}(u_k)\,,
$$
for some $p_{k,i}(z)\in \kk[z]\otimes \kk[z]$. 

Next, we remark that we can write \eqref{Eq:vvu1} as 
\begin{equation}
 \begin{aligned} \label{Eq:vvu2}
  \dgal{v_\lambda v_\mu u}=& 
  \sum_{k,l \in \Z} \lambda^{l} \mu^k \, \left[ S^k \left(\frac{\partial}{\partial u_{-k}} \right)_L (g_{-k}^\sigma) 
\bullet_3 S^{l}(g_{-l}^\sigma) 
- S^l \left(\frac{\partial}{\partial u_{-l}} \right)_R(g_{-l}^\sigma) \bullet_1 S^{k}(g_{-k}^\sigma)
\right] \,.
 \end{aligned}
\end{equation}
So the term in $\lambda^l \mu^k$ is 
\begin{equation} \label{T2lk}
\begin{aligned}
  T^2_{l,k}:=& 
  \left(\frac{\del S^k(g_{-k}'')}{\del u} \right)' S^l(g_{-l}'') \otimes S^l(g_{-l}') \left(\frac{\del S^k(g_{-k}'')}{\del u} \right)''
  \otimes S^k(g_{-k}') \\
& - S^l(g_{-l}'') \otimes \left(\frac{\del S^l(g_{-l}')}{\del u} \right)' S^k(g_{-k}'') 
\otimes S^k(g_{-k}') \left(\frac{\del S^l(g_{-l}')}{\del u} \right)''\,.
\end{aligned}
\end{equation}
Looking at the first and third copies in the tensor product, we see that we need 
$$
 \left(\frac{\del S^k(g_{-k}'')}{\del u} \right)' \in \kk\,, \quad \left(\frac{\del S^l(g_{-l}')}{\del u} \right)''\in \kk\,.
$$
These conditions mean that we must have for any $k$ that 
$$
 g_{-k}^\sigma=u_{-k} q_{k,1}(v) u_{-k}  + u_{-k} q_{k,2}(v) + q_{k,3}(v) u_{-k} + q_{k,4}(v)\,,
$$
for some $q_{k,i}(z)\in \kk[z]\otimes \kk[z]$. Gathering both conditions yield the result.  
\end{proof}

To finish the proof of the proposition, we remark that the two conditions \eqref{Eq:d2}--\eqref{Eq:c2} with fixed $k\in \Z$ and $a,b,c,d\in \{0,1\}$ are equivalent to the two identities 
\begin{subequations}
  \begin{align}
   K_{a b 1 0}^k K_{1 0 c d}^k + K_{a b 1 1}^k K_{0 0 c d}^k=& 
   K_{a b 0 0}^k K_{1 1 c d}^k + K_{a b 0 1}^k K_{0 1 c d}^k \,, \label{Eq:d2-alt} \\
   K_{a b 1 0}^k K_{1 0 c d}^k + K_{a b 0 0}^k K_{1 1 c d}^k=& 
   K_{a b 1 1}^k K_{0 0 c d}^k + K_{a b 0 1}^k K_{0 1 c d}^k \,. \label{Eq:c2-alt}   
  \end{align}
\end{subequations}

\begin{lemma}
 We have $\dgal{u_\lambda u_\mu v}=0$ if and only if the identities \eqref{Eq:kl1}, \eqref{Eq:c1} and \eqref{Eq:c2-alt} hold. 
\end{lemma}
\begin{proof}
 The vanishing of $\dgal{u_\lambda u_\mu v}$ is equivalent to the vanishing of each $T^1_{l,k}$ \eqref{T1lk}. Plugging the form of $g_k$ in it, we get that 
$$
\sum_{a_ib_ic_id_i=0,1} K^k_{a_1 b_1 c_1 d_1} K^l_{a_2 b_2 c_2 d_2}
\left[ a_1\, v^{a_2} u_l^{b_2} \otimes u_l^{c_2} v^{d_2}  u_k^{b_1} \otimes u_k^{c_1} v^{d_1}
- d_2\, v^{a_2} u_l^{b_2} \otimes u_l^{c_2} v^{a_1}  u_k^{b_1} \otimes u_k^{c_1} v^{d_1}\right]\,.
$$
So the factor appearing in front of the term $v^{a_2} u_l^{b_2} \otimes u_l^{c_2} v^{e}  u_k^{b_1} \otimes u_k^{c_1} v^{d_1}$ is 
\begin{equation} \label{T1lkConst}
  K^k_{1 b_1 c_1 d_1} K^l_{a_2 b_2 c_2 e}
 -  K^k_{e b_1 c_1 d_1} K^l_{a_2 b_2 c_2 1}\,.
\end{equation}
If $e=+1$, this factor must vanish, but this is always true. If $e=0$ instead and $k\neq l$, \eqref{T1lkConst} must vanish, and this is equivalent to \eqref{Eq:kl1}. If $e=0$ and $k=l$, note that the  terms with the same $b_1+c_2$ add up since the second factor of the tensor product becomes $u_k^{b_1+c_2}$. 
If $c_2=b_1=0$ or $c_2=b_1=+1$, we get that \eqref{T1lkConst} vanishes which is equivalent to \eqref{Eq:c1}. 

If $c_2+b_1=+1$, we sum up the coefficients \eqref{T1lkConst} for $c_2=0,b_1=+1$ and $c_2=+1,b_1=0$, and this sum must vanish. This is \eqref{Eq:c2-alt}. 
\end{proof}

In the same way, we prove the next result : 
\begin{lemma}
 We have $\dgal{v_\lambda v_\mu u}=0$ if and only if the identities \eqref{Eq:kl2},  \eqref{Eq:d1} and \eqref{Eq:d2-alt} hold. 
\end{lemma}
\begin{proof} 
The vanishing of $\dgal{v_\lambda v_\mu u}$ is equivalent to the vanishing of each $T^2_{l,k}$ \eqref{T2lk}, which can be simplified as  
\begin{equation} \label{T2lkConst}
\sum_{a_ib_ic_id_i=0,1} [ K^k_{a_1 b_1 1 d_1} K^l_{a_2 0 c_2 d_2}
 -  K^k_{a_1 b_1 0 d_1} K^l_{a_2 0 c_2 d_2}]\,
u^{c_2} v_l^{d_2} \otimes v_l^{a_2} v_k^{d_1} \otimes v_k^{a_1} u^{b_1} \,.
\end{equation}
The terms with $k\neq l$ must vanish and are equivalent to \eqref{Eq:kl2}. When $k=l$, analysing the cases for $a_2+d_1\in \{0,1,2\}$ gives 
\eqref{Eq:d1} and \eqref{Eq:d2-alt}. 
\end{proof}

This finishes the proof of Proposition \ref{Pr:class2}. 

\subsubsection{Proof of Theorem \ref{Thm:2-explicit}} \label{sss:Pf-Thm2}

By Proposition \ref{Pr:class2}, we can write  
$$
 g=\sum_{a,b,c,d=0,1} K_{abcd} \, v^a u_k^b \otimes u_k^c v^d\,, \quad K_{abcd}\in \kk\,,
$$
and the constants $K_{abcd}$ satisfy \eqref{Eq:d1}--\eqref{Eq:c2} (with the index $k$ omitted). We will repeatedly need the following identities, which are special cases of  \eqref{Eq:d1}--\eqref{Eq:c2}  containing squares: 
\begin{subequations}
 \begin{align}
K_{1011}^2=&K_{1111}K_{1001}=K_{1101}^2\,, \quad 
K_{0010}^2=K_{0000}K_{0110}=K_{0100}^2\,, \label{Eq:18a}\\
K_{1110}^2=&K_{1111}K_{0110}=K_{0111}^2\,, \quad 
K_{1000}^2=K_{0000}K_{1001}=K_{0001}^2\,, \label{Eq:18b}\\
K_{1100}^2=&K_{1111}K_{0000}=K_{0011}^2\,, \quad
K_{1010}^2=K_{1001}K_{0110}=K^2_{0101}\,. \label{Eq:18cd}
 \end{align}
\end{subequations}

\underline{\textbf{A.} $K_{1111}=0$.} We must have  
\begin{align*}
    g=&K_{0110} u_k\otimes u_k + K_{1001} \,v\otimes v + K_{1010} v\otimes u_k + K_{0101} u_k\otimes v\\
 &+ K_{1000}\, v\otimes 1+ K_{0100} u_k \otimes 1 + K_{0010} 1 \otimes u_k + K_{0001} 1\otimes v + K_{0000}\, 1\otimes 1\,.
\end{align*}

\underline{\textbf{A.1.}} Assume furthermore that $K_{1001}=K_{0110}=0$. All the coefficients except $K_{0000}$ must be zero, and the latter can take an arbitrary value while \eqref{Eq:d1}--\eqref{Eq:c2} are satisfied. This is case (i). 

\underline{\textbf{A.2.}} Assume furthermore that $K_{0110}=0$, $K_{1001}\neq 0$. We must have 
\begin{equation*}
   g= K_{1001} \,v\otimes v + K_{1000}\, v\otimes 1 + K_{0001} 1\otimes v + K_{0000}\, 1\otimes 1\,.
\end{equation*}
 By \eqref{Eq:c1}, $K_{1000}K_{1001}=K_{1001}K_{0001}$ so that $K_{1000}=K_{0001}$. Up to making the translation $v\mapsto v-\frac{K_{1000}}{K_{1001}}$, we can assume that $g=K_{1001} \,v\otimes v + K_{0000}\, 1\otimes 1$. Using \eqref{Eq:18b}, $K_{0000}=0$ and all the conditions \eqref{Eq:d1}--\eqref{Eq:c2} are satisfied for an arbitrary $K_{1001}$; this is case (ii). 

\underline{\textbf{A.3.}} Assume furthermore that $K_{1001}=0$, $K_{0110}\neq 0$. By an argument similar to \textbf{A.2}, we are in case (iii) of the statement after a translation. 
 
\underline{\textbf{A.4.}} If $K_{1001}\neq 0$ and $K_{0110}\neq 0$, we can adapt the previous arguments to reduce to the case 
\begin{equation*}
g=K_{0110} u_k\otimes u_k + K_{1001} \,v\otimes v + K_{1010} v\otimes u_k + K_{0101} u_k\otimes v \,,
\end{equation*}
after a translation. Note that the second identity in \eqref{Eq:18cd} must hold, hence the four coefficients are non-zero. Moreover, we get from \eqref{Eq:d2} that $K_{1001}K_{1010}=K_{1001}K_{0101}$, from which $K_{1010}=K_{0101}$. We must then be in case (iv), and it is easy to see that such a form will always satisfy the conditions \eqref{Eq:d1}--\eqref{Eq:c2}. 

\medskip 

\underline{\textbf{B.} $K_{1111}\neq 0$.}  
We have as special cases of \eqref{Eq:d1} and \eqref{Eq:c1} that 
\begin{equation*}
 K_{1111}K_{1011}=K_{1101}K_{1111}\,, \quad K_{1110}K_{1111}=K_{1111}K_{0111}\,.
\end{equation*}
Hence after the translation $u\mapsto u-\frac{K_{1011}}{K_{1111}}$, $v\mapsto v-\frac{K_{0111}}{K_{1111}}$, we can assume that 
there is no term in $g$ which is cubic in $(u_k,v)$. This in turn implies that  
\begin{equation*}
g=K_{1111}\, vu_k \otimes u_k v+ K_{1100} \, vu_k\otimes 1 + K_{0011}\, 1\otimes u_kv+K_{0000}\, 1\otimes 1\,.
\end{equation*}
 The remaining terms are subject to the first identity in \eqref{Eq:18cd}, as well as $K_{0000}K_{1100}=K_{0011}K_{0000}$  which is a special case of \eqref{Eq:c2}. Therefore $g$ must be of the form given in case (v), and it can be checked that $g$ satisfies all the identities in  \eqref{Eq:d1}--\eqref{Eq:c2}. 


\section{Relation to representation spaces}
\label{S:Represent}

Given an associative algebra $\VV$, recall that for $N\geq 1$ we can form the $N$-th representation algebra $\VV_N$ defined in \S\ref{sss:RepAlg}. We also have that each $S\in \Aut(\VV)$  induces an automorphism of $\VV_N$ from its definition on generators by $S(a_{ij})=(S(a))_{ij}$. 
We will prove the analogue of Theorem \ref{Thm:RepdP} and \cite[\S3.7]{DSKV} for multiplicative Poisson vertex algebras. 

\begin{theorem} \label{Thm:MainRep}
 Assume that $\dgal{-_\lambda-}$ is a double multiplicative  $\lambda$-bracket on $\VV$. Then there is a unique multiplicative $\lambda$-bracket on $\VV_N$ which satisfies for any $a,b \in \VV$, $1 \leq i,j \leq N$, 
\begin{equation} \label{Eq:relMPVA}
  \br{a_{ij\, \lambda}b_{kl}}=\sum_{n \in \Z} (a_n b)'_{kj} (a_n b)''_{il} \lambda^n\,, \quad 
\text{ where }\dgal{a_\lambda b}=\sum_{n \in \Z} ((a_n b)' \otimes (a_n b)'') \lambda^n\,.
\end{equation}
Furthermore, if $(\VV,\dgal{-_\lambda -})$ is a double multiplicative Poisson vertex algebra, then $(\VV_N,\br{-_\lambda-})$ is a multiplicative Poisson vertex algebra. 
\end{theorem}

\begin{proof}
We begin by proving the first part, which is similar to \cite[Prop. 3.20]{DSKV}.
  
The operation $\br{-_\lambda-}$ given by equation \eqref{Eq:relMPVA} is defined on generators and it is extended uniquely to all elements of $\VV_N$ by the Leibniz rules \eqref{Eq:ML}--\eqref{Eq:MR}. To ensure that $\br{-_\lambda-}$ is well-defined, we need to show 
\begin{equation} \label{Eq:RepLeibn}
    \br{a_{ij\, \lambda}(bc)_{kl}}=\sum_{u=1}^N \br{a_{ij\, \lambda}b_{cu} c_{ul}}\,,
\end{equation}
and do the same with respect to the first entry.
To see that \eqref{Eq:RepLeibn} holds, we compute the left-hand side using  \eqref{Eq:DML} as follows : 
\begin{equation*}
    \begin{aligned}
\br{a_{ij\, \lambda}(bc)_{kl}}=&
\sum_{n \in \Z} (a_n bc)'_{kj} (a_n bc)''_{il} \lambda^n \\
=&
\sum_{n \in \Z} \Big[(a_n b)'_{kj} ((a_n b)''c)_{il} + (b(a_n c)')_{kj} (a_n c)''_{il}\Big] \lambda^n  \\
=&  \sum_{n \in \Z}\sum_{u=1}^N \left[ 
(a_nb)'_{kj}(a_nb)''_{iu} c_{ul}+b_{ku}(a_nc)'_{uj}(a_nc)''_{il}
\right] \lambda^n\,.
    \end{aligned}
\end{equation*}
The same result can easily be obtained for the right-hand side of  \eqref{Eq:RepLeibn} using the Leibniz rule. 

To get that  $\br{-_\lambda-}$ defined by \eqref{Eq:relMPVA}  is a multiplicative $\lambda$-bracket, it remains to check sesquilinearity \eqref{Eq:MA1}. This will follow if we can show that 
\begin{equation} \label{Eq:RepSesq}
\br{S(a_{ij})_\lambda b_{kl}}= \lambda^{-1}\br{a_{ij\,\lambda} b_{kl}}\,, \quad \br{a_{ij}{}_\lambda S(b_{kl})}= \lambda S(\br{a_{ij\,\lambda} b_{kl}})\,. 
\end{equation}
For the first identity in \eqref{Eq:RepSesq}, we have 
\begin{equation*}
\begin{aligned}
  \br{S(a_{ij})_\lambda b_{kl}}=&\br{S(a)_{ij\,\lambda}b_{kl}}
=\sum_{n \in \Z}(S(a)_nb)'_{kj} (S(a)_nb)''_{il}\lambda^n \\
=&\sum_{n \in \Z}\lambda^{-1}(a_nb)'_{kj} (a_nb)''_{il}\lambda^n 
=\lambda^{-1}\br{a_{ij\,\lambda} b_{kl}}\,, 
\end{aligned}
\end{equation*}
where we used \eqref{Eq:DMA1} for the third equality. The second identity in \eqref{Eq:RepSesq} is checked in the same way. 

If we have a double multiplicative Poisson vertex algebra, we use \eqref{Eq:DMA2} in the form 
\begin{equation*}
  \sum_{n \in \Z} (a_nb)' \otimes (a_nb)''\lambda^n
= - \sum_{n \in \Z} \lambda^{-n}S^{-n}\left((b_na)'' \otimes (b_na)'\right)
\end{equation*}
to get that 
\begin{equation*}
\begin{aligned}
  \br{a_{ij\,\lambda} b_{kl}}=&\sum_{n \in \Z} (a_nb)'_{kj} (a_nb)''_{il}\lambda^n
= - \sum_{n \in \Z} \lambda^{-n}S^{-n}\left((b_na)''_{kj} (b_na)'_{il}\right) \\
=& - \sum_{n \in \Z} \lambda^{-n}S^{-n}\left((b_na)'_{il}(b_na)''_{kj} \right) 
=-\restriction{\Big(}{x=S}\br{b_{kl\,\lambda^{-1}x^{-1}}a_{ij}}\Big)\,,
\end{aligned}
\end{equation*}
which gives that $\br{-_\lambda-}$ defined by \eqref{Eq:relMPVA} satisfies the skewsymmetry property \eqref{Eq:MA2}. Then, we can conclude because Jacobi identity \eqref{Eq:MA3} holds by Lemma \ref{Lem:RelJac} whenever \eqref{Eq:DMA3} does. 
\end{proof}

For the next lemma, we introduce some notations to go from $\VV^{\otimes 3}$ to $\VV_N$. 
For any $A= a'\otimes a''\otimes a'''\in \VV^{\otimes 3}$, and $1\leq i,j,k,l,m,n \leq N$, we define 
\begin{equation*}
  A_{ij,kl,mn}:=a'_{ij}\, a''_{kl}\, a'''_{mn}\,.
\end{equation*}
We extend this operation in the obvious way to associate $A_{ij,kl,mn}\in \VV_N[\lambda^{\pm 1},\mu^{\pm1}]$ to any $A \in \VV^{\otimes 3}[\lambda^{\pm 1},\mu^{\pm1}]$.

\begin{lemma} \label{Lem:RelJac}
  For any $a,b,c \in \VV$ and $1\leq i,j,k,l,u,v\leq N$, if $\dgal{-_\lambda-}$ is a double $\lambda$-bracket such that  \eqref{Eq:DMA2} holds, then we have that 
\begin{equation*}
  \br{a_{ij\,\lambda}\br{b_{kl\,\mu}c_{uv}}}-\br{b_{kl\,\lambda}\br{a_{ij\,\mu}c_{uv}}}
-\br{\br{a_{ij\,\lambda}b_{kl}}_{\lambda \mu}c_{uv}}
=\dgal{a_\lambda b_\mu c}_{uj,il,kv} - \dgal{b_\mu a_\lambda c}_{ul,kj,iv}\,,
\end{equation*}
where 
$\br{-_\lambda-}$ is defined by \eqref{Eq:relMPVA}, while  
$\dgal{-_{\lambda}-_{\mu}-}$ is given by \eqref{Eq:Triple}. 
\end{lemma}
\begin{proof}
  Using the Leibniz rules for a multiplicative $\lambda$-bracket and the definition \eqref{Eq:relMPVA}, we have 
\begin{equation*}
\begin{aligned}
 &  \br{a_{ij\,\lambda}\br{b_{kl\,\mu}c_{uv}}}
=\sum_{q \in \Z} \br{a_{ij\,\lambda} (b_q c)'_{ul}(b_q c)''_{kv}} \mu^q \\
=& \sum_{p,q \in \Z} \Big[ (a_p (b_q c)')'_{uj} (a_p (b_q c)')''_{il}(b_q c)''_{kv}
 + (a_p (b_q c)'')'_{kj} (a_p (b_q c)'')''_{iv}(b_q c)'_{ul} \Big] \lambda^p \mu^q\,.
\end{aligned}
\end{equation*}
Similarly, $-\br{b_{kl\,\mu}\br{a_{ij\,\lambda}c_{uv}}} $ can be written as 
\begin{equation*}
\begin{aligned}
\sum_{p,q \in \Z} \Big[ - (b_q (a_p c)')'_{ul} (b_q (a_p c)')''_{kj}(a_p c)''_{iv}
-(b_q (a_p c)'')'_{il} (b_q (a_p c)'')''_{kv}(a_p c)'_{uj}\Big] \lambda^p \mu^q\,,
\end{aligned}
\end{equation*}
while $-\br{\br{a_{ij\,\lambda}b_{kl}}_{\lambda \mu}c_{uv}} $ can be written as 
\begin{equation*}
\begin{aligned}
 \sum_{p,q \in \Z} \Big[ - ((a_p b)'_q c)'_{uj} ((a_p b)'_q c)''_{kv} S^q((a_p b)''_{il})
- ((a_p b)''_q c)'_{ul} ((a_p b)''_q c)''_{iv} S^q((a_p b)'_{kj})
\Big] \lambda^p \mu^q\,,
\end{aligned}
\end{equation*}
so that we get six terms for the left-hand side. Meanwhile, we can use \eqref{Eq:Triple} to get $\dgal{a_\lambda b_\mu c}$, and we can write 
\begin{equation*}
  \begin{aligned}
    \dgal{a_\lambda b_\mu c}_{uj,il,kv}=&
\sum_{p,q \in \Z} (a_p (b_qc)')'_{uj} (a_p (b_qc)')''_{il} (b_q c)''_{kv}\lambda^p \mu^q \\
&-\sum_{p,q \in \Z}(a_pc)'_{uj}(b_q (a_p c)'')'_{il}(b_q (a_p c)'')''_{kv}\lambda^p \mu^q\\
&-\sum_{p,q \in \Z} 
((a_pb)'_q c)'_{uj} S^q((a_p b)'')_{il} ((a_pb)'_qc)_{kv}\lambda^{p+q} \mu^q\,.
  \end{aligned}
\end{equation*}
If we use \eqref{Eq:Triple2} to expand $\dgal{b_\mu a_\lambda  c}$, we can write 
\begin{equation*}
  \begin{aligned}
- \dgal{b_\mu a_\lambda c}_{ul,kj,iv}=&
-\sum_{p,q \in \Z} (b_q (a_pc)')'_{ul} (b_q (a_p c)')''_{kj} (a_p c)''_{iv}
\lambda^p \mu^q \\
&+\sum_{p,q \in \Z}  (b_qc)'_{ul}(a_p(b_qc)'')'_{kj}(a_p(b_qc)'')''_{iv}
 \lambda^p \mu^q\\
&-\sum_{p,q \in \Z} ((a_pb)''_qc)'_{ul}S^q((a_pb)')_{kj}((a_pb)''_qc)''_{iv}
\lambda^{p+q} \mu^q\,.
  \end{aligned}
\end{equation*}
It now suffices to see that the left- and right-hand sides coincide since $S(d_{mn})=S(d)_{mn}$ for any $d \in A$ and indices $1\leq m,n \leq N$. 
\end{proof}

\begin{example} \label{Ex:Rep-1}
Using $\VV$ from Example \ref{Ex:FreeA}, we get for $N \geq 1$ the representation algebra $\VV_N=\kk[u_{m,ij} \mid m\in \Z,\, 1\leq i,j\leq N]$ which is a multiplicative Poisson vertex algebra by Theorem \ref{Thm:MainRep}. The automorphism on $\VV_N$ is given by $S(u_{m,ij})=u_{m+1,ij}$, and the multiplicative $\lambda$-bracket satisfies 
$$
    \br{u_{m,ij}{}_\lambda u_{n,kl}}=\lambda^{n-m} \big(u_{n,kj} \delta_{il} - u_{n,il} \delta_{kj} \big)\,.
$$
\end{example}

\begin{example}
Combining  Proposition \ref{Pr:class1} with $\VV=\mc R_1$ and Theorem \ref{Thm:MainRep}, we get for $N \geq 1$ that the representation algebra $\VV_N$  (as in Example \ref{Ex:Rep-1}) is a multiplicative Poisson vertex algebra for the multiplicative $\lambda$-bracket 
$$
    \br{u_{ij}{}_\lambda u_{kl}}=(u u_M)_{kj} (u_M u)_{il} \lambda^M - (u u_{-M})_{kj} (u_{-M} u)_{il} \lambda^{-M}\,,
$$
where $u_{ij}:=u_{0,ij}$ and $M\geq 1$. 
For $M=N=1$, if we set $\mathtt{u}:=u_{11}$ we get that the \emph{commutative} polynomial algebra in one variable  $\kk[\mathtt{u}_m\mid m\in \Z]$ is equipped with the following  multiplicative $\lambda$-bracket  
\begin{equation} \label{Eq:Rep-Volt2}
    \br{\mathtt{u} {}_\lambda \mathtt{u}} =\mathtt{u}^2 \mathtt{u}_1^2 \lambda - \mathtt{u}^2 \mathtt{u}_{-1}^2 \lambda^{-1}\,.
\end{equation}
This can be seen as the ``square'' of the $\lambda$-bracket for the Volterra lattice \cite{DSKVWclass} given on $\kk[\mathtt{v}_s\mid s\in \Z]$ by 
$$\br{\mathtt{v}_\lambda \mathtt{v}}=\mathtt{v}\mathtt{v}_1 \lambda - \mathtt{v}\mathtt{v}_{-1}\lambda^{-1}.$$ 
Indeed,  we recover \eqref{Eq:Rep-Volt2} for $\mathtt{u}=\mathtt{v}^2$ up to a factor.
\end{example}

\begin{corollary} \label{Cor:CommuteRep}
The (non-commutative) correspondence between local lattice double Poisson algebras and double multiplicative Poisson vertex algebras from Proposition  \ref{prop:20210506} induces 
the (commutative) correspondence between local lattice Poisson algebras and multiplicative Poisson vertex algebras from Proposition  \ref{Prop:Corresp} on representation spaces. 
\end{corollary}
\begin{proof}
Fix $a,b\in \VV$ for $\VV$ a local lattice double Poisson algebra with double Poisson bracket $\dgal{-,-}$. By Proposition  \ref{prop:20210506}, $\VV$ is a double multiplicative Poisson vertex algebra  
with multiplicative $\lambda$-bracket given in \eqref{eq:Corr-2}. 

Using Van den Bergh's work \cite{VdB1}, \eqref{Eq:dbrRep} defines a Poisson bracket on $\VV_N$. It is easy to check that $\VV_N$ is a local lattice Poisson algebra by inducing $S$ from $\VV$ to $\VV_N$ as in \S\ref{sss:RepAlg}. 
Alternatively, we can use Theorem \ref{Thm:MainRep} to get a multiplicative $\lambda$-bracket on $\VV_N$ as follows. Recalling the first equation in \eqref{Eq:relMPVA}, we have
\begin{align*}
    \br{a_{ij}{}_\lambda b_{kl}}=& 
\sum_{n\in\mb Z}\lambda^n(a_nb)'_{kj}(a_nb)''_{il}\stackrel{\eqref{Eq:DPtoDPV}}{=}
\sum_{n\in \Z} \lambda^n \dgal{S^n(a),b}'_{kj} \dgal{S^n(a),b}''_{il} \\
=&\sum_{n\in \Z} \lambda^n \br{S^n(a)_{ij},b_{kl}} 
=\sum_{n\in \Z} \lambda^n \br{S^n(a_{ij}),b_{kl}}\,,
\end{align*}
in agreement with \eqref{eq:Corr-2}. 

Starting from the double multiplicative Poisson vertex algebra instead, we get in a similar fashion that the induced Poisson bracket on $\VV_N$ satisfies 
\begin{align*}
    \br{a_{ij},b_{kl}} = \dgal{a,b}'_{kj} \dgal{a,b}''_{il} 
    \stackrel{\eqref{eq:2.15}}{=}
    (a_0b)'_{kj}(a_0b)''_{il}
    =\mres_\lambda \sum_{n\in\mb Z}(a_nb)'_{kj}(a_nb)''_{il}\lambda^n
    =\mres_\lambda \br{a_{ij}{}_\lambda b_{kl}}\,.
\end{align*}
This is consistent with \eqref{eq:Corr-1}.
\end{proof}
Corollary \ref{Cor:CommuteRep} can be summarised in terms of the commutative diagram depicted in Figure \ref{Fig:A}. 
\begin{figure}[h]
\centering
  \begin{tikzpicture}
 \node  (dPA) at (-3,1.5) {$(\VV,\dgal{-,-},S)$};
 \node  (mdPA) at (3,1.5) {$(\VV,\dgal{- _{\lambda}-},S)$};
  \node  (PA) at (-3,-1.5) {$(\VV_N,\br{-,-},S)$};
 \node  (mPA) at (3,-1.5) {$(\VV_N,\br{- _{\lambda}-},S)$};
\path[<->,>=angle 90,font=\small]  
   (dPA) edge node[above] {Prop. \ref{prop:20210506}}  (mdPA) ;
   \path[<->,>=angle 90,font=\small] 
   (PA) edge node[above] {Prop. \ref{Prop:Corresp}}  (mPA) ;
\path[->,>=angle 90]  
   (dPA) edge node[left] {\eqref{Eq:dbrRep}}  (PA) ;
\path[->,>=angle 90]  
   (mdPA) edge node[right] {\eqref{Eq:relMPVA}}  (mPA) ;
   \end{tikzpicture}
   \caption{}
   \label{Fig:A}
\end{figure}


\section{Connection to Integrable Systems}
\label{S:Integr}

It is shown in \cite{DSKV} that double Poisson vertex algebras provide a convenient framework to study non-commutative partial differential equations. In this section we provide``multiplicative versions''
of some of the results in \cite{DSKV} aimed at showing that double multiplicative Poisson vertex algebras provide a convenient framework
to study non-commutative differential-difference equations. Several examples are presented in Section \ref{ss:IS-ex}.

\subsection{The trace map and connection to Lie algebras}\label{S:Integr2}
Let $\mc V$ be a unital associative algebra endowed with an automorphism $S$
of infinite order. We denote by $\mult:\mc V\otimes\mc V\to\mc V$ the multiplication map.
We denote also by $[\mc V,\mc V]\subset\mc V$ the commutator subspace.
For $f\in\mc V$, we let $\tr(f)\in\mc V/[\mc V,\mc V]$ be the corresponding coset.
Furthermore, for $f\in\mc V$, we also let
$\tint f$ be the coset of $f\in\mc V$ in the quotient space
$$\mc F:=\mc V/([\mc V,\mc V]+(S-1)\mc V)\,.$$

Given a double multiplicative $\lambda$-bracket $\ldb-_\lambda-\rdb$ on $\mc V$ 
we define the following map
$\{-_{\lambda}-\}:\mc V\otimes \mc V\to \mc V[\lambda,\lambda^{-1}]$
by
\begin{equation}\label{20140707:eq4}
\{a_\lambda b\}=\mult\ldb a_{\lambda}b\rdb
\end{equation}
(the multiplication map
on $\mc V\otimes\mc V$ is extended to a multiplication map 
$\mult:\,(\mc V\otimes\mc V)[\lambda,\lambda^{-1}]\to\mc V[\lambda,\lambda^{-1}]$ 
in the obvious way),
and we also define the map
$\{-,-\}:\mc V\otimes \mc V\to \mc V$
by
\begin{equation}\label{20140707:eq5}
\{a,b\}=\mult\ldb a_{\lambda}b\rdb|_{\lambda=1}
\,.
\end{equation}

\begin{lemma}\label{prop:1bis}
If $\ldb-_\lambda-\rdb$ is a skewsymmetric double multiplicative $\lambda$-bracket on $\mc V$,
then the following identity
holds in $\mc V^{\otimes2}[\lambda^{\pm1},\mu^{\pm1}]$ ($a,b,c\in\mc V$):
\begin{align}
\begin{split}\label{eq:quasijacobi}
&\{a_\lambda\ldb b_\mu c\rdb\}
-\ldb b_\mu\{a_\lambda c\}\rdb
-\ldb\{a_\lambda b\}_{\lambda\mu} c\rdb
=(\mult\otimes1)\ldb a_\lambda b_\mu c\rdb
-(1\otimes \mult)\ldb b_\mu a_\lambda c\rdb\,.
\end{split}
\end{align}
where we set $\{a_\lambda b\otimes c\}=\{a_\lambda b\}\otimes c+b\otimes\{a_\lambda c\}$.
\end{lemma}
\begin{proof}
Let us compute the three terms on the LHS of \eqref{eq:quasijacobi}.
We get after a straightforward computation 
$$
\begin{array}{l}
\displaystyle{
\vphantom{\Big(}
\{a_\lambda\ldb b_\mu c\rdb\}
=(\mult\otimes1)\ldb a_{\lambda}\ldb b_\mu c\rdb\rdb_L
+(1\otimes \mult)\ldb a_{\lambda}\ldb b_{\mu}c\rdb\rdb_R
\,,} \\
\displaystyle{
\vphantom{\Big(}
\ldb b_\mu\{a_\lambda c\}\rdb
=(1\otimes \mult)\ldb b_{\mu}\ldb a_\lambda c\rdb\rdb_L
+(\mult\otimes1)\ldb b_{\mu}\ldb a_{\lambda}c\rdb\rdb_R
\,,} \\
\displaystyle{
\vphantom{\Big(}
\ldb \{a_{\lambda}b\}_{\lambda\mu}c\rdb
=(\mult\otimes1)\ldb \ldb a_\lambda b\rdb_{\lambda\mu}c\rdb_L
+(1\otimes \mult)\ldb\ldb a_\lambda b\rdb_{\lambda\mu}c\rdb_R
\,.}
\end{array}
$$
By skewsymmetry and Lemma \ref{Lem:Trick}, 
we can replace the last term in the RHS of the third equation
by $-(1\otimes \mult)\ldb\ldb b_\mu a\rdb_{\lambda\mu}c\rdb_L$.
Combining the three equations above, we then get \eqref{eq:quasijacobi}.
\end{proof}

\begin{theorem}\label{20140707:thm}
Let $\mc V$ be a unital associative algebra with an automorphism $S$, endowed with a
double multiplicative $\lambda$-bracket $\ldb-_{\lambda}-\rdb$. Let
$\{-_{\lambda}-\}$ and $\{-,-\}$ be defined as in \eqref{20140707:eq4} and \eqref{20140707:eq5}.
\begin{enumerate}[(a)]
\item
$S[\mc V,\mc V]\subset[\mc V,\mc V]$.
Hence, we have a well defined induced map 
(denoted, by abuse of notation, by the same symbol)
$S:\mc V/[\mc V,\mc V]\to\mc V/[\mc V,\mc V]$, given by
$S(\tr f)=\tr(S f)$.
\item
$\{[\mc V,\mc V]_\lambda\mc V\}=0$,
and $\{\mc V_\lambda{[\mc V,\mc V]}\}\subset[\mc V,\mc V]\otimes\kk[\lambda^{\pm1}]$.
Hence, we have well defined induced maps
(denoted, by abuse of notation, by the same symbol)
$$
\{-_\lambda-\}:\,\mc V/[\mc V,\mc V]\times\mc V\to\mc V[\lambda^{\pm1}]
\,,
$$
and
$$
\{-_\lambda-\}:\,\mc V/[\mc V,\mc V]\times\mc V/[\mc V,\mc V]
\to\mc V/[\mc V,\mc V][\lambda^{\pm1}]
\,,
$$
given, respectively, by
\begin{equation}\label{20140707:eq3b}
\{\tr(f)_\lambda g\}=\mult\ldb f_\lambda g\rdb
\,,
\end{equation}
and 
\begin{equation}\label{20140707:eq3c}
\{\tr(f)_\lambda \tr(g)\}=\tr(\mult\ldb f_\lambda g\rdb)
\,.
\end{equation}
\item
If the double multiplicative $\lambda$-bracket is skewsymmetric,
then so is the $\lambda$-bracket \eqref{20140707:eq3c}:
$$
\{\tr(f)_\lambda\tr(g)\}=-\{\tr(g)_{\lambda^{-1}S^{-1}}\tr(f)\}\,.
$$
\item
If the double multiplicative $\lambda$-bracket defines a structure of a double multiplicative Poisson vertex algebra on $\mc V$,
then the multiplicative $\lambda$-bracket \eqref{20140707:eq3c} satisfies the Jacobi identity 
($f,g,h\in\mc V$)
$$
\{\tr(f)_\lambda\{\tr(g)_\mu \tr(h)\}\}-\{\tr(g)_\mu\{\tr(f)_\lambda \tr(h)\}\}
=
\{\{\tr(f)_\lambda \tr(g)\}_{\lambda\mu} \tr(h)\}\,.
$$
Hence, \eqref{20140707:eq3c} endows  $\mc V/[\mc V,\mc V]$ with a structure of multiplicative Lie conformal algebra (cf. Remark \ref{Rem:mLCA}).
Furthermore, the $\lambda$-action \eqref{20140707:eq3b} of $\mc V/[\mc V,\mc V]$ on $\mc V$ defines a representation of the multiplicative Lie conformal algebra
$\mc V/[\mc V,\mc V]$ given by conformal derivations of $\mc V$.
\item
$\{[\mc V,\mc V]+(S-1)\mc V,\mc V\}=0$,
and $\{\mc V,{[\mc V,\mc V]+(S-1)\mc V}\}\subset([\mc V,\mc V]+(S-1)\mc V)\otimes\kk[\lambda^{\pm1}]$.
Thus, we have well defined induced brackets
(denoted, by abuse of notation, by the same symbol)
$$
\{-,-\}:\,\mc F\times\mc V\to\mc V
\,,
$$
and
$$
\{-,-\}:\,\mc F\times \mc F\to\mc F
\,,
$$
given, respectively, by
\begin{equation}\label{20140707:eq3b-lie}
\{\tint f, g\}:=\mult\ldb f_\lambda g\rdb|_{\lambda=1}
\,,
\end{equation}
and 
\begin{equation}\label{20140707:eq3c-lie}
\{\tint f,\tint g\}:=\tint \mult\ldb f_\lambda g\rdb|_{\lambda=1}
\,.
\end{equation}
\item
If the double multiplicative $\lambda$-bracket is skewsymmetric,
then so is the bracket \eqref{20140707:eq3c-lie}.
\item
If the double multiplicative $\lambda$-bracket defines a structure of a double multiplicative Poisson vertex algebra on $\mc V$,
then the bracket \eqref{20140707:eq3c-lie} defines a
structure of a Lie algebra on $\mc F$.
Furthermore, the action of $\mc F$ on $\mc V$,
given by \eqref{20140707:eq3b-lie}, defines a representation of the Lie  algebra
$\mc F$ by derivations of $\mc V$
commuting with $S$.
\end{enumerate}
\end{theorem}
\begin{proof}
Part (a) is straightforward, since $S$ is an automorphism of $\mc V$.
Using the right Leibniz rule we have, for $a,b,c\in\mc V$,
$$
\{ab_\lambda c\}
=
\mult\ldb ab_\lambda c\rdb
=
\mult\Big(
(|_{x=S}a)\otimes_1\ldb b_{\lambda x} c\rdb
+\ldb a_{\lambda x}c\rdb\otimes_1 (|_{x=S}b)
\Big)
\,.
$$
The expression in parenthesis in the RHS above is
unchanged if we switch $a$ and $b$
(since $x\otimes_1(y'\otimes y'')=y'\otimes x\otimes y'=(y'\otimes y'')\otimes_1 x$), so that we have $\{ab_\lambda c\}=\{ba_\lambda c\}$.
Furthermore, we can compute that 
$$
\{a_\lambda bc\}
=
\mult\ldb a_\lambda bc\rdb
=
\mult\big(\ldb a_\lambda b\rdb c+ b\ldb a_\lambda c\rdb\big)
=
\{a_\lambda b\}c+b\{a_\lambda c\}\,.
$$
Namely, the $\lambda$-action $\{a_\lambda\,-\}$ 
is by derivations of the associative product of $\mc V$.
This yields 
$$
\{a_\lambda bc-cb\}
=
[b,\{a_\lambda c\}]+[\{a_\lambda b\},c]\,\in[\mc V,\mc V]\otimes\kk[\lambda^{\pm1}]
\,,
$$
which finishes the proof of part (b).
Part (c) is immediate.
Part (d) is a direct consequence of Lemma \ref{prop:1bis}.
Finally, parts (e), (f) and (g)
can be proven in the same way. Alternatively, one can use  the standard construction that associates
with a multiplicative Lie conformal algebra $V$
the corresponding Lie algebra $V/(S-1)V$
and its representation on $V$.
\end{proof}
\begin{remark}
Lemma \ref{prop:1bis} and Theorem \ref{20140707:thm} are multiplicative versions of \cite[Lemma 3.5]{DSKV} and \cite[Theorem 3.6]{DSKV}, respectively.
\end{remark}

\subsection{Evolution differential-difference equations and Hamiltonian differential-difference equations}
\label{ss:EvoDDE}
Let $\mc V$ be an algebra of difference functions in the variable $u_i$, $i\in I=\{1,\dots,\ell\}$, see Definition \ref{def:diff-func}. An \emph{evolution differential-difference equation} over $\mc V$ has the form
\begin{equation}\label{evol-eq-var}
\frac{du_i}{dt}
= P_i\,\in\mc V
\,\,,\,\,\,\,
i\in I
\,.
\end{equation}
Assuming that time derivative commutes with the automorphism $S$, we have
$\frac{du_{i,n}}{dt}=S^n( P_i)$,
and, by the chain rule,
a function $f\in\mc V$ evolves according to
\begin{equation}\label{evol-eq2-var}
\frac{df}{dt}
= \sum_{(i,n)\in I\times\mb Z}\mult\Big((S^nP_i)\ast_1\frac{\partial f}{\partial u_{i,n}}\Big)
= X_P(f)
\,.
\end{equation}

We call $\mc F=\mc V/([\mc V,\mc V]+(S-1)\mc V)$ the space of local functionals.
An \emph{integral of motion} is a local functional $\tint f\in\mc F$
constant in time:
\begin{equation}\label{integral-of-motion-var}
\frac{d\tint f}{dt}
= \int \sum_{(i,n)\in I\times\mb Z}\mult\Big((S^nP_i)\ast_1\frac{\partial f}{\partial u_{i,n}}\Big)
= 0
\,.
\end{equation}
We call \emph{vector field} on $\mc V$ any derivation $X:\,\mc V\to\mc V$
of the form
\begin{equation}\label{20140704:eq1}
X(f)=\sum_{(i,n)\in I\times\mb Z} \mult \Big(P_{i,n}\ast_1\frac{\partial f}{\partial u_{i,n}}\Big)
\,,
\end{equation}
where $P_{i,n}\in\mc V$ for all $i,n$.
Note that the RHS of \eqref{20140704:eq1} is a finite sum because  
$\frac{\partial f}{\partial u_{i,n}}=0$
for all but finitely many choices of indices $(i,n)$, 
as $\VV$ is an algebra of difference functions.

An \emph{evolutionary vector field}
is a vector field commuting with the automorphism $S$.
Gathering \eqref{eq:comm} and  \eqref{evol-eq2-var}, it must have the form 
\begin{equation}\label{20140704:eq2}
X_P(f)=\sum_{(i,n)\in I\times\mb Z} \mult \Big((S^n P_i)\ast_1\frac{\partial f}{\partial u_{i,n}}\Big)
\,,
\end{equation}
for $P=(P_i)_{i=1}^\ell\in\mc V^\ell$, called the \emph{characteristics} of the evolutionary vector field $X_P$.

Vector fields form a Lie algebra, 
and evolutionary vector fields form a Lie subalgebra, which we denote respectively  by $\Vect(\mc V)$ and  $\Vect^S(\mc V)$.
The Lie bracket of two evolutionary vector fields $X_P,X_Q\in\Vect^S(\mc V)$ takes
the usual form
$$
[X_P,X_Q]=X_{[P,Q]}\,,
\quad\text{where }[P,Q]_i=X_P(Q_i)-X_Q(P_i)\,, i\in I\,.
$$
Equation \eqref{evol-eq-var} is called \emph{compatible} with another evolution differential-difference equation
$\frac{du_i}{d\tau}=Q_i$, $i\in I$,
if the corresponding evolutionary vector fields commute.

More generally, let $\mc V$ be a double multiplicative Poisson vertex algebra.
The \emph{Hamiltonian equation} on $\mc V$,
associated with  the \emph{Hamiltonian functional} $\tint h\in\mc F$ 
is
\begin{equation}\label{ham-eq}
\frac{du}{dt}
=
\{\tint h,u\}
\,,
\end{equation}
for any $u\in\mc V$.
An \emph{integral of motion} for the Hamiltonian equation \eqref{ham-eq}
is a local functional $\tint f\in\mc F$ such that 
$\{\tint h,\tint f\}=0$.
In this case, by Theorem \ref{20140707:thm}(g), 
the evolutionary vector fields $X^h=\{\tint h,-\}$ and $X^f=\{\tint f,-\}$ 
(called Hamiltonian vector fields) commute,
hence  equations $\frac{du}{dt}=\{\tint f,u\}$ and \eqref{ham-eq} are compatible.
If we are given a family of independent local functionals $\{\tint f_k\mid k\in\Z_+ \}$ whose evolutionary Hamiltonian vector fields $\{X^{f_k}\mid k\in \Z_+\}$ are pairwise commuting, we will say that the corresponding Hamiltonian equations defined through  \eqref{ham-eq} form an \emph{integrable hierarchy} of Hamiltonian  differential-difference equations.

Let us now assume that $\mc V$ is an algebra of difference functions in the variables
$u_i$, $i\in I$,
and that the double multiplicative $\lambda$-bracket $\ldb-_\lambda-\rdb$ on $\mc V$ is given by 
a Poisson structure $H(S)$ via \eqref{master-infinite},
where $\ldb {u_i}_\lambda{u_j}\rdb=H_{ji}(\lambda)$.  
The Hamiltonian equation \eqref{ham-eq} becomes the following evolution equation
\begin{equation}\label{20140702:eq1-var}
\frac{du_i}{dt}
=
\mult\sum_{j=1}^\ell
H_{ij}(S)\bullet \Big(\frac{\delta h}{\delta u_j}\Big)^\sigma
\,,
\end{equation}
where we introduce the \emph{difference variational derivative}  $\frac{\delta h}{\delta u_j}\in\mc V\otimes\mc V$ of $h$ by 
\begin{equation}\label{var-der}
\frac{\delta h}{\delta u_j}
=
\sum_{n\in\mb Z_+}S^{-n}\frac{\partial h}{\partial u_{j,n}}
\,.
\end{equation}
In equation \eqref{20140702:eq1-var}
$S$ is moved to the right of the $\bullet$ product, acting on $\frac{\delta h}{\delta u_j}$.
Moreover,  the Lie bracket $\{-,-\}$ on $\mc F$ defined by 
\eqref{20140707:eq3c-lie},
becomes such that for all $f,g\in \mc V$:
\begin{equation}\label{20140709:eq1b}
\{\tint f,\tint g\}
=\int\sum_{i,j\in I}
\mult\left(\frac{\delta g}{\delta u_j}\right)^\sigma
\mult\left(H_{ji}(S)\ast_1\mult\left(\frac{\delta f}{\delta u_i}\right)^\sigma
\right)
\,.
\end{equation}
Then, the notions of compatibility and of integrals of motion are consistent with those
for general evolution differential-difference equations, due to Theorem \ref{20130921:prop1}.
\begin{remark}\label{20140709:rem1b}
For $H(S)\in\Mat_{\ell\times\ell}(\mc V\otimes \mc V)[S]$
and $F\in \mc V^{\oplus\ell}$, let $H(S)F\in \mc V^\ell$
be defined by 
\begin{equation}\label{20140709:eq2b}
(H(S)F)_i=\sum_{j\in I} \mult(H_{ij}(S)\ast_1 F_j)
=\sum_{j\in I,n\in\mb Z} H_{ij,n}^\prime (S^nF_j) H_{ij,n}^{\prime\prime}
\,. 
\end{equation}
(Here, we used the notation 
$H_{ij}(S)=\sum_{n\in\mb Z}(H_{ij,n}^\prime\otimes H_{ij,n}^{\prime\prime})S^n$.) 
Then, formula \eqref{20140709:eq1b} can be written in the more traditional form
$$
\{\tint f,\tint g\}
=\int
\delta g\cdot \left(H(S)\delta f\right)
\,,
$$
where 
$(\delta f)_i=\mult\left(\frac{\delta f}{\delta u_i}\right)^\sigma$ and  $\cdot$ denotes the usual dot product of vectors.
The latter notation is compatible with the theory of the  variational complex developed in \S\ref{sec:3.5}, 
cf. \eqref{20140626:eq4a-aff}.
\end{remark}

\begin{remark} \label{Rem:EqRepSpace}
Many of the results derived so far are compatible with their commutative analogues \cite{DSKVW1,DSKVWclass} when we go from $\VV$ to the $N$-th representation algebra $\VV_N$, $N\geq 1$ (see Section \ref{S:Integr}). The role of $\VV/[\VV,\VV]$ is played by 
$$\VV_N^{\tr}:=\{\tr \mc X(a)=\sum_{1\leq i \leq N} a_{ii} \mid a\in \VV  \}$$
where we denote by $\mc X(a)= (a_{ij})$ the matrix-valued function representing the element $a\in \VV$. Similarly, we have to replace $\mc F$ by  
$$\mc F_N:= \VV_N^{\tr}/( \, (S-1)\VV_N^{\tr})\,.$$
We note two important such results (assuming that $\VV$ is a double multiplicative Poisson vertex algebra). First, \eqref{20140707:eq3b-lie} induces a representation of the Lie algebra $\mc F_N$ on $\VV$ by derivations commuting with $S$ through 
$$
\br{\tint \tr \mc X(f), \mc X(g)}=\mc X(\mult \dgal{f_\lambda g}|_{\lambda=1}) \,.
$$
(The Lie bracket on $\mc F_N$ is obtained by projecting this identity to $\mc F_N\times \VV_N^{\tr}$ then ${\mc F_N}\times {\mc F_N}$, in agreement with \eqref{20140707:eq3c-lie}.) 
Second, a Hamiltonian functional $\tint h \in \mc F$ gives rise to such a functional $\tint \tr \mc X(h) \in \mc F_N$, and \eqref{ham-eq} induces the following Hamiltonian equation at the level of the representation algebra $\VV_N$: $$
    \frac{d u_{ij}}{dt}= \br{\tint \tr \mc X(h), u_{ij}}
    = (\mult \dgal{h_\lambda u}|_{\lambda=1})_{ij}\,, 
$$
for all $u\in \VV$ and $1\leq i,j \leq N$. In particular, an integral of motion $\tint f\in \mc F$ for $\tint h$ induces the integral of motion $\tint \tr \mc X(f)\in \mc F_N$ for $\tint \tr \mc X(h)$, and a (``non-commutative") integrable hierarchy on $\VV$ as defined above induces a non-abelian (i.e. matrix-valued) integrable hierarchy on $\VV_N$ in the usual sense.  
Applying this point of view to the different examples gathered in \S\ref{ss:IS-ex} gives non-abelian integrable hierarchies of differential-difference equations. 
\end{remark}

\subsection{de Rham complex over an algebra of difference functions}\label{sec:3.5}

Let $\mc V$ be an algebra of difference functions.
The \emph{de Rham complex} $\widetilde{\Omega}(\mc V)$ of $\mc V$ is defined 
as the free product of the algebra $\mc V$ and 
the algebra $\kk\langle\delta u_{i,n}\,|\,i\in I=\{1,\dots,\ell\},\,n\in\mb Z\rangle$
of non-commutative polynomials in the variables $\delta u_{i,n}$.
The action of the automorphism $S$ is extended from $\mc V$ to
$\widetilde{\Omega}(\mc V)$
by letting $S(\delta u_{i,n})=\delta u_{i,n+1}$ for all $(i,n)\in I\times\mb Z$.

The algebra $\widetilde{\Omega}(\mc V)$ has a $\mb Z_+$-grading, denoted by $p$,
such that $f\in\mc V$ has degree $p(f)=0$ and the $\delta u_{i,n}$'s have degree $p(\delta u_{i,n})=1$.
We consider $\widetilde{\Omega}(\mc V)$ as a superalgebra,
with superstructure compatible with the $\mb Z_+$-grading.
Then, the subspace of elements of degree $k$, denoted  $\widetilde{\Omega}^k(\mc V)$,   consists of linear combinations
of terms of the form
\begin{equation}\label{20140623:eq1-var}
\widetilde\omega
=
f_1\delta u_{i_1,m_1}f_2\delta u_{i_2,m_2}\dots f_k\delta u_{i_k,m_k}f_{k+1}
\,\,,\,\text{ where }\,
f_1,\dots,f_{k+1}\in\mc V
\,.
\end{equation}
Note that $\widetilde{\Omega}^0(\mc V)=\mc V$
and $\widetilde{\Omega}^1(\mc V)\simeq\oplus_{(i,n)\in I\times\mb Z}\mc V \delta u_{i,n}\mc V$.

We turn $\widetilde{\Omega}(\mc V)$ into a differential algebra by considering the \emph{de Rham differential} $\delta$ on $\widetilde{\Omega}(\mc V)$ defined 
as the odd derivation of degree $1$
on the superalgebra $\widetilde{\Omega}(\mc V)$
satisfying 
\begin{equation}\label{20140623:eq2-var}
\delta f
=
\sum_{(i,n)\in I\times\mb Z}
\Big(\frac{\partial f}{\partial u_{i,n}}\Big)^\prime
\delta u_{i,n}
\Big(\frac{\partial f}{\partial u_{i,n}}\Big)^{\prime\prime}
\in\widetilde{\Omega}^1(\mc V)
\,\text{ for }\,
f\in\mc V
\,,\,\text{ and }
\delta (\delta u_{i,n})=0
\,.
\end{equation}
The proof that $\delta$ is a differential, i.e. $\delta^2=0$, is a direct computation (it is the same as for the algebra of differential functions, cf. \cite[\S2.7]{DSKV}).
Therefore we can consider the corresponding cohomology complex 
$(\widetilde{\Omega}(\mc V),\delta)$.

Given a vector field 
$X_P=\sum_{(i,n)\in I\times\mb Z}
\mult\circ\Big(P_{i,n}\ast_1\frac{\partial}{\partial u_{i,n}}\Big)\in\Vect(\mc V)$ 
(cf. \eqref{20140704:eq1}),
we define the associated \emph{Lie derivative} 
$L_P:\,\widetilde{\Omega}(\mc V)\to\widetilde{\Omega}(\mc V)$
as the even derivation of degree $0$ which extends $X_P$ from $\mc V$,
in such a way that $L_P(\delta u_{i,n})=\delta P_{i,n}$, $i\in I,\,n\in\mb Z$.
We can also define the associated  \emph{contraction operator} 
$\iota_P:\,\widetilde{\Omega}(\mc V)\to\widetilde{\Omega}(\mc V)$
as the odd derivation of degree $-1$ given on generators
by $\iota_P(f)=0$, for $f\in \mc V$, and $\iota_P(\delta u_{i,n})=P_{i,n}$.
In analogy with \cite[Proposition 2.17]{DSKV}, we remark
that $\widetilde{\Omega}(\mc V)$ is a $\Vect(\mc V)$-complex, which means that the following results hold in the multiplicative setting as well.

\begin{proposition}
Fix $P,Q\in \VV^{I\times \Z}$. Under the identifications $P\leftrightarrow X_P$ and $Q\leftrightarrow X_Q$, we have:
\begin{enumerate}
    \item[(a)] $[\iota_P,\iota_Q]=0$ ; 
    \item[(b)] $[L_P,\iota_Q]=\iota_{[P,Q]}$ ; 
    \item[(c)] $[L_P,L_Q]=L_{[P,Q]}$ ; 
    \item[(d)] $L_P=\iota_P \delta + \delta \iota_P$ (Cartan's formula). 
\end{enumerate}
\end{proposition}
\begin{proof}
These are equalities of derivations of the superalgebra $\widetilde{\Omega}(\mc V)$, so they only need to be checked on generators. For example, let us establish (c). 
We start by noting that $L_P \delta - \delta L_P=0$ because it is an equality of derivations that is easily checked on the generators $u_{i,n},\delta u_{i,n}$. Thus, 
using the identification from the statement, the left-hand side of (c) satisfies 
\begin{align*}
    [L_P,L_Q](f)=&L_P(X_Q(f))-L_Q(X_P)(f)=[X_P,X_Q](f)\,, \quad \forall f\in \mc V\,, \\
[L_P,L_Q](\delta u_{i,n})=& 
\delta([L_P,L_Q] (u_{i,n}) ) 
= \delta ([X_P,X_Q](u_{i,n}))\,.
\end{align*}
Meanwhile, we get for the right-hand side 
\begin{align*}
    L_{[P,Q]}(f)=&[X_P,X_Q](f)\,, \quad \forall f\in \mc V\,, \\
L_{[P,Q]}(\delta u_{i,n})=& \delta(L_{[P,Q]}(u_{i,n})) 
= \delta ([X_P,X_Q](u_{i,n}))\,. \qedhere
\end{align*}
\end{proof}

Our next step is to construct a reduction of the complex  $(\widetilde{\Omega}(\mc V),\delta)$. 

\begin{proposition}\label{20140704:prop}
In the de Rham complex $(\widetilde{\Omega}(\mc V),\delta)$
we have:
\begin{enumerate}[(a)]
\item
The commutator subspace 
$[\widetilde{\Omega}(\mc V),\widetilde{\Omega}(\mc V)]$
is compatible with the $\mb Z_+$-grading and is preserved by $\delta$.
\item
$\delta$ and $S$ commute,
therefore $(S-1)\widetilde{\Omega}(\mc V)$
is compatible with the $\mb Z_+$-grading and is preserved by $\delta$.
\item
Given an evolutionary vector field $X_P$
of characteristics $P=(P_i)_{i=1}^\ell$ (cf. \eqref{20140704:eq2}), 
The associated Lie derivative $L_P$ and contraction operator $\iota_P$ 
commute with the action of $S$ on $\widetilde{\Omega}(\mc V)$.
\end{enumerate}
\end{proposition}
\begin{proof}
The proof is analog to the proof of Proposition 3.15 in \cite{DSKV}. Part (a) follows immediately since $\delta$ is a derivation of the associative product 
on $\widetilde{\Omega}(\mc V)$.
Part (b) is proven if we can show that  
$\delta(S\widetilde{\omega})=S(\delta\widetilde{\omega})$
for every $\widetilde{\omega}\in\widetilde{\Omega}(\mc V)$.
Given $\tilde\omega_1,\tilde\omega_2\in\widetilde{\Omega}(\mc V)$, it is easy to check that
$$
[\delta,S](\tilde\omega_1\tilde\omega_2)=[\delta,S](\tilde\omega_1)S(\tilde\omega_2)
+(-1)^{p(\tilde\omega_1)}S(\tilde\omega_1)[\delta,S](\tilde\omega_2)
\,.
$$
Hence, to prove the claim it suffices to check that $[\delta,S]$ is zero on $\delta u_{i,n}$
and $f\in\mc V$. The identity $[\delta,S](\delta u_{i,n})=0$ is obvious from the second
equation in \eqref{20140623:eq2-var} and the action of $S$ on $\widetilde{\Omega}(\mc V)$.
On the other hand, using the first identity in \eqref{20140623:eq2-var} we have
$$
S(\delta f)
=
\sum_{(i,n)\in I\times\mb Z}
\Big(S\frac{\partial f}{\partial u_{i,n}}\Big)^\prime
\delta u_{i,n+1}
\Big(S\frac{\partial f}{\partial u_{i,n}}\Big)^{\prime\prime}
$$
and
$$
\delta(S f)
=
\sum_{(i,n)\in I\times\mb Z}
\Big(\frac{\partial (S f)}{\partial u_{i,n}}\Big)^\prime
\delta u_{i,n}
\Big(\frac{\partial(S f)}{\partial u_{i,n}}\Big)^{\prime\prime}
\,.
$$
Hence, $[\delta,S](f)=0$ by \eqref{eq:comm}.
Part (c) can be proven similarly.
\end{proof}

Thanks to Proposition \ref{20140704:prop}(a-b),
we can form the $\mb Z_+$-graded \emph{variational complex}
\begin{equation}\label{20140623:eq6a-b}
\Omega(\mc V)
=\widetilde{\Omega}(\mc V)/
\big((S-1)\widetilde{\Omega}(\mc V)+[\widetilde{\Omega}(\mc V),\widetilde{\Omega}(\mc V)]\big)
=\oplus_{n\in\mb Z_+}\Omega^n(\mc V)
\,,
\end{equation}
which is equipped with a differential induced by $\delta$.
Using Proposition \ref{20140704:prop}(c), 
the Lie derivatives $L_P$ and contraction operators $\iota_P$, 
associated with the evolutionary vector field $X_P$ of characteristics $P\in\mc V^\ell$, 
descend to  well defined maps on the variational complex $\Omega(\mc V)$.

\begin{example}
For $\mc R_\ell$ as in \S\ref{ss:AlgRl}, the total degree vector field $X_\Delta$, with characteristics $\Delta=(u_i)_{i=1}^\ell$,
is an evolutionary vector field on $\mc R_\ell$. By adapting \cite[Theorem 2.18]{DSKV}, we can show that the contraction operator $\iota_\Delta$ associated with $X_\Delta$ is a homotopy operator  for the complex $(\widetilde{\Omega}(\mc R_\ell),\delta)$, hence it is acyclic:
$H^n(\widetilde{\Omega}(\mc R_\ell),\delta)=\delta_{n,0}\kk$. In the exact same way, we can see that the complex $(\Omega(\mc R_\ell),\delta)$ is acyclic as well,
i.e.
\begin{equation}\label{20140623:eq6b-aff}
H^k(\Omega(\mc R_\ell),\delta)=\delta_{k,0}\kk
\,.
\end{equation}
\end{example}

Next, we give an explicit description of the complex $(\Omega(\mc V),\delta)$ by adapting \cite{DSK-coh1,DSK-coh2,DSKV}.
It is clear that $\Omega^0(\mc V)=\mc F$, the space of local functionals.
For $k\geq1$, let us introduce the space $\Sigma^k(\mc V)$
of arrays
$\big(A_{i_1\dots i_k}(\lambda_1,\dots,\lambda_{k-1})\big)_{i_1,\dots,i_k=1}^\ell$
with entries 
$$A_{i_1\dots i_k}(\lambda_1,\dots,\lambda_{k-1})\in \mc V^{\otimes k}[\lambda_1^{\pm1},\dots,\lambda_{k-1}^{\pm1}],$$
satisfying the following skewadjointness condition ($i_1,\dots,i_k\in I$):
\begin{equation}\label{20140626:eq1-b}
A_{i_1\dots i_k}(\lambda_1,\dots,\lambda_{k-1})
=
-(-1)^{k}|_{x=S}\big(
A_{i_2\dots i_k i_1}(\lambda_2,\dots,\lambda_{k-1},
(\lambda_1\dots\lambda_{k-1}x)^{-1})
\big)^\sigma
\,,
\end{equation}
where $\sigma$ denotes the action of the cyclic permutation on $\mc V^{\otimes k}$ 
as in \eqref{eq:sigma},
and we are using the same notation as in \eqref{notation-bar}.
We claim that there is an isomorphism $\Omega^k(\mc V)\simeq\Sigma^k(\mc V)$, which we prove by writing explicitly the maps in both directions.

Fix a coset $\omega=[\widetilde{\omega}]\in\Omega^k(\mc V)$,
where $\widetilde{\omega}$ is as in \eqref{20140623:eq1-var}. 
We map $\omega$ to the array
$A=\big(A_{j_1\dots j_k}(\lambda_1,\dots,\lambda_{k-1})\big)_{i_1,\dots,i_k=1}^\ell\in\Sigma^k(\mc V)$, 
with entries 
$A_{j_1\dots j_k}(\lambda_1,\dots,\lambda_{k-1})=0$
unless $(j_1,\dots,j_k)$ is a cyclic permutation of $(i_1,\dots,i_k)$,
and
\begin{equation}\label{20140626:eq2-aff}
\begin{array}{l}
\displaystyle{
\vphantom{\Big)}
A_{j_1\dots j_k}(\lambda_1,\dots,\lambda_{k-1})
=
\frac1k (-1)^{s(k-s)}
\lambda_1^{n_{s+1}}\dots\lambda_{k-s}^{n_k}
\lambda_{k-s+1}^{n_1}\dots\lambda_{k-1}^{n_{s-1}}
} \\
\displaystyle{
\vphantom{\Big)}
(\lambda_1\dots\lambda_{k-1}S)^{-n_s}
\big(
f_{s+1}\otimes\dots\otimes f_k\otimes f_{k+1}f_1\otimes f_2\otimes\dots\otimes f_s
\big)
\,,}
\end{array}
\end{equation}
for $(j_1,\dots,j_k)=(i_{\sigma^s(1)},\dots,i_{\sigma^s(k)})$.
The inverse map $\Sigma^k(V)\to\Omega^k(V)$
is given by
\begin{equation}\label{20140626:eq3-aff}
\begin{array}{l}
\displaystyle{
\Big(
\sum_{n_1,\dots,n_{k-1}\in\mb Z}A_{i_1\dots i_k}^{n_1\dots n_{k-1}}
\lambda_1^{n_1}\dots\lambda_{k-1}^{n_{k-1}}
\Big)_{i_1,\dots,i_k=1}^\ell
\mapsto
} \\
\displaystyle{
\sum_{\substack{i_1,\dots,i_k\in I \\ n_1,\dots,n_{k-1}\in\mb Z}}
\!\!\!\!\!
\big[
(A_{i_1\dots i_k}^{n_1\dots n_{k-1}})^\prime \delta u_{i_1,n_1}
\dots
(A_{i_1\dots i_k}^{n_1 \dots n_{k-1}})^{\overbrace{\prime\dots\prime}^{k-1}} \delta u_{i_{k-1},n_{k-1}}
(A_{i_1\dots i_k}^{n_1 \dots n_{k-1}})^{\overbrace{\prime\dots\prime}^{k}} \delta u_{i_k}
\big]
\,.}
\end{array}
\end{equation}
(Here, we use Sweedler's notation.)
It is not hard to verify that the maps \eqref{20140626:eq2-aff} and \eqref{20140626:eq3-aff} are well defined and inverse to each other. 
Hence the space of degree $k$ elements in the variational complex $\Omega^k(\mc V)$ and the space of arrays $\Sigma^k(\mc V)$
can be identified using these maps.

We can explicitly translate the differential $\delta$ of the variational complex $\Omega(\mc V)$ to a differential $\delta:\,\Sigma^k(\mc V)\to\Sigma^{k+1}(\mc V)$
under the above identification.
For $k=0$, we have
\begin{equation}\label{20140626:eq4a-aff}
\delta(\tint f)
=
\Big(\sum_{n\in\mb Z}
S^{-n}\mult \Big(\frac{\partial f}{\partial u_{i,n}}\Big)^\sigma \Big)_{i=1}^\ell
=\Big(
\mult \Big(\frac{\delta f}{\delta u_i}\Big)^\sigma \Big)_{i=1}^\ell
\,,
\end{equation}
where in the second identity we used equation \eqref{var-der}.
More generally, if $k\geq1$ and  $A=
\big(A_{i_1\dots i_k}(\lambda_1,\dots,\lambda_{k-1})\big)_{i_1,\dots,i_k=1}^\ell\in\Sigma^k(\mc V)$, we have that 
\begin{equation}\label{20140626:eq4b2-aff}
\begin{array}{l}
\displaystyle{
(\delta A)_{i_1\dots i_{k+1}}(\lambda_1,\dots,\lambda_k)
} \\
\displaystyle{
=
\frac{k}{k+1}\sum_{n\in\mb Z}
\bigg(
\sum_{s=1}^{k} (-1)^{s+1}
\Big(
\frac{\partial}{\partial u_{i_{s},n}}
\Big)_{(s)}
A_{i_{1}
\stackrel{s}{\check{\dots}}
i_{k+1}}
(\lambda_{1},
\stackrel{s}{\check{\dots}},
\lambda_{k})
\lambda_{s}^{n}
} \\
\displaystyle{
+
(-1)^{k}
(\lambda_1\dots\lambda_k S)^{-n}
\Big(
\Big(
\frac{\partial}{\partial u_{i_{k+1},n}}
\Big)_{(1)}
A_{i_1,\dots,i_k}
(\lambda_{1},\dots,\lambda_{k-1})
\Big)^{\sigma^k}
\bigg)
\,.}
\end{array}
\end{equation}
The notation $\stackrel{t}{\check{\dots}}$ means that we skip the the object in position $t$, $\sigma$ denotes the action
of the cyclic permutation in \eqref{eq:sigma}, 
and we use the extended derivations  $(\partial/\partial u_{i,n})_{(s)}:\VV^{\otimes k}\mapsto \VV^{\otimes (k+1)}$ defined through \eqref{Eq:DerExt}.

Equation \eqref{20140626:eq4b2-aff}, for $k=1$ and $F=\big(F_j\big)_{j=1}^\ell\in\mc V^\ell=\Sigma^1(\mc V)$, gives
\begin{equation}\label{20140626:eq4c-aff}
(\delta F)_{ij}(\lambda)
=
\frac12
\sum_{n\in\mb Z}
\Big(
\frac{\partial F_j}{\partial u_{i,n}} \lambda^n
- (\lambda S)^{-n} \Big(\frac{\partial F_i}{\partial u_{j,n}}\Big)^\sigma
\Big)
\,.
\end{equation}
For $F\in \mc V^{\oplus\ell}=\Sigma^1(\mc V)$, define the corresponding \emph{Frechet derivative}
$$
D_F(\lambda)
=
\left(
\sum_{n\in\mb Z}\frac{\partial F_i}{\partial u_{j,n}} \lambda^n
\right)_{i,j=1}^{\ell}
\in\Mat_{\ell\times\ell}(\mc   V \otimes\mc V)[\lambda,\lambda^{-1}]
\,.
$$
From \eqref{20140626:eq4c-aff} we see that $\delta F=0$ if and only if
$D_F(S)$ is a selfadjoint non-commutative difference operator.

For $k=2$, let
$A=\big(A_{ij}(\lambda)\big)_{i,j=1}^\ell\in\Sigma^2(\mc V)$,
i.e. the entries $A_{ij}(\lambda)\in\mc V^{\otimes2}[\lambda,\lambda^{-1}]$
 satisfy $(|_{x=S}A_{ji}(\lambda^{-1}S^{-1}))^\sigma=-A_{ij}(\lambda)$.
Equation \eqref{20140626:eq4b2-aff} gives
\begin{equation}\label{20140626:eq4d-aff}
\begin{array}{l}
\displaystyle{
(\delta A)_{ijk}(\lambda,\mu)
=
\frac{2}{3}\sum_{n\in\mb Z}
\bigg(
\Big(
\frac{\partial}{\partial u_{i,n}}
\Big)_L
A_{jk}
(\mu)
\lambda^{n}
} \\
\displaystyle{
-
\Big(
\frac{\partial}{\partial u_{j,n}}
\Big)_R
A_{ik}
(\lambda)
\mu^{n}
+
(\lambda\mu S)^{-n}
\Big(
\Big(
\frac{\partial}{\partial u_{k,n}}
\Big)_L
A_{ij}
(\lambda)
\Big)^{\sigma^2}
\bigg)
\,.}
\end{array}
\end{equation}
As an application of \eqref{20140623:eq6b-aff}, we get the following result
which is a multiplicative version of \cite[Corollary 3.17]{DSKV}.
\begin{corollary}\label{victor:cor2}
\begin{enumerate}[(a)]
\item
A $0$-form $\tint f\in\Omega^0(\mc R_\ell)$ is closed if and only if 
$f\in\kk+[\mc R_\ell,\mc R_\ell]+(S-1)\mc R_\ell$.
\item
A $1$-form $F=\big(F_i\big)_{i=1}^\ell\in\mc R_\ell^{\oplus\ell}=\Sigma^1(\mc R_\ell)$ is closed
if and only if
there exists a local functional $\tint f\in \mc R_\ell/([\mc R_\ell,\mc R_\ell]+(S-1)\mc R_\ell)$
such that $F_i=\mult\Big(\frac{\delta f}{\delta u_i}\Big)^\sigma$
for every $i=1,\dots,\ell$.
\item
A $2$-form 
$\alpha=\big(A_{ij}(\lambda)\big)_{i,j=1}^\ell\in\Sigma^2(\mc R_\ell)$
is closed if and only if
there exists $F=\big(F_i\big)_{i=1}^\ell\in \mc R_\ell^{\oplus\ell}$
such that 
$$
A_{ij}(\lambda)
=
\frac12
\sum_{n\in\mb Z}
\Big(
\frac{\partial F_j}{\partial u_{i,n}} \lambda^n
- (\lambda S)^{-n} \Big(\frac{\partial F_i}{\partial u_{j,n}}\Big)^\sigma
\Big)
\,,
$$
for every $i,j=1,\dots,\ell$.
\end{enumerate}
\end{corollary}

\subsection{Examples} \label{ss:IS-ex}

\subsubsection{Integrable hierarchies on $\mc{R}_1$} 
Recall from Example \ref{Ex:FreeA} that $\mc{R}_1=\kk\langle u_i\mid i\in \Z \rangle$ is a double multiplicative Poisson vertex algebra for 
\begin{equation*}
    \dgal{u_i{}_\lambda u_j}=\lambda^{j-i} (u_j\otimes 1 - 1 \otimes u_j)\,.
\end{equation*}
It is easy to check that the local functionals $\{\tint u^k\mid k \in \Z_+\}$ satisfy $\{\tint u^k,\tint u^l\}=0$ for any $k,l\in \Z_+$, therefore the Hamiltonian vector fields that they define are pairwise commuting. Note however that 
$$
    \frac{du}{dt_k}:=\{\tint u^k,u\}=0\,, \quad k\in \Z_+\,,
$$
so these Hamiltonian vector fields do not yield non-trivial differential-difference equations. The same observation can be made if we use the double multiplicative $\lambda$-bracket obtained by combining Proposition \ref{prop:20210506} with any double Poisson bracket from Example \ref{Ex:Powell} (the case above corresponds to $\alpha=1$, $\beta=\gamma=0$). 
We have been unable to construct an integrable hierarchy on $\mc{R}_1$ starting with one of the cases from the classification given in Proposition \ref{Pr:class1}. To get non-trivial examples on  $\mc{R}_1$, it seems necessary to use non-local double multiplicative Poisson vertex algebras as defined in Section \ref{S:nonloc}, see \cite{CW1,CW2}.  

\subsubsection{Integrable hierarchies on $\mc{R}_2$}  \label{ss:IntR2}
As part of Proposition \ref{Pr:class2}, we have constructed double multiplicative Poisson vertex algebra structures on $\mc{R}_2=\kk\langle u_i,v_i\mid i\in \Z \rangle$ such that 
\begin{equation*}
    \dgal{u_\lambda u}=0\,, \quad \dgal{v_\lambda v}=0\,.
\end{equation*}
In particular, the local functionals $\{\tint u^k\mid k \in \Z_+\}$ satisfy $\{\tint u^k,\tint u^l\}=0$ trivially, hence they yield commuting Hamiltonian vector fields on $\mc{R}_2$. A similar result holds for the set of functionals $\{\tint v^k\mid k\in \Z_+\}$ by symmetry.

\begin{example} \label{Exmp:Simple}
Consider $\dgal{u_\lambda v}=1\otimes 1$, which corresponds to case (i) of Theorem \ref{Thm:2-explicit}. We easily compute that for the vector field $d/dt_k:=\frac{1}{k}\{\tint u^k,-\}$, $k\geq1$, we have  
\begin{equation} \label{Eq:IntR2-a}
    \frac{dv}{dt_k}=\frac{1}{k}\{\tint u^k,v\}=\frac{1}{k}\mult\dgal{u^k{}_\lambda v}\big|_{\lambda=1}= u^{k-1}\,, \quad 
 \frac{du}{dt_k}=0\,.
\end{equation}
Since $d/dt_k$ commutes with $S$ by part (g) of Theorem \ref{20140707:thm}, note that \eqref{Eq:IntR2-a} is equivalent to the Hamiltonian differential-difference equations 
$$
    \frac{dv_i}{dt_k}=u_i^{k-1}\,, \quad \frac{du_i}{dt_k}=0\,, \quad i\in \Z\,. 
$$
Since the vector fields $(d/dt_k)_{k\geq1}$ are pairwise commuting for different $k\in \Z_+$ due to $\{\tint u^k,\tint u^l\}=0$, we get in this way an integrable hierarchy of differential-difference equations. 
The solution to the $k$-th system of equations is simply given by $u_i(t_k)=\alpha_i$, $v_i(t_k)=\beta_i+t_k \alpha_i^{k-1}$ for $i\in \Z$, where $\alpha_i,\beta_i\in\kk$.
Compatibility of the solution with $S$ implies that  $\alpha_i=\alpha_0$, $\beta_i=\beta_0$ for each $i\in \Z$.
\end{example}
\begin{remark}
While we observed in Remark \ref{Rem:EqRepSpace} that differential-difference equations on an associative algebra $\VV$ induce such equations on the representation algebra $\VV_N$, $N\geq 1$, solving the equation on $\VV$ does not provide all the solutions on $\VV_N$. 
Combining Remark \ref{Rem:EqRepSpace} and Example \ref{Exmp:Simple}, we see that \eqref{Eq:IntR2-a} induces the non-abelian equation 
\begin{equation} \label{Eq:IntR2-a-rep}
    \frac{d \mc X(v)}{dt_k}= \mc X(u)^{k-1}\,, \quad 
 \frac{d\mc X(u)}{dt_k}=0\,, 
\end{equation}
while its solution $u(t_k)=\alpha_0$, $v(t_k)=\beta_0 + t_k \alpha_0^{k-1}$, leads to 
$$\mc X(u)(t_k)=\alpha_0 \Id_N, \quad \mc X(v)(t_k)=\beta_0 \Id_N + t_k \alpha_0^{k-1} \Id_N\,.$$
However, an arbitrary solution of \eqref{Eq:IntR2-a-rep} is of the form 
$\mc X(u)(t_k)=A_0$, $\mc X(v)(t_k)=B_0 + t_k A_0^{k-1}$ for $A_0,B_0\in \Mat_{n\times n}(\kk)$.
\end{remark}

\begin{example}
Fix $r\in \Z$, $\alpha \in\kk^\times$ and take 
$$\dgal{u_\lambda v}=\left(\alpha \,vu_r\otimes u_r v+ vu_r\otimes 1+1\otimes u_rv+\alpha^{-1} 1\otimes 1\right) \lambda^r\,,$$ 
corresponding to case (v) of Theorem \ref{Thm:2-explicit}. The Hamiltonian vector fields $\frac{d}{dt_k}=\frac1k \{\tint u_k,-\}$ are commuting and they define the following differential-difference equations: 
\begin{equation} \label{Eq:Int-1}
\frac{dv}{dt_k}=\alpha\, vu_r^{k+1}v+v u_r^k+u_r^k v+ \alpha^{-1} u_r^{k-1}\,, \quad 
\frac{du}{dt_k} =0\,.    
\end{equation}
\end{example}

\begin{example}
Fix $r\in \Z$, $\alpha \in\kk^\times$ and take as in Example \ref{Exmp:2-quad} 
$$\dgal{u_\lambda v}=\left(\alpha \,v\otimes v +v\otimes u_r + u_r\otimes v+\alpha^{-1} u_r\otimes u_r\right) \lambda^r\,,$$ 
that corresponds to case (iv) of Theorem \ref{Thm:2-explicit}. The Hamiltonian vector fields $\frac{d}{dt_k}=\frac1k \{\tint u_k,-\}$ are commuting and they define the following differential-difference equations: 
\begin{equation} \label{Eq:Int-2}
\frac{dv}{dt_k}=\alpha v u_r^{k-1}v + (vu_r^k+u_r^k v)+ \alpha^{-1} u_r^{k+1}\,, \quad 
\frac{du}{dt_k} =0\,.    
\end{equation}
Note that if we allow each $u_i$ to be invertible by working in $\kk\langle u_i^{\pm 1},v_i\mid i\in \Z \rangle$, we have commuting Hamiltonian vector fields $\frac{d}{dt_k}$ for any $k\in \Z$ (not only for $k\in \Z_+$). In particular, remark that the differential-difference equation \eqref{Eq:Int-2} defined for  $\frac{d}{dt_{-k}}$ can be transformed into \eqref{Eq:Int-1} if we relabel $u_r\leftrightarrow u_r^{-1}$.
\end{example}

\subsubsection{Integrable hierarchies using a weak version of Jacobi identity}

As a slight generalisation of \eqref{Eq:Int-2}, it can be checked that the vector fields defining the differential-difference equations 
\begin{equation} \label{Eq:Int-3}
\frac{dv}{dt_k}=\alpha v u_r^{k-1}v + (vu_r^k+u_r^k v)+ \beta u_r^{k+1}\,, \quad 
\frac{du}{dt_k} =0\,,    \qquad k\in \Z_+\,,
\end{equation}
commute for any fixed $r\in \Z$ and $\alpha,\beta\in \kk$. By considering the equations associated with the local functionals $\{\tint u^k\mid k\in\Z_+\}$ for all the cases from Theorem \ref{Thm:2-explicit}, we can see that \eqref{Eq:Int-3} can not always be obtained from a double multiplicative Poisson vertex algebra structure on $\mc R_2$. It can, nevertheless, be obtained from a double multiplicative $\lambda$-bracket (see Example \ref{Ex:NotJacobi})
using the framework that we introduce in this paragraph.

From now on, we consider a skewsymmetric double multiplicative $\lambda$-bracket $\dgal{-_\lambda-}$ on a unital associative algebra $\VV$ with an automorphism $S$. Recall that $\br{-_\lambda-}:\VV\otimes \VV\to\VV[\lambda^{\pm1}]$ denotes the associated map \eqref{20140707:eq4} obtained from $\dgal{-_\lambda-}$ by multiplication of the two factors. In the same way, if $\dgal{-_\lambda-_\mu-}$ is the map \eqref{Eq:Triple} defined from  $\dgal{-_\lambda-}$, we introduce 
\begin{equation} \label{Eq:mTriple}
    \br{-_\lambda-_\mu-}:\VV^{\otimes 3}\to \VV[\lambda^{\pm1},\mu^{\pm1}]\,, \quad 
    \br{a_\lambda b_\mu c}=\mult \circ (\mult\otimes 1) \dgal{a_\lambda b_\mu c}\,.
\end{equation}

\begin{lemma}
For any $a,b,c\in \VV$, 
\begin{equation} \label{eq:m-quasijacobi}
\br{a_\lambda \br{b_\mu c}} -\br{b_\mu \br{a_\lambda c}}  - \br{\br{a_\lambda b}_{\lambda\mu}c}=
\br{a_\lambda b_\mu c} - \br{b_\mu a_\lambda c}\,.
\end{equation}
\end{lemma}
\begin{proof}
It follows from \eqref{eq:quasijacobi} by applying the multiplication map $\mult$.
\end{proof}
Recall from \S\ref{ss:EvoDDE}, that for the local functional $\tint f\in\mc F$,
we denote by $X^f$ its associated Hamiltonian vector field. Recall also that $\mc F$ is
a Lie algebra with Lie bracket \eqref{20140707:eq3c-lie} satisfying
$\{\tint f,\tint g\}=\tint \{f,g\}$, where $f,g\in\mc V$, and $\{f,g\}$ is given by \eqref{20140707:eq5}.
\begin{lemma} \label{Lem:ComVF}
Let $\tint f,\tint h\in\mc F$. Then $[X^f,X^h]=X^{\{f,h\}}$ if and only if the derivation  
\begin{equation} \label{Eq:Dfh}
    \mathtt{D}_{f,h}:=\br{f_\lambda h_\mu -}\big|_{\lambda=\mu=1} - \br{h_\mu f_\lambda -}\big|_{\lambda=\mu=1}
\end{equation}
vanishes identically. 
\end{lemma}
\begin{proof}
By \eqref{eq:m-quasijacobi}, we have for any $c\in \VV$ that 
\begin{align*}
[X^f,X^h](c)=&
\left( \br{f_\lambda \br{h_\mu c}} -\br{h_\mu \br{f_\lambda c}}\right) \big|_{\lambda=\mu=1} \\
=&\left(\br{\br{f_\lambda h}_{\lambda\mu}c}+  \br{f_\lambda h_\mu -} - \br{h_\mu f_\lambda -} \right) \big|_{\lambda=\mu=1}  \\
=&\br{\tint \br{f_\lambda h} \big|_{\lambda=1} ,c} + \mathtt{D}_{f,h}(c)
= X^{\{f,h\}}(c)+\mathtt{D}_{f,h}(c)\,.
\end{align*}
Thus $[X^f,X^h](c)=X^{\{f,h\}}(c)$ if and only if $\mathtt{D}_{f,h}(c)$, and we can conclude as $c\in \VV$ is arbitrary. The fact that $\mathtt{D}_{f,h}$ is a derivation easily follows from \eqref{Eq:TriDer}.  
\end{proof}

As a consequence of Lemma \ref{Lem:ComVF}, the two Hamiltonian vector fields $X^f,X^h$ commute whenever $\tint \br{f,h}=\br{\tint f,\tint h}=0$ and $\mathtt{D}_{f,h}=0$ for $\mathtt{D}_{f,h}$ defined through \eqref{Eq:Dfh}. 
In the case of a a double multiplicative Poisson vertex algebra, the operation  $\dgal{-_\lambda-_\mu-}$ \eqref{Eq:Triple} is identically vanishing, so for any $f,h\in \VV$ we have $\mathtt{D}_{f,h}=0$. Therefore we have $[X^f,X^h]=0$ whenever $\tint \br{f,h}=0$, as we already observed in \S\ref{ss:EvoDDE}. 

Building on the observation that we have just made, 
we can seek to construct commuting families of vector fields in the presence of a skewsymmetric double multiplicative $\lambda$-bracket that has some failure to satisfy Jacobi identity, i.e. when $\dgal{-_\lambda-_\mu-}$ is non-zero. 
This weaker notion has been identified in \cite{CW2} on $\VV=\R\langle u_{i,n}\mid i\in I, n\in \Z\rangle$ by studying bivector fields, see \cite[\S4.4]{CW2}. The main examples outside the class of double multiplicative Poisson vertex algebras that they investigated \cite[Sect.6]{CW2} are the double \emph{quasi-Poisson} brackets due to Van den Bergh \cite{VdB1}. Such double brackets have the property that  $\dgal{-_\lambda-_\mu-}\neq 0$ has a particular form that entails  the vanishing\footnote{This property was already known by Van den Bergh, see \cite[Proposition 5.1.2]{VdB1}.} of the associated map $\br{-_\lambda-_\mu-}$. Note that these are very special double multiplicative $\lambda$-brackets, because they have image in $\VV^{\otimes 2}$ and not in $\VV^{\otimes 2}[\lambda^{\pm1}]$. 
Below, we provide a new non-trivial example which is \emph{not} independent of $\lambda$.

\begin{example} \label{Ex:NotJacobi}
Fix $\alpha,\beta\in \kk$ and $r\in \Z$. 
Consider the skewsymmetric double multiplicative $\lambda$-bracket on $\mc{R}_2=\kk\langle u_i,v_i\mid i\in \Z \rangle$ given by 
$$\dgal{u_\lambda u}=0=\dgal{v_\lambda v}\,, \quad 
\dgal{u_\lambda v}=(v\otimes u_r + u_r \otimes v + \alpha v\otimes v + \beta u_r \otimes u_r)\lambda^r\,.$$
This operation does not satisfy Jacobi identity when $\alpha\beta\neq 1$ because 
$$\dgal{u_\lambda u_\mu v}= (1-\alpha\beta) (v\otimes u_r\otimes u_r - u_r\otimes u_r\otimes v) \, \lambda^r\mu^r\,.$$
Using this identity and \eqref{Eq:TriDer2}-\eqref{Eq:TriDer3}, we get that for any $M,N\geq 1$,
\begin{align*}
&\dgal{u^M{}_\lambda u^N{}_\mu v}\\
&=
\sum_{m=0}^{M-1}\sum_{n=0}^{N-1}
\left(\big|_{y=S}u^m\right)\ast_1 \left(\big|_{x=S}u^n\right)\ast_2 
\dgal{u_{\lambda y} u_{\mu x} v}
\ast_1 \left(\big|_{x=S}u^{N-n-1}\right) \ast_2 \left(\big|_{y=S}u^{M-m-1}\right)\\
&=\sum_{m=0}^{M-1}\sum_{n=0}^{N-1} 
(1-\alpha\beta)\left[vu_r^{M-m-1}\otimes u_r^{N-n+m}\otimes u_r^{n+1} 
- u_r^{M-m}\otimes u_r^{N-n+m}\otimes u_r^{n}v \right]\, \lambda^r\mu^r\,.
\end{align*}
Therefore, from \eqref{Eq:mTriple} we have $\{u^M{}_\lambda u^N{}_\mu v\}\big|_{\lambda=\mu=1}=(1-\alpha\beta)MN (vu_r^{M+N}-u_{r}^{M+N}v)$.
In particular, this implies that for any $k,l\in\Z_+$ the derivation $\mathtt{D}_{u^k,u^l}:\mc{R}_2\to \mc{R}_2$ defined through \eqref{Eq:Dfh} with $f=u^k,h=u^l$ vanishes on $v$. We trivially have $\mathtt{D}_{u^k,u^l}(u)=0$ as $\br{u_\lambda u}=0$, so that  $\mathtt{D}_{u^k,u^l}=0$ identically. As a consequence of Lemma \ref{Lem:ComVF}, we get that 
$$
    \left[\frac{d}{dt_k},\frac{d}{dt_l} \right]= \frac{1}{kl}\br{\tint \br{u^k,u^l},-}\,, \quad \text{ where }
    \frac{d}{dt_k}:=\frac{1}{k}\br{\tint u^k,-}\,.
$$
As the local functionals  $\{\tint \frac1k u^k\mid k\in \Z_+\}$ are such that $\br{\tint u^k,\tint u^l}=0$, we thus obtain that their Hamiltonian vector fields  $\frac{d}{dt_k}$ pairwise commute. The associated differential-difference equations are given by \eqref{Eq:Int-3}.
\end{example}


\section{Non-local and rational double multiplicative Poisson vertex algebras} 
\label{S:nonloc}
In this section we formalize the theory of non-local and rational double multiplicative Poisson vertex algebras.
They play a crucial role in the context of non-commutative Hamiltonian differential-difference equations,
see \cite{CW1,CW2}.
The exposition follows \cite{DSKVW1} where the commutative case is treated.

\subsection{Non-local double multiplicative Poisson vertex algebras}\label{sec:7.1}
Let $\mc V$ be a unital associative algebra with an automorphism $S$.
We denote by $(\mc V\otimes\mc V)[[\lambda,\lambda^{-1}]]$
the space of bilateral series $\sum_{n\in\mb Z}a_n\lambda^n$, 
where $a_n\in\mc V\otimes\mc V$ for all $n\in\mb Z$.

Non-local double multiplicative $\lambda$-brackets differ from local ones just in replacing
in Definition \ref{def:DMlamBr} the algebra $(\mc V\otimes\mc V)[\lambda,\lambda^{-1}]$ by 
the vector space $(\mc V\otimes\mc V)[[\lambda,\lambda^{-1}]]$.
Note that in the non-local case, despite the fact that $(\mc V\otimes\mc V)[[\lambda,\lambda^{-1}]]$
is not an algebra, all axioms \eqref{Eq:DMA1}, \eqref{Eq:DML}, \eqref{Eq:DMR},
\eqref{Eq:DMA2} and \eqref{Eq:DMA3} still make perfect sense. Hence we have the
following definition.

\begin{definition}\label{def:nonlocal}
A \emph{non-local double multiplicative Poisson vertex algebra} is a unital associative algebra $\mc V$
endowed with an automorphism $S:\,\mc V\to\mc V$
and a non-local double multiplicative $\lambda$-bracket,
$\dgal{-\,_\lambda-}\,:\,\mc V\otimes\mc V\to(\mc V\otimes\mc V)[[\lambda,\lambda^{-1}]]$
satisfying sesquilinearity \eqref{Eq:DMA1}, Leibniz rules \eqref{Eq:DML} and \eqref{Eq:DMR},
skewsymmetry \eqref{Eq:DMA2} and Jacobi identity \eqref{Eq:DMA3}.
\end{definition}
Let $\mc V$ be an algebra of non-commutative difference functions in $\ell$ variables
$u_i$, $i\in I$. We call the space $\Mat_{\ell\times\ell}(\mc V\otimes\mc V)[[S,S^{-1}]]$, the space
of non-local difference operators. Note that this space is not an algebra with respect to the
product \eqref{eq:bullet-op}, and its elements do not act on $\mc V^\ell$ (such an action
would involve divergent sums, cf. \eqref{eq:action-diff}). However, if
$H(S)=(H_{i,j}(S))_{i,j\in I}\in\Mat_{\ell\times\ell}(\mc V\otimes\mc V)[[S,S^{-1}]]$,
then we can define a non-local double multiplicative $\lambda$-bracket on $\mc V$ using 
the Master Formula \eqref{master-infinite} with $\dgal{{u_i}_\lambda u_j}=H_{ji}(\lambda)$,
which makes sense also for bilateral series. One can check that Theorem \ref{20130921:prop1}
still holds in the non-local case.
\begin{theorem}\label{prop:master-nl}
Given an algebra of non-commutative difference functions $\mc V$ in $\ell$ variables $u_i$, $i\in I$,
and an $\ell\times\ell$ matrix 
$H(\lambda)=\big(H_{ij}(\lambda)\big)_{i,j=1}^\ell
\in\Mat_{\ell\times\ell}\mc V[[\lambda,\lambda^{-1}]]$,
the double multiplicative $\lambda$-bracket \eqref{master-infinite}
defines a structure of non-local double multiplicative Poisson vertex algebra on $\mc V$ 
if and only if
skewsymmetry and the Jacobi identity hold on the generators $u_i$.
In this case we call the matrix $H$
a non-local Poisson structure on $\mc V$.
\end{theorem}
\begin{example}\label{Ex:non-local}
We can get examples of non-local double multiplicative Poisson vertex algebras
on an algebra of difference function in one variable $u$,
which generalize the $\lambda$-bracket given by \eqref{20210812:eq1}-\eqref{eq:g}.
Indeed, note that in the proof of Proposition \ref{20210812:prop1}, equation \eqref{20210812:eq2},
which is the Jacobi identity on generators, still holds if we assume
$r(\lambda)=\sum_{n\in\mb Z}r_n\lambda^n\in\kk[[\lambda,\lambda^{-1}]]$ and such that $r(\lambda)=-r(\lambda^{-1})$
in \eqref{20210812:eq1}.
Hence, for example,
the non-local multiplicative $\lambda$-bracket defined on $\mc V$ by letting
$$
\dgal{u_\lambda u}=(u\otimes u)\bullet r(\lambda S)(u\otimes u)
=\sum_{n\in\mb Z}r_n(u u_n\otimes u_nu)\lambda^n\,,
$$
and extended to $\mc V$ by the Master Formula \eqref{master-infinite}, defines a non-local
double multiplicative Poisson vertex algebra structure on $\mc V$.
On the other hand, the proof of the ``only if" part of Proposition \ref{20210812:prop1} does not generalize to the
non-local setting since it heavily relies on the fact that $r_n=0$ for every $n>N$ for some $N\in\mb Z_{>0}$.
\end{example}

\subsection{Pseudodifference operators}\label{sec:pseudodiff}
Let $\mc V$ be a unital associative algebra with an automorphism $S$.
The algebra of non-commutative difference operators $(\mc V\otimes\mc V)[S,S^{-1}]$ defined in
\S\ref{sec:Classif1} is $\mb Z$-graded by the powers of $S$ and can be completed either in the
positive or negative directions, giving rise to two algebras of
non-commutative pseudodifference operators:
$$
(\mc V\otimes\mc V)((S))=(\mc V\otimes\mc V)[[S]][S^{-1}]
\qquad\text{and}\qquad
(\mc V\otimes\mc V)((S^{-1}))=(\mc V\otimes\mc V)[[S^{-1}]][S]\,.
$$
In the sequel we will use the notation $(\mc V\otimes\mc V)((S^{\pm1}))$ to denote $(\mc V\otimes\mc V)((S))$ or $(\mc V\otimes\mc V)((S^{-1}))$ respectively, and it should not be confused with
$(\mc V\otimes\mc V)((S,S^{-1}))$.
Given a non-commutative pseudodifference operator $A(S)=\sum_na_nS^n\in(\mc V\otimes\mc V)((S^{\pm1}))$, its formal adjoint is (cf. \eqref{eq:adjoint})
\begin{equation}\label{eq:adjoint2}
A^*(S)=\sum_{n}S^{-n}\bullet a_n^{\sigma}\in(\mc V\otimes\mc V)((S^{\mp1}))\,,
\end{equation}
and its symbol is (cf. \eqref{eq:symb})
\begin{equation}\label{eq:symb2}
A(z)=\sum_na_nz^n\in(\mc V\otimes\mc V)((z^{\pm1})).
\end{equation}
Formulas \eqref{20210914:eq1} still make sense for non-commutative pseudodifference operators.

\subsection{Pseudodifference operators of rational type}

Let
$$
\kk(z)=\left\{\frac{p(z)}{q(z)}\mid p(z),q(z)\in\kk[z], q(z)\neq0\right\}
$$
denote the field of rational functions in the indeterminate $z$.
It can be embedded in both spaces of Laurent series $\kk((z))$ or $\kk((z^{-1}))$. Indeed, if
$q(z)=\sum_{n=M}^Nq_nz^n\in\kk[z]$, where $0\leq M\leq N$, is non-zero,
then we can factor it as
$$
q(z)=q_Mz^{M}\left(1+\sum_{n=M+1}^N\frac{q_n}{q_M}z^{n-M}\right)
$$
and expand $\frac{1}{q(z)}$, via geometric series expansion, as an element of $\kk((z))$,
or we can factor
$$
q(z)=q_Nz^{N}\left(1+\sum_{n=M}^{N-1}\frac{q_n}{q_N}z^{n-N}\right)
$$
and expand $\frac{1}{q(z)}$, via geometric series expansion, as an element of $\kk((z^{-1}))$. 
We denote by $\iota_{\pm}$ the resulting embedding of the field of rational functions into the space of 
Laurent series
\begin{equation}\label{20210914:eq3}
\iota_{\pm}:\kk(z)\hookrightarrow\kk((z^{\pm1}))
\,.
\end{equation}
For example, we have
\begin{equation}\label{20210914:eq4}
\iota_+\frac{1}{1-z}=\sum_{n\geq0}z^{n}\in\kk((z))
\,,
\qquad
\iota_-\frac{1}{1-z}=-\sum_{n\geq1}z^{-n}\in\kk((z^{-1}))
\,.
\end{equation}
From now on we will work with the algebra of pseudodifference operators
$(\mc V\otimes\mc V)((S))$ but all the definitions and results still hold if we replace it by
$(\mc V\otimes\mc V)((S^{-1}))$ and use $\iota_-$ in place of $\iota_+$.
Let $f_1,\dots f_{n+1}\in\mc V\otimes\mc V$ and $r_1(z),\dots,r_n(z)\in\kk(z)$,
using \eqref{eq:bullet-op} and the embedding $\iota_{+}$ defined in \eqref{20210914:eq3}
we define the following non-commutative pseudodifference operator
\begin{equation}\label{20210914:eq5}
f_1\iota_{+}r_1(S)\bullet f_2\iota_{+}r_2(S)\bullet \dots\bullet f_n\iota_{+}r_{n}(S)\bullet f_{n+1}
\in(\mc V\otimes\mc V)((S))
\,.
\end{equation}
For example, for $f,g\in\mc V\otimes\mc V$, we have, using \eqref{20210914:eq4},
$$
f\iota_+\frac{1}{1-S}\bullet g=\sum_{n\geq0}(f\bullet S^n(g))S^n\in(\mc V\otimes\mc V)((S))\,.
$$
\begin{definition}\label{20210914:def1}
A non-commutative pseudodifference operator of rational type with values in $\mc V$ is a linear combination of non-commutative pseudodifference operators of the form \eqref{20210914:eq5}. We denote by
$\mc Q(\mc V)\subset(\mc V\otimes\mc V)((S))$ the space of non-commutative
pseudodifference operators of rational type.
\end{definition}
It is clear from \eqref{20210914:eq5} and Definition \ref{20210914:def1} that $\mc Q(\mc V)$ is an algebra with respect to the product \eqref{eq:bullet-op}. Given a 
non-commutative pseudodifference operator of rational type
$$
A(S)=\sum f_1\iota_{+}r_1(S)\bullet f_2\iota_{+}r_2(S)\bullet \dots\bullet f_n\iota_{+}r_{n}(S)\bullet f_{n+1}
$$
we define its
adjoint $A^*(S)$ by
\begin{equation}\label{20210914:eq6}
A^*(S)=\sum f_{n+1}^{\sigma}\iota_{+}r_n(S^{-1})\bullet f_{n}^{\sigma}\iota_{+}r_{n-1}(S^{-1})\bullet
\dots\bullet f_{2}^{\sigma}\iota_{+}r_1(S^{-1})\bullet f_1^{\sigma}
\in(\mc V\otimes\mc V)((S))
\,.
\end{equation}
Note that \eqref{20210914:eq6} is an element of $\mc Q(\mc V)$ and does not coincide with the formal adjoint in the space
$(\mc V\otimes\mc V)((S))$ defined in \S\ref{sec:pseudodiff} even though, by an abuse of notation, we are denoting it with the same symbol. In fact, the adjoint of a pseudodifference operator
in $(\mc V\otimes\mc V)((S))$ is an element of $(\mc V\otimes\mc V)((S^{-1}))$,
see \eqref{eq:adjoint2}.

\begin{remark}
Pseudodifference operators of rational type may not be rewritten as the ratio of two difference operators
(that is, an expression of the form $A(S)\bullet B(S)^{-1}$, $A(S),B(S)\in(\mc V\otimes\mc V)[S]$).
Indeed, let us assume that $\mc V$ is a division ring. In general $\mc V\otimes\mc V$ is not a division ring and the
(non-commutative) field
of fractions of $(\mc V\otimes\mc V)[S]$ may not exist.
\end{remark}

\subsection{Rational double multiplicative Poisson vertex algebras}
Let $\mc V$ be a unital associative algebra with an automorphism $S$.
By a rational double multiplicative $\lambda$-bracket on $\mc V$, we mean a double multiplicative 
$\lambda$-bracket as in Definition \ref{def:DMlamBr} with the only difference that we assume 
\begin{equation}\label{20211013:eq1}
\dgal{a_\lambda b}=A_{ab}(\lambda)\in(\mc V\otimes \mc V)((\lambda))
\end{equation}
being the symbol of a pseudodifference operator of rational type $A_{ab}(S)\in\mc Q(\mc V)$, for every $a,b\in\mc V$.

\subsubsection{Definition}
In analogy with the vector space $(\VV\otimes \VV)((\lambda))$, we  introduce for $k\geq 1$
$$\VV^{\otimes k}((\lambda,\mu))
=\Big \{\sum_{m\geq M,n\geq N} a_{m,n} \lambda^m \mu^n \mid 
a_{m,n} \in \VV^{\otimes k}, M,N\in\mb Z\Big\}\,.$$
Let us remark that 
$$\VV^{\otimes k}((\lambda,\mu))=
\VV^{\otimes k}((\lambda))((\mu))\cap 
\VV^{\otimes k}((\mu))((\lambda))\,.$$
The following results will be used throughout this section.
\begin{lemma}\label{20210918:lem1}
Let $T_1,T_2$ be automorphisms of $\mc V$, and let $A(\lambda,\mu)\in(\mc V\otimes\mc V\otimes\mc V)((\lambda,\mu))$
and $B(\lambda,\mu)\in(\mc V\otimes\mc V)((\lambda,\mu))$. Then,
    $$
    A(\lambda T_1,\mu T_2)\bullet_{i}B(\lambda,\mu)\,,
    B(\lambda T_1,\mu T_2)\bullet_{i}A(\lambda,\mu)\in(\mc V\otimes\mc V\otimes\mc V)((\lambda,\mu))\,.
    $$
\end{lemma}
\begin{proof}
Straightforward.
\end{proof}
\begin{lemma}\label{20210918:lem2}
Let $A(S)\in\mc Q(\mc V)$ be a pseudodifference operator of rational type. Then,
$$
\dgal{a_\lambda A(\mu)}_L\,,\dgal{a_\lambda A(\mu)}_R\,,\dgal{A(\lambda)_{\lambda\mu} a}_L
\in(\mc V\otimes\mc V\otimes\mc V)((\lambda,\mu))\,,
$$
for every $a\in\mc V$.
\end{lemma}
\begin{proof}
Recall from Definition \ref{20210914:def1} that $A(S)$ is a finite linear combination
of pseudodifference operators as in \eqref{20210914:eq5}. Hence, it suffices to prove the claim for
$$
A(S)=A_1(S)\bullet\dots\bullet A_n(S)\,,
$$
where $A_i=f_i\iota_+(r(S))$, $f_i\in\mc V\otimes\mc V$, $r_i(S)\in\kk(S)$, $i=1,\dots,n$.
We will prove that $\dgal{a_\lambda A(\mu)}_L\in(\mc V\otimes\mc V\otimes\mc V)((\lambda,\mu))$,
by induction on $n$.

For $n=1$ we have $\dgal{a_\lambda A(\mu)}_L=\dgal{a_\lambda f_1}\iota_+(r_1(\mu))$ which clearly lies in
$(\mc V\otimes\mc V)((\lambda,\mu))$. Let us now assume that $\dgal{a_\lambda A(\mu)}_L\in(\mc V\otimes\mc V\otimes\mc V)((\lambda,\mu))$, and let $f\in\mc V\otimes\mc V$, $r(S)\in\kk (S)$. Then, by Lemma \ref{20210917:lem1} we have
\begin{equation}\label{20210918:eq1}
    \dgal{ a_\lambda f\iota_+(r(\mu S))\bullet A(\mu)}_L=
    \dgal{a_\lambda f}\iota_+(r(\mu S))\bullet_1 A(\mu)
    +f\iota_+(r(\lambda\mu S))\bullet_2\dgal{ a_\lambda A(\mu)}_L
    \,.
\end{equation}
Note that, by the base  case we have that $\dgal{a_\lambda f}\iota_+(r(\mu))\in(\mc V\otimes\mc V\otimes\mc V)((\lambda,\mu))$. Moreover, 
$f\iota_+(r(\lambda\mu))\in(\mc V\otimes\mc V)((\lambda,\mu))$. Hence, the RHS of equation
\eqref{20210917:lem1} lies in $(\mc V\otimes\mc V\otimes\mc V)((\lambda,\mu))$ by Lemma \ref{20210918:lem1} and the inductive assumption.  The other claims can be proven similarly.
\end{proof}

We say that a rational double multiplicative $\lambda$-bracket as in \eqref{20211013:eq1}
is skewsymmetric if
\begin{equation}\label{20211013:eq2}
A_{ab}(\lambda)=-A_{ba}^*(\lambda)\,,
\end{equation}
for every $a,b\in\mc V$ (cf. Remark \ref{rem:skew}). In \eqref{20211013:eq2} we are using the formal adjoint
\eqref{20210914:eq6} in $\mc Q(\mc V)$. Hence, \eqref{20211013:eq2} is equivalent to an identity in the space $\mc Q(\mc V)$.

\begin{definition}\label{rdmPVA}
A rational double multiplicative Poisson vertex algebra is a unital associative algebra $\mc V$ endowed 
with an automorphism $S$ and  a rational double multiplicative $\lambda$-bracket
$\dgal{-_\lambda -}:\mc V\otimes \mc V \to (\mc V\otimes\mc V)((\lambda))$ 
(i.e. for every $a,b\in\mc V$,
$\dgal{a_\lambda b}$ is the symbol of a pseudodifference operator of rational type $A_{ab}(S)\in \mc Q(\mc V)$, cf. \eqref{20211013:eq1})
satisfying skewsymmetry \eqref{20211013:eq2} and Jacobi identity \eqref{Eq:DMA3}.
\end{definition}
Note that, if 
$\dgal{-_\lambda -}$ has values in $(\mc V\otimes\mc V)((\lambda))$ 
then
$$\dgal{a_\lambda\dgal{b_\mu c}}_L\in(\mc V\otimes\mc V\otimes\mc V)((\lambda))((\mu))\,,$$
while
$$\dgal{b_\mu\dgal{a_\mu c}}_R\in(\mc V\otimes\mc V\otimes\mc V)((\mu))((\lambda))\,.$$
Since
these two terms lie in different spaces, the Jacobi identity could not make sense.
However,
for a rational double multiplicative $\lambda$-bracket \eqref{20211013:eq1},
by Lemma \ref{20210918:lem2}, 
all three terms of the Jacobi identity 
lie in $(\mc V\otimes\mc V\otimes\mc V)((\lambda,\mu))$, hence it is well-defined.

Let $\mc V$ be an algebra of non-commutative difference functions in $\ell$ variables
$u_i$, $i\in I$. We call the space $\Mat_{\ell\times\ell}(\mc Q(\mc V))$, the space
of matrix pseudodifference operators of rational type.
If
$H(S)=(H_{i,j}(S))_{i,j\in I}\in\Mat_{\ell\times\ell}(\mc Q(\mc V))$,
then we can define a rational double multiplicative $\lambda$-bracket on $\mc V$ using 
the Master Formula \eqref{master-infinite} with $\dgal{{u_i}_\lambda u_j}=H_{ji}(\lambda)$
and get an analogue of Theorem \ref{20130921:prop1}.
\begin{theorem}\label{prop:master-rat}
Given an algebra of non-commutative difference functions $\mc V$ in $\ell$ variables $u_i$,
and an $\ell\times\ell$ matrix 
$H(\lambda)=\big(H_{ij}(\lambda)\big)_{i,j=1}^\ell
\in\Mat_{\ell\times\ell}(\mc Q(\mc V))$,
the double multiplicative $\lambda$-bracket \eqref{master-infinite}
defines a structure of rational double multiplicative Poisson vertex algebra on $\mc V$ 
if and only if
skewsymmetry and the Jacobi identity hold on the generators $u_i$.
In this case we call the matrix $H$
a Poisson structure of rational type on $\mc V$.
\end{theorem}
\begin{proof}
It suffices to show that $\dgal{f_\lambda g}$ defined by the RHS of
\eqref{master-infinite} is the symbol of a pseudodifference operator of rational type, for every
$f,g\in\mc V$. The rest is the same as in the proof of Theorem \ref{20130921:prop1}.
The claim follows since the RHS of \eqref{master-infinite} is the symbol of a finite sum of products of elements
in $\mc Q(\mc V)$ and $(\mc V\otimes\mc V)[S,S^{-1}]\subset\mc Q(\mc V)$.
\end{proof}
\begin{example}\label{20211002:exa1}
We have the analogue of the $\lambda$-bracket \eqref{20210812:eq1} (see also Example \ref{Ex:non-local})
in the rational case. Let $\mc R$ denote the algebra of non-commutative difference polynomials in one variable $u$ and let $r(\lambda)\in\kk(\lambda)$ be a rational function such that
$r(\lambda)=-r(\lambda^{-1})$. Then, consider the double multiplicative $\lambda$-bracket on $\mc R$
defined by 
$$
\dgal{u_\lambda u}=g\iota_+r(\lambda S)\bullet g^{\sigma}
\,,
$$
where $g=(\alpha u+\beta)\otimes(\alpha u+\beta)$, $\alpha,\beta\in\kk$. 
Skewsymmetry and Jacobi identity hold on generators (same proof as for
Proposition \ref{20210812:prop1}) and by Theorem \ref{prop:master-rat} we get a rational double multiplicative Poisson vertex algebra
structure on $\mc R$.
\end{example}

\begin{remark}
Let $\mc V$ be an algebra of non-commutative difference functions in $\ell$ variables. We introduce the following sets
\begin{align*}
&\loc(\mc V):=\{H(S)\in\Mat_{\ell\times\ell}\left((\mc V\otimes\mc V)[S,S^{-1}]\right)\mid H(S)\text{ local Poisson structure on }\mc V\}
\,,
\\
&\nonloc(\mc V):=\{H(S)\in\Mat_{\ell\times\ell}\left((\mc V\otimes\mc V)[[S,S^{-1}]]\right)\mid H(S)\text{ non-local Poisson structure on }\mc V\}
\,,
\\
&\rat(\mc V):=\{H(S)\in\Mat_{\ell\times\ell}\left(\mc Q(\mc V)\right)\mid H(S)\text{ Poisson structure of rational type on }\mc V\}
\,,
\end{align*}
(cf. Definition \ref{poisson-structure} and Theorems \ref{prop:master-nl} and \ref{prop:master-rat}).
We point out that despite there is the obvious inclusion $\Mat_{\ell\times\ell}\left(\mc Q(\mc V)\right)\subset \Mat_{\ell\times\ell}\left((\mc V\otimes\mc V)[[S,S^{-1}]]\right)$,
we have
$$
\rat(\mc V)\not\subset\nonloc(\mc V)
$$
in view of the difference in the skewsymmetry axiom in Definitions \ref{def:nonlocal} and \ref{rdmPVA}. As an example, let $\mc V$ be an algebra of 
non-commutative difference functions in one variable $u$. By Example \ref{20211002:exa1} we have that
\begin{equation}\label{20211011:eq1}
\dgal{u_\lambda u}=\iota_+\frac{1+\lambda}{1-\lambda}(1\otimes1)=1\otimes1+\sum_{n\geq1}2(1\otimes1)\lambda^n\,,
\end{equation}
defines a rational double multiplicative Poisson vertex algebra structure on $\mc V$, since $r(\lambda)=\frac{1+\lambda}{1-\lambda}$
satisfies $r(\lambda)=-r(\lambda^{-1})$. The latter condition is equivalent to the skew-symmetry axiom \eqref{20211013:eq2}. On the other hand, if we think of \eqref{20211011:eq1} as a non-local double multiplicative
$\lambda$-bracket, we can compute 
$$-|_{x=S}\dgal{u_{\lambda^{-1}x^{-1}} u}
=-1\otimes1-\sum_{n\geq1}2(1\otimes1)\lambda^{-n}\,,$$
which is clearly different from $\dgal{u_\lambda u}$ 
so that the skew-symmetry axiom \eqref{Eq:DMA2} does not hold in the non-local case. 
Finally, it is immediate to check that
$$
\rat(\mc V)\cap\nonloc(\mc V)=\loc(\mc V)
\,.
$$
\end{remark}

\subsubsection{Poisson structure of rational type for the non-commutative Narita-Itoh-Bogoyavlensky hierarchy}
Let $\mc R=\mc R_1$ be the algebra of non-commutative difference polynomials in one variable $u$.
For $a(z),b(z),c(z),d(z)\in\kk(z)$ we consider, using the symbol map \eqref{eq:symb2},
the following pseudodifference operator of rational type
\begin{equation}\label{eq:NIH}
\begin{split}
H(S)&=(1\otimes u)\iota_+a(S)\bullet(1\otimes u)+(1\otimes u)\iota_+b(S)\bullet(u\otimes1)
\\
&+(u\otimes 1)\iota_+c(S)\bullet(1\otimes u)+(u\otimes 1)\iota_+d(S)\bullet(u\otimes1)
\in\mc Q(\mc R)\,.
\end{split}
\end{equation}
By \eqref{20210914:eq6} we have
\begin{equation}\label{eq:NIH-adj}
\begin{split}
H^*(S)&=(u\otimes 1)\iota_+a(S^{-1})\bullet(u\otimes 1)+(1\otimes u)\iota_+b(S^{-1})\bullet(u\otimes1)
\\
&+(u\otimes 1)\iota_+c(S^{-1})\bullet(1\otimes u)+(1\otimes u)\iota_+d(S^{-1})\bullet(1\otimes u)
\,.
\end{split}
\end{equation}
Using the notation \eqref{notation-bar}, the skewsymmetry condition $H(S)=-H^*(S)$ is equivalent to the identity
\begin{equation}\label{eq3}
\begin{split}
&\left[\iota_+a(x)+\iota_+d(x^{-1})\right]\left(1\otimes (|_{x= S}u)u\right)
+\left[\iota_+d(y)+\iota_+a(y^{-1})\right]\left(u(|_{y=S} u)\otimes1\right)\\
&+\left[\iota_+b(x)+\iota_+c(y)+\iota_+b(x^{-1})+\iota_+c(y^{-1})
\right]\left((_{x=S}u)\otimes (|_{y=S}u)\right)=0\,.
\end{split}
\end{equation}
Since $1\otimes (|_{x= S}u)u$, $u(|_{y=S} u)\otimes1$ and $(_{x=S}u)\otimes (|_{y=S}u)$ are linearly independent, then the skewsymmetry 
of $H(S)$ is equivalent to the conditions
\begin{equation}\label{eq1}
d(z)=-a(z^{-1})\,,
\qquad
b(x)+b(x^{-1})=-c(y)-c(y^{-1})\,.
\end{equation}
Since the LHS of the second equation in \eqref{eq1} is independent of $y$ and the RHS is independent of $x$, then they need to be both equal to a constant $2\alpha\in\kk$.
Hence,
\begin{equation}\label{eq2}
b(z)=b_1(z)+\alpha\,,\quad b_1(z)=-b_1(z^{-1})\,,
\qquad c(z)=c_1(z)-\alpha\,,\quad c_1(z)=-c_1(z^{-1})
\,.
\end{equation}
Inserting the first condition in \eqref{eq1} and the conditions in \eqref{eq2} in the definition of $H(S)$ given in
\eqref{eq:NIH}, we see that $H(S)$ is skewadjoint if and only if it has the form
\begin{equation}\label{eq:NIH2}
\begin{split}
H(S)&=(1\otimes u)\iota_+a(S)\bullet(1\otimes u)+(1\otimes u)\iota_+b(S)\bullet(u\otimes1)
\\
&+(u\otimes 1)\iota_+c(S)\bullet(1\otimes u)-(u\otimes 1)\iota_+a(S^{-1})\bullet(u\otimes1)
\,,
\end{split}
\end{equation}
where $b(z)=-b(z^{-1})$ and $c(z)=-c(z^{-1})$.
\begin{theorem}\label{thm:NIH}
The pseudodifference operator $H(S)$ in \eqref{eq:NIH2} defines a
rational double multiplicative Poisson vertex algebra structure
on $\mc R$ if and only if for some $k\geq 1$ and $p\in \mb Z$, 
\begin{equation}\label{20211002:eq1}
\begin{split}
&a(z)=z^p a_1(z^k)\,, \qquad a_1(z):=\alpha \frac{1}{1-z}\,,  \\
&b(z)=c(z)=b_1(z^k)\,, \quad b_1(z)=\beta \frac{1+z}{1-z}\,,
\end{split} 
\end{equation}
where $\alpha,\beta\in\kk$ are such that $\alpha(2\beta+\alpha)=0$.
\end{theorem}
Before proving Theorem \ref{thm:NIH} we need some preliminary results.
Let $H(S)$ be as in \eqref{eq:NIH2} and let us define a double multiplicative $\lambda$-bracket on $\mc R$ by setting
\begin{equation}\label{NIH-lambda}
\begin{split}
\dgal{ u_\lambda u}&=H(\lambda)=(1\otimes u)\iota_+a(\lambda S)\bullet(1\otimes u)+(1\otimes u)\iota_+b(\lambda S)\bullet(u\otimes1)
\\
&+(u\otimes 1)\iota_+c(\lambda S)\bullet(1\otimes u)-(u\otimes 1)\iota_+a(\lambda^{-1} S^{-1})\bullet(u\otimes1)\in(\mc R\otimes\mc R)((\lambda))\,,
\end{split}
\end{equation}
and extending to $\mc R$ by the Master Formula \eqref{master-infinite}.
\begin{proposition}\label{20210920:prop1}
Jacobi identity on generators holds for the double multiplicative $\lambda$-bracket \eqref{NIH-lambda} if and only if the rational functions
$a(z),b(z),c(z)\in\kk(z)$, with $b(z^{-1})=-b(z)$ and $c(z^{-1})=-c(z)$, satisfy
\begin{equation}\label{conditions}
\begin{split}
&\left(\iota_+b(z)+\iota_+ b(w)\right)\iota_+b(zw)-\iota_+ b(z)\iota_+b(w)=\gamma\,,
\\
&\left(\iota_+c(z)+\iota_+ c(w)\right)\iota_+c(zw)-\iota_+ c(z)\iota_+c(w)=\gamma\,,
\\
&\left(\iota_+b(z)+\iota_+ b(w)\right)\iota_+a(zw)+\iota_+ a(z)\iota_+a(w)=0\,,
\\
&\left(\iota_+c(z)+\iota_+ c(w)\right)\iota_+a(zw)+\iota_+ a(z)\iota_+a(w)=0\,,
\end{split}
\end{equation}
for some  $\gamma\in\kk$.
\end{proposition}
\begin{proof}
For convenience in the computations, let us use notation \eqref{notation-bar}
to rewrite \eqref{NIH-lambda} as
\begin{equation}\label{NIH-lambda2}
\begin{split}
\dgal{ u_\lambda u}&=\iota_+a(\lambda x)\left(1\otimes (|_{x=S}u)u\right)
+\iota_+b(\lambda x)\left((|_{x=S}u)\otimes u\right)
\\&
+\iota_+c(\lambda y)\left(u\otimes (|_{y=S}u)\right)
-\iota_+a(\lambda^{-1} y^{-1})\left(u(|_{y=S}u)\otimes 1\right)
\,.
\end{split}
\end{equation}
We start by computing explicitly the three terms appearing in the Jacobi identity
\begin{equation}\label{tocheck=0}
\dgal{u_\lambda \dgal{u_\mu u}}_L
-\dgal{u_\mu \dgal{u_\lambda u}}_R=\dgal{\dgal{u_\lambda u}_{\lambda\mu}u}_L
\,.
\end{equation}
By a long but straightforward computation, using the first equation in \eqref{20210616:eq2a}, sesquilinearity \eqref{Eq:DMA1}, the left Leibniz rule \eqref{Eq:DML} and \eqref{NIH-lambda2},  we get
\begin{align}
&\dgal{u_\lambda\dgal{u_\mu u}}_L\notag
\\
&=\iota_+a(\lambda x)\iota_+b(\lambda\mu x y)
\big(1\otimes(|_{x=S}u)(|_{y=S}u)\otimes u\big)
\label{A1}
\\
&+\iota_+b(\lambda x)\iota_+b(\lambda\mu xy)
\big((|_{x=S}u)\otimes (|_{y=S}u)\otimes u\big)
\label{A2}
\\
&+\iota_+b(\lambda \mu xy)\iota_+c(\lambda y)
\big((|_{x=S}u)\otimes(|_{y=S}u)\otimes u\big)
\label{A3}
\\
&-\iota_+a(\lambda^{-1} y^{-1})\iota_+b(\lambda\mu x y)
\big((|_{x=S}u)(|_{y=S}u)\otimes 1\otimes u\big)
\label{A4}
\\
&+\iota_+a(\lambda x)\iota_+c(\mu z)
\big(1\otimes (|_{x=S}u)u\otimes (|_{z=S}u)\big)
\label{A5}
\\
&+\iota_+b(\lambda x)\iota_+c(\mu z)
\big((|_{x=S}u)\otimes u\otimes (|_{z=S}u)\big)
\label{A6}
\\
&+\iota_+c(\lambda y)\iota_+c(\mu z)
\big(u\otimes (|_{y=S}u)\otimes (|_{z=S}u)\big)
\label{A7}
\\
&-\iota_+a(\lambda^{-1} y^{-1})\iota_+c(\mu z)
\big(u(|_{y=S}u)\otimes 1\otimes (|_{z=S}u)\big)
\label{A8}
\\
&-\iota_+a(\lambda x)\iota_+a(\mu^{-1}z^{-1})
\big(1\otimes(|_{x=S}u)u(|_{z=S}u)\otimes1\big)
\label{A9}
\\
&-\iota_{+}a(\mu^{-1}z^{-1})\iota_{+}b(\lambda x)
\big((|_{x=S}u)\otimes u(|_{z=S}u)\otimes 1\big)
\label{A10}
\\
&-\iota_+a(\mu^{-1}z^{-1})\iota_{+}c(\lambda y)
\big(u\otimes(|_{y=S}u)(|_{z=S}u)\otimes1\big)
\label{A11}
\\
&+\iota_+a(\lambda ^{-1}y^{-1})\iota_+a(\mu^{-1}z^{-1})
\big(u(|_{y=S}u)\otimes (|_{z=S}u)\otimes 1\big)
\label{A12}
\\
&-\iota_+a(\lambda^{-1}\mu^{-1}y^{-1}z^{-1})\iota_+a(\lambda y)
\big(u\otimes (|_{y=S}u)(|_{z=S}u)\otimes 1\big)
\label{A13}
\\
&-\iota_+a(\lambda^{-1} \mu^{-1} y^{-1}z^{-1})\iota_+b(\lambda y)
\big(u(|_{y=S}u)\otimes (|_{z=S}u)\otimes 1\big)
\label{A14}
\\
&-\iota_+a(\lambda^{-1} \mu^{-1} y^{-1}z^{-1})\iota_+c(\lambda z)
\big(u(|_{y=S}u)\otimes(|_{z=S}u)\otimes 1\big)
\label{A15}
\\
&+\iota_+a(\lambda^{-1} \mu^{-1} y^{-1}z^{-1})\iota_+a(\lambda^{-1} z^{-1})
\big(u(|_{y=S}u)(|_{z=S}u)\otimes 1\otimes1\big)\,.
\label{A16}
\end{align}
Similarly, but using the second equation in \eqref{20210616:eq2a} instead, we get
\begin{align}
&\dgal{u_\mu\dgal{u_\lambda u}}_R\notag
\\
&=\iota_+a(\lambda\mu xy)\iota_+a(\mu x)
\big(1\otimes 1\otimes(|_{x=S}u)(|_{y=S}u) u\big)
\label{B1}
\\
&+\iota_+a(\lambda\mu xy)\iota_+b(\mu x)
\big(1\otimes(|_{x=S}u)\otimes (|_{y=S}u)u\big)
\label{B2}
\\
&+\iota_+a(\lambda\mu xy)\iota_+c(\mu y)
\big(1\otimes (|_{x=S}u)\otimes (|_{y=S}u)u\big)
\label{B3}
\\
&-\iota_+a(\lambda\mu xy)\iota_+a(\mu^{-1} y^{-1})
\big(1\otimes (|_{x=S}u)(|_{y=S}u)\otimes u\big)
\label{B4}
\\
&+\iota_+a(\lambda x)\iota_+a(\mu y)
\big(1\otimes(|_{x=S}u)\otimes (|_{y=S}u) u\big)
\label{B5}
\\
&+\iota_+a(\lambda x)\iota_+b(\mu y)
\big(1\otimes(|_{x=S}u)(|_{y=S}u)\otimes u\big)
\label{B6}
\\
&+\iota_+a(\lambda x)\iota_+c(\mu z)
\big(1\otimes (|_{x=S}u)u\otimes (|_{z=S}u)\big)
\label{B7}
\\
&-\iota_+a(\lambda x)\iota_+a(\mu^{-1} z^{-1})
\big(1\otimes(|_{x=S}u)u(|_{z=S}u)\otimes 1\big)
\label{B8}
\\
&+\iota_+a(\mu y)\iota_+b(\lambda x)
\big((|_{x=S}u)\otimes1\otimes (|_{y=S}u)u\big)
\label{B9}
\\
&+\iota_+b(\lambda x)\iota_+b(\mu y)
\big((|_{x=S}u)\otimes(|_{y=S}u)\otimes u\big)
\label{B10}
\\
&+\iota_+b(\lambda x)\iota_+c(\mu z)
\big((|_{x=S}u)\otimes u\otimes (|_{z=S}u)\big)
\label{B11}
\\
&-\iota_+a(\mu^{-1}z^{-1})\iota_+b(\lambda x)
\big((|_{x=S}u)\otimes u(|_{z=S}u)\otimes 1\big)
\label{B12}
\\
&+\iota_+a(\mu y)\iota_+c(\lambda\mu yz)
\big(u\otimes 1\otimes (|_{y=S}u)(|_{z=S}u)\big)
\label{B13}
\\
&+\iota_+b(\mu y)\iota_+c(\lambda\mu yz)
\big(u\otimes (|_{y=S}u)\otimes (|_{z=S}u)\big)
\label{B14}
\\
&+\iota_+c(\lambda \mu yz)\iota_+c(\mu z)
\big(u\otimes (|_{y=S}u)\otimes (|_{z=S}u)\big)
\label{B15}
\\
&-\iota_+a(\mu^{-1}z^{-1})\iota_+c(\lambda\mu yz)
\big(u\otimes (|_{y=S}u)(|_{z=S}u)\otimes 1\big)
\label{B16}
\,.
\end{align}
Finally, using \eqref{20210616:eq2}, sesquilinearity \eqref{Eq:DMA1}, the right Leibniz rule \eqref{Eq:DMR} and \eqref{NIH-lambda2} we get
\begin{align}
&\dgal{\dgal{u_\lambda u}_{\lambda \mu} u}_L\notag
\\
&=\iota_+a(\lambda \mu x y)\iota_+b(\mu^{-1}x^{-1})
\big(1\otimes (|_{x=S}u)\otimes (|_{y=S}u)u\big)
\label{C1}
\\
&+\iota_+b(\lambda\mu x y)\iota_+b(\mu^{-1} y^{-1})
\big((|_{x=S}u)\otimes (|_{y=S}u)\otimes u\big)
\label{C2}
\\
&+\iota_+b(\mu^{-1} y^{-1})\iota_+c(\lambda\mu y z )
\big(u\otimes (|_{y=S}u)\otimes (|_{z=S}u)\big)
\label{C3}
\\
&-\iota_+a(\lambda^{-1}\mu^{-1} y^{-1} z^{-1})\iota_+b(\mu^{-1} z^{-1} )
\big(u(|_{y=S}u)\otimes (|_{z=S}u)\otimes 1\big)
\label{C4}
\\
&+\iota_+a(\lambda\mu xy)\iota_+c(\lambda x)
\big(1\otimes(|_{x=S}u)\otimes(|_{y=S}u)u\big)
\label{C5}
\\
&+\iota_+b(\lambda \mu xy)\iota_+c(\lambda y)
\big((|_{x=S}u)\otimes(|_{y=S}u)\otimes u\big)
\label{C6}
\\
&+\iota_+c(\lambda\mu yz)\iota_+c(\lambda y)
\big(u\otimes(|_{y=S}u)\otimes (|_{z=S}u)\big)
\label{C7}
\\
&-\iota_+a(\lambda^{-1}\mu^{-1} y^{-1}z^{-1})\iota_+c(\lambda z)
\big(u(|_{y=S}u)\otimes (|_{z=S}u)\otimes 1\big)
\label{C8}
\\
&-\iota_+a(\lambda\mu xy)\iota_+a(\lambda^{-1} x^{-1})
\big((|_{x=S}u)\otimes 1\otimes (|_{y=S}u)u\big)
\label{C9}
\\
&-\iota_+a(\lambda^{-1}y^{-1})\iota_+b(\lambda\mu xy)
\big((|_{x=S}u)(|_{y=S}u)\otimes 1\otimes u\big)
\label{C10}
\\
&-\iota_+a(\lambda^{-1}y^{-1})\iota_+c(\lambda\mu y z)
\big(u(|_{y=S}u)\otimes1\otimes(|_{z=S}u)\big)
\label{C11}
\\
&+\iota_+a(\lambda^{-1} \mu^{-1} y^{-1}z^{-1})\iota_+a(\lambda^{-1}z^{-1})
\big(u(|_{y=S}u)(|_{z=S}u)\otimes 1\otimes1\big)
\label{C12}
\\
&-\iota_+a(\lambda\mu xy)\iota_+a(\mu x)
\big(1\otimes 1\otimes(|_{x=S}u)(|_{y=S}u) u\big)
\label{C13}
\\
&-\iota_+a(\mu y)\iota_+b(\lambda \mu xy)
\big((|_{x=S}u)\otimes1\otimes(|_{y=S}u)u\big)
\label{C14}
\\
&-\iota_+a(\mu y)\iota_+c(\lambda\mu yz)
\big(u\otimes 1\otimes (|_{y=S}u)(|_{z=S}u)\big)
\label{C15}
\\
&+\iota_+a(\mu z)\iota_+a(\lambda^{-1}\mu^{-1}y^{-1}z^{-1})
\big(u(|_{y=S}u)\otimes 1\otimes (|_{z=S}u)\big)\,.
\label{C16}
\end{align}
The following terms cancel in the Jacobi identity \eqref{tocheck=0}:
\begin{align*}
\eqref{A3}-\eqref{C6}=0
\,,\quad
&\eqref{A4}-\eqref{C10}=0
\,,\quad
\eqref{A5}-\eqref{B7}=0\,,
\quad
\eqref{A6}-\eqref{B11}=0\,,
\\
\eqref{A9}-\eqref{B8}=0\,,
\quad
&\eqref{A10}-\eqref{B12}=0\,,
\quad
\eqref{A15}-\eqref{C8}=0\,,
\quad
\eqref{A16}-\eqref{C12}=0\,,
\\
&\eqref{B1}+\eqref{C13}=0\,,
\quad
\eqref{B13}+\eqref{C15}=0\,,
\end{align*}
and using the fact that $b(z)=-b(z^{-1})$ we have also the cancellation
$$
\eqref{B2}+\eqref{C1}=0\,,
\quad
\eqref{B14}+\eqref{C3}=0\,.
$$
Next, observe that equation \eqref{tocheck=0} can be understood as an identity in
the space $\mc V\otimes\mc V\otimes\mc V$ with coefficients in
$\kk[x^{\pm1},y^{\pm1},z^{\pm1}]((\lambda,\mu))$. Since the elements
$u\otimes u\otimes u$, $1\otimes u^2\otimes u$, $u^2\otimes 1 \otimes u$,
$u\otimes u^2\otimes 1$, $ u^2\otimes u\otimes 1$, $1\otimes u\otimes u^2$
and 
$u\otimes 1\otimes u^2$ are linearly independent, the Jacobi identity \eqref{tocheck=0} holds if and only if each coefficient of these elements
vanishes, leading to seven further identities that we want to prove being equivalent
to the four conditions \eqref{conditions}.

Collecting the terms \eqref{A2}, \eqref{A7},
\eqref{B10}, \eqref{B15}, \eqref{C2} and \eqref{C7}
acting on $u\otimes u \otimes u$, and using the fact that $b(z)=-b(z^{-1})$
we then get the following identity
\begin{equation}\label{X1}
\begin{split}
   & \Big(
    \iota_+b(\lambda x)\iota_+b(\lambda\mu xy)
    +\iota_+b(\mu y)\iota_+b(\lambda\mu x y)
    -\iota_+b(\lambda x)\iota_+b(\mu y)
    \Big)
        \big((|_{x=S}u)\otimes(|_{y=S}u)\otimes u\big)
    \\
    &
    =\Big(
    \iota_+c(\lambda y)\iota_+c(\lambda\mu yz)
    +\iota_+c(\mu z)\iota_+c(\lambda \mu yz)
    -\iota_+c(\lambda y)\iota_+c(\mu z)\Big)
       \big(u\otimes (|_{y=S}u)\otimes(|_{z=S}u)\big)
\,.
\end{split}
\end{equation}
Note that the LHS of \eqref{X1} is independent of $z$, hence of $\mu$, while the RHS is independent of $x$, hence of $\lambda$. This forces both sides to be a
constant multiple of $u\otimes u \otimes u$. This condition is equivalent to the first two conditions in \eqref{conditions}.

Next, collecting the terms \eqref{A1}, \eqref{B4},
and \eqref{B6}
acting on $1\otimes u^2 \otimes u$
we get the identity
\begin{equation}\label{X2}
\begin{split}
   & \Big(
    \iota_+a(\lambda x)\iota_+b(\lambda\mu x y)
    +\iota_+a(\lambda\mu xy)\iota_+a(\mu^{-1} y^{-1})
    -\iota_+a(\lambda x)\iota_+b(\mu y)
    \Big)
    \big(1\otimes(|_{x=S}u)(|_{y=S}u)\otimes u\big)
    =0
    \,.
\end{split}
\end{equation}
Using the fact that $b(z)=-b(z^{-1})$ the identity
\eqref{X2} is equivalent to the third condition in \eqref{conditions}.
We get the same condition looking at the coefficient of $u^2\otimes u\otimes 1$
and $u\otimes 1\otimes u^2$.

Finally, collecting the terms \eqref{A8}, \eqref{C11},
and \eqref{C16}
acting on $u^2\otimes 1\otimes u$
we get the identity
\begin{equation}\label{X3}
\begin{split}
   & \Big(
\iota_+a(\lambda^{-1} y^{-1})\iota_+c(\mu z)
-\iota_+a(\lambda^{-1}y^{-1})\iota_+c(\lambda\mu y z)
\\
&+\iota_+a(\mu z)\iota_+a(\lambda^{-1}\mu^{-1}y^{-1}z^{-1})
\Big)
\big(u(|_{y=S}u)\otimes 1\otimes (|_{z=S}u)\big)
=0
    \,.
\end{split}
\end{equation}
Using the fact that $c(z)=-c(z^{-1})$ the identity
\eqref{X3} is equivalent to the fourth condition in \eqref{conditions}.
We get the same condition by looking at the coefficient of $1\otimes u\otimes u^2$
and $u\otimes u^2\otimes 1$. This concludes the proof.
\end{proof}
\begin{lemma}\label{20210920:lem1}
Let $R(z),Q(z)\in\kk((z))$.
\begin{enumerate}[(a)]
    \item Let $\gamma\in\kk$. Then, $R(z)$ satisfies the equation (in $\kk((z,w))$)
\begin{equation}\label{tosolve1}
(R(z)+R(w))R(zw)-R(z)R(w)=\gamma\,,
\end{equation}
    if and only if $R(z)=\beta$ or $R(z)=R_1(z^k)$ for $k\geq1$ and 
    \begin{equation}\label{20210920:eq1}
    R_1(z)=\beta\big(1+2\sum_{n\geq1}z^n\big)
    =\beta\, \iota_+\frac{1+z}{1-z}
    \,,
    \end{equation}
    where $\beta^2=\gamma$.
    \item Let $R(z)=R_1(z^k)$ for some $k\geq1$ and $R_1(z)$ as in \eqref{20210920:eq1}. Then $Q(z)$, satisfies the equation (in $\kk((z,w))$)
\begin{equation}\label{tosolve2}
    (R(z)+R(w))Q(zw)+Q(z)Q(w)=0\,,
\end{equation}
    if and only if 
\begin{equation}\label{20210920:eq2}
    Q(z)=\alpha\sum_{n\geq 0}z^{nk+p}=\iota_+\frac{\alpha z^p}{1-z^k}
    \,,\quad p\in\mb Z\,,
\end{equation}
    where $\alpha,\beta\in\kk$ satisfy $\alpha(2\beta+\alpha)=0$.
\end{enumerate}
\end{lemma}
\begin{proof}
For part (a) it is straightforward to check that $R(z)=\beta$ or $R(z)$ as in
\eqref{20210920:eq1}, with $\beta^2=\gamma$, solve \eqref{tosolve1}.
On the other hand, let $N\in\mb Z$ and let us write $R(z)=\sum_{n\geq N}r_nz^n$, with $r_N\neq0$. Then, \eqref{tosolve1} becomes
\begin{equation}\label{tosolve1b}
\sum_{m\geq2N,n\geq N}r_{m-n}r_nz^mw^n
+\sum_{m\geq N,n\geq 2N}r_{n-m}r_mz^mw^n
-\sum_{n,m\geq N}r_mr_nz^mw^n=\gamma\,.
\end{equation}
If $N<0$, then equating the coefficient of $z^{2N}w^N$ in both sides of
\eqref{tosolve1b} we get $r^2_N=0$, hence $r_N=0$, which is a contradiction. If $N>0$,
equating the coefficient of $z^Nw^N$ in both sides of \eqref{tosolve2} we get again $r_N^2=0$, which leads to a contradiction. Hence, it remains to consider the case
$N=0$. By equating the coefficient of $z^0w^0$ in both sides of \eqref{tosolve2}
we get
$r_0^2=\gamma$, which leads to $r_0=\beta$ with $\beta^2=\gamma$. By equating the coefficient of $z^kw^k$, $k>0$, in both sides of \eqref{tosolve2} we get
$$
(2\beta-r_k)r_k=0\,.
$$
Hence, $r_k=2\beta$ or $r_k=0$, for $k\geq1$. 
Let $k\geq 1$ be the smallest integer such that $r_k=2\beta\neq 0$ but $r_h=0$ for all $1\leq h<k$. 
We claim that for each $n\geq1$, $r_{nk}=r_k$ while $r_{nk+h}=0$ for $1\leq h<k$. This is shown by induction. Looking at the coefficient of $z^k w^{nk}$ in \eqref{tosolve2}, we get $r_k r_{(n-1)k} - r_k r_{nk}=0$ which yields the first equality. For the second, we look at the coefficient of $z^k w^{nk+h}$ in \eqref{tosolve2} which gives $r_k r_{(n-1)k+h}-r_k r_{nk+h}=0$. 

Hence, either $r_n=0$ for every $n\geq1$, thus $R(z)=\beta$, or the non-zero terms are  $r_{nk}=2\beta$, for every $n\geq 1$ and some $k\geq1$, which gives that $R(z)=R_1(z^k)$ for $R_1(z)$ as in \eqref{20210920:eq1}. This proves part (a).

For part (b), it is straightforward to check that $Q(z)$ in \eqref{20210920:eq2} solves \eqref{tosolve2}. 
Moreover, if $Q(z)$ is a solution to \eqref{tosolve2}, then $z^pQ(z)$ is also
a solution, for every $p\in\mb Z$. Hence, discarding the trivial solution $Q(z)=0$, we are left to seek solutions
of the form $Q(z)=\sum_{n\geq0}q_nz^n$ with $q_0\neq 0$. Using \eqref{20210920:eq1} to expand $R(z)=R_1(z^k)$, we rewrite \eqref{tosolve2} as
\begin{equation}\label{tosolve2b}
2\beta\big(1+\sum_{m\geq1}z^{mk}+\sum_{n\geq1}w^{nk}\big)
\sum_{\ell\geq0}q_\ell z^\ell w^\ell
+\sum_{n,m\geq 0}q_mq_nz^mw^n=0\,.
\end{equation}
Equating the coefficient of $z^\ell w^\ell$, $\ell\geq0$, in both sides of \eqref{tosolve2b} we get
\begin{equation}\label{Eq:solv3}
    2\beta q_\ell+q_\ell^2=0\,.
\end{equation}
Next, looking at the coefficient of $z^{h+nk}w^0$, for $1\leq h <k$ and $n\geq1$, we find that 
$q_{h+nk}q_0=0$. Since $q_0\neq 0$ by assumption, this yields 
\begin{equation} \label{Eq:solv4}
    q_j=0 \quad \text{ if }\quad j \notin k\mb Z_{\geq 0}\,.
\end{equation}
Finally, if we look at the coefficient of $z^{nk}w^{0}$, $n\geq 1$,  we get
$q_0(2\beta+q_{nk})=0$. Together with the condition \eqref{Eq:solv3}, we get that $q_0 (q_0-q_{nk})=0$, hence $q_{nk}=q_0$.  
Combining this identity with \eqref{Eq:solv4}, we see that $q_\ell=0$ except if $\ell$ is a multiple of $k$, in which case $q_{nk}=q_0$. 
Due to \eqref{Eq:solv3}, we see by adding the trivial solution $Q(z)=0$ that we can write $q_0=\alpha$, where $\alpha\in\kk$ is such that $\alpha(2\beta+\alpha)=0$.
This concludes the proof of part (b).
\end{proof}
\begin{proof}[Proof of Theorem \ref{thm:NIH}]
By Theorem \ref{prop:master-rat} we need to show that the double multiplicative $\lambda$-bracket defined by \eqref{NIH-lambda} satisfies skewsymmetry and Jacobi identity on generators.
Skewsymmetry holds since, by construction, $H(S)=-H^*(S)$. By Proposition \ref{20210920:prop1} Jacobi identity holds on generators if and only if
the four conditions in \eqref{conditions} are satisfied. By Lemma \ref{20210920:lem1}(a) and the fact that $b(z)=-b(z^{-1})$,
$c(z)=-c(z^{-1})$, the first two conditions in \eqref{conditions} give that
$$
b(z)=\beta\frac{1+z^k}{1-z^k}\,, \quad 
c(z)=\beta\frac{1+z^{\tilde{k}}}{1-z^{\tilde{k}}}\,,
$$
for some $\beta\in\kk$ and $k,\tilde{k}\geq 1$. 
By Lemma \ref{20210920:lem1}(b) , 
the third equation in \eqref{conditions} is satisfied if and only if
$$
a(z)=\frac{\alpha z^p}{1-z^k}\,,\quad p\in\mb Z\,,
$$
where $\alpha\in\kk$ is such that $\alpha(2\beta+\alpha)=0$. 
Similarly, the fourth equation in \eqref{conditions} is satisfied if and only if
$$
a(z)=\frac{\tilde{\alpha} z^{\tilde{p}}}{1-z^{\tilde{k}}}\,,\quad {\tilde{p}}\in\mb Z\,,
$$
where $\tilde{\alpha}\in\kk$ is such that $\tilde{\alpha}(2\beta+\tilde{\alpha})=0$. Equating both forms for $a(z)$ yields that ${\tilde{k}}=k$, ${\tilde{p}}=p$, and ${\tilde{\alpha}}=\alpha$, which concludes the proof.
\end{proof}
\begin{remark}\label{Rem:HamCW}
Let us introduce the following notation (motivated by \eqref{eq:action-diff}): $\mathrm{r}_u=1\otimes u$, $\mathrm{l}_u=u\otimes 1$, $\mathrm{c}_u=\mathrm{l}_u-\mathrm{r}_u$ and $\mathrm{a}_u=\mathrm{l}_u+\mathrm{r}_u$. 
By Theorem \ref{thm:NIH} with $\alpha=-1$, $\beta=1/2$, $k=1$ and $p=q+1$ for $q\geq1$, we have that the pseudodifference operator of rational type
\begin{equation}\label{casatiwang}
\begin{split}
H(S)&=-\mathrm{r}_u\iota_+\frac{S^{q+1}}{1-S}\bullet\mathrm{r}_u
+\frac12\mathrm{r}_u\iota_+\frac{1+S}{1-S}\bullet\mathrm{l}_u
+\frac12\mathrm{l}_u\iota_+\frac{1+S}{1-S}\bullet\mathrm{r}_u
-\mathrm{l}_u\iota_+\frac{S^{-q}}{1-S}\bullet\mathrm{l}_u
\\
&=\sum_{i=1}^q\left(\mathrm{r}_u S^i\bullet \mathrm{r}_u
-\mathrm{l}_u S^{-i}\bullet \mathrm{l}_u\right)
-\frac12\mathrm{a}_u\bullet\mathrm{c}_u
-\frac12\mathrm{c}_u\iota_+\frac{1+S}{1-S}\bullet\mathrm{c}_u
\,,
\end{split}
\end{equation}
defines a Poisson structure of rational type on $\mc R$. The operator $H(S)$ in \eqref{casatiwang} appeared in \cite{CW2} (without the use of the embedding $\iota_+$) and it is called
the non-local Poisson structure of the non-commutative Narita-Itoh-Bogoyavlensky lattice hierarchy.
However, note that $H(S)$ does not define a non-local Poisson structure on $\mc R$ in the sense of Theorem \ref{prop:master-nl}. Indeed, let us replace in \eqref{eq:NIH2}, the Laurent series
$\iota_+ a(z)$, $\iota_+ b(z)$ and $\iota_+ c(z)$ by bilateral series $A(z),B(z),C(z)\in\kk[[z,z^{-1}]]$ such that $B(z)=-B(z^{-1})$, $C(z)=-C(z^{-1})$.
The same computations as in the proof of Proposition \ref{20210920:prop1} show that $H(S)$ is a non-local Poisson structure if and only if conditions \eqref{conditions}, obtained by replacing $\iota_+ a(z)$, $\iota_+ b(z)$ and $\iota_+ c(z)$ with $A(z),B(z),C(z)$, hold. It is
not hard to check that the only solution to those equations is then $A(z)=B(z)=C(z)=0$.
\end{remark}

\begin{remark}
Motivated by the works \cite{EKV17,EKV19} on a modification of the commutative Narita-Itoh-Bogoyavlensky lattice hierarchy, it is natural to ask whether there exists a rational Poisson structure with constant
coefficient $K(S)=\iota_+ r(S)(1\otimes1)$, $r(z)\in\kk(z)$, compatible with the rational Poisson structure $H(S)$ from Theorem \ref{thm:NIH}. Compatibility means that $H(S)+K(S)$ is a rational Poisson structure as well. By Example \ref{20211002:exa1}, $K(S)$ is a rational Poisson structure if and only if
$r(z)=-r(z^{-1})$. Let $\dgal{u_\lambda u}_H=H(\lambda)$ be as in \eqref{NIH-lambda2}, and let us set $\dgal{u_\lambda u}_K=K(\lambda)=\iota_+r(\lambda)(1\otimes1)$. Then,
by Theorem \ref{prop:master-rat}, $H(S)+K(S)$ is a rational
Poisson structure if and only if $\dgal{u_\lambda u}=\dgal{u_\lambda u}_H+\dgal{u_\lambda u}_K$
defines a rational double multiplicative Poisson vertex algebra. To check this, it suffices to verify
that Jacobi identity \eqref{tocheck=0} holds. Using the fact that $H(S)$ and $K(S)$ are rational
Poisson structures this reduces to the condition
\begin{equation}\label{tocheck=0bis}
\dgal{u_\lambda \dgal{u_\mu u}_H}_{K,L}
-\dgal{u_\mu \dgal{u_\lambda u}_H}_{K,R}=\dgal{{\dgal{u_\lambda u}_H}_{\lambda\mu}u}_{K,L}
\,.
\end{equation}
Equation \eqref{tocheck=0bis} is equivalent to the following three equations for the rational function $r(z)$ (we omit the details of the computations):
\begin{equation}\label{conditions-rem}
\begin{split}
&\left(\iota_+b(z x)-\iota_+ b(w x)\right)\iota_+r(z w x)
+\iota_+ a(z x)\iota_+r(w)+\iota_+ a(w^{-1}x^{-1})\iota_+r(z)
=0\,,
\\
&\left(\iota_+b(z x)+\iota_+ b( z w)\right)\iota_+r(w)
-\iota_+ a(z w)\iota_+r(z)+\iota_+ a(z x)\iota_+r(z w x)
=0\,,
\\
&\left(\iota_+b(z x)+\iota_+ b(z w)\right)\iota_+r(w)
+\iota_+ a(z^{-1} w^{-1})\iota_+r(z)-\iota_+ a(z^{-1}x^{-1})\iota_+r(z w x)
=0\,,
\end{split}
\end{equation}
where $a(z)$ and $b(z)$ are as in \eqref{20211002:eq1}. 
Let us show that these conditions imply $r(z)=0$, i.e. there is no compatibility between the two Poisson structures. (We explain the case when $a(z)\neq0$, and we leave to the reader the case with $a(z)=0$, $b(z)\neq 0$.)   
Subtracting the second and third equations in \eqref{conditions-rem},
and using \eqref{20211002:eq1} we get that $r(z)$ should satisfy the identity
\begin{equation}\label{20211002:eq2}
(1-(zw)^{k-2p})\iota_+a(zw)\iota_+r(z)=(1-(zx)^{k-2p})\iota_+a(zx)\iota_+r(z w x)
\,.
\end{equation}
Note that the LHS of \eqref{20211002:eq2} does not depend on $x$. Hence, we can set $x=w$ in the RHS of \eqref{20211002:eq2} and get that $r(z)$ satisfies $\iota_+r(z)=\iota_+r(zw^2)$ whenever $k\neq 2p$. This forces 
$r(z)=\gamma$, for some $\gamma\in\kk$. Since $r(z)=-r(z^{-1})$, we then have $r(z)=0$.
When $k=2p$ (which is a positive even integer by Theorem \ref{thm:NIH}), the first equation in \eqref{conditions-rem} for $w=x$ becomes 
$$
\iota_+
\frac{\alpha ((zx)^{p}-(zx)^{-p})}{1-(zx)^{2p}}
\iota_+r(z)=0\,,
$$
after using the form of $a(z)$ given in \eqref{20211002:eq1}. Since the first term does not vanish, we must have $r(z)=0$. 
\end{remark}


  

\begin{thebibliography}{}

\bibitem[CW1]{CW1}
Casati M., Wang J.P.: 
\emph{Recursion and Hamiltonian operators for integrable nonabelian difference equations}.
Nonlinearity 34, no. 1, 205-236 (2021).

\bibitem[CW2]{CW2}
Casati M., Wang J.P.:
\emph{Hamiltonian structures for integrable nonabelian difference equations}.
Comm. Math. Phys. 392, no. 1, 219-278 (2022).

\bibitem[DSK1]{DSK-coh1} 
De Sole A., Kac V.G.:
\emph{Lie conformal algebra cohomology and the variational complex}. 
Comm. Math. Phys. 292, no. 3, 667-719 (2009).

\bibitem[DSK2]{DSK-coh2}
De Sole A., Kac V.G.: \emph{The variational Poisson cohomology}. 
Jpn. J. Math. 8, no. 1, 1-145 (2013).

\bibitem[DSK3]{DSK-nonloc}
De Sole A., Kac V.G.: \emph{Non-local Poisson structures and applications to the theory of integrable systems}.
Jpn. J. Math. 8, no. 2, 233-347 (2013).

\bibitem[DSKV]{DSKV}  De Sole A., Kac V.G., Valeri D.: {\it Double {P}oisson vertex algebras and non-commutative {H}amiltonian equations}.  Adv. Math. 281, 1025-1099 (2015).

\bibitem[DSKVW1]{DSKVW1}
De Sole A., Kac V.G., Valeri D., Wakimoto M.:
\emph{Local and non-local multiplicative Poisson vertex algebras and differential-difference equations}. 
Comm. Math. Phys. 370, no. 3, 1019–1068 (2019).

\bibitem[DSKVW2]{DSKVWclass}  De Sole A., Kac V.G., Valeri D.: {\it Poisson $\lambda$-brackets for differential-difference equations}.
Int. Math. Res. Not. IMRN 13, 4144–4190 (2020).

\bibitem[EKV1]{EKV17} 
Evripidou C., Kassotakis P., Vanhaecke P.: 
{\it Integrable deformations of the Bogoyavlenskij-Itoh Lotka-Volterra systems}.
Regul. Chaotic Dyn. 22, no. 6, 721-739 (2017).

\bibitem[EKV2]{EKV19}
Evripidou C., Kassotakis P., Vanhaecke P.:
\emph{Integrable reductions of the dressing chain}.
J. Comput. Dyn. 6, no. 2, 277-306 (2019).

\bibitem[FH]{FH}
Fern\'{a}ndez D., Heluani R.:
\emph{Noncommutative Poisson vertex algebras and Courant-Dorfman algebras}. Preprint,
\href{https://arxiv.org/abs/2106.00270}{arXiv:2106.00270 [math.QA]}.

\bibitem[GKK]{GKK98} 
Golenishcheva-Kutuzova M., Kac V.G.: 
\emph{$\Gamma$-conformal algebras}. J. Math. Phys. 39, no. 4, 2290-2305 (1998).
    
 
 \bibitem[Ko]{Ko} Kontsevich, M.: {\it Formal (non)-commutative symplectic geometry}. In: I.M. Gelfand, L. Corwin, J. Lepowsky (Eds.), The Gelfand Mathematical Seminars, 1990--1992, Birkh\"auser, Boston, pp. 173-187 (1993).

\bibitem[KR]{KR} Kontsevich, M., Rosenberg, A.L.: {\it Noncommutative smooth spaces}. The Gelfand Mathematical Seminars, 1996–1999, Gelfand Math. Sem., Birkh\"auser, Boston, pp. 85-108 (2000).
 

\bibitem[P]{P16}  Powell G.: {\it On double Poisson structures on commutative algebras}. J. Geom. Phys. 110, 1-8  (2016).

\bibitem[S]{Sch} Schedler T.: {\it Poisson algebras and Yang-Baxter equations}. In: Advances in quantum computation, Contemp. Math., vol. 482, Amer. Math. Soc., Providence, RI, 2009, pp. 91-106.

\bibitem[ORS]{ORS} Odesskii A.V., Rubtsov V.N., Sokolov V.V.: {\it Parameter-dependent associative Yang-Baxter equations and Poisson brackets}. Int. J. Geom. Methods Mod. Phys. 11, no. 9, 1460036, 18 pages (2014).

\bibitem[OS]{OS} Olver P.J., Sokolov V.V.: {\it Integrable evolution equations on associative algebras}.  
Comm. Math. Phys. 193, no. 2, 245-268 (1998).


\bibitem[VdB1]{VdB1} Van den Bergh M.: {\it Double Poisson algebras}. Trans. Amer. Math. Soc., 360, no. 11, 5711-5769 (2008).

\bibitem[VdB2]{VdB2} Van den Bergh M.: {\it Non-commutative quasi-{H}amiltonian spaces}. In: Poisson geometry in mathematics and physics, volume 450 of Contemp. Math., 273-299. Amer. Math. Soc., Providence, RI (2008).
  \end{thebibliography}
\end{document}